\documentclass[12pt,reqno]{amsart}

\usepackage{euscript}
\usepackage{amsthm}
\usepackage{amssymb,amscd}
\usepackage{amsmath}
\usepackage{graphicx}
\usepackage{color}
\usepackage{marginnote}
\usepackage{hyperref}
\hypersetup{colorlinks, linkcolor=blue,citecolor=blue}

\newtheorem{theorem}{Theorem}[section]

\newtheorem{proposition}{Proposition}[section]

\newtheorem{corollary}[proposition]{Corollary}
\newtheorem{lemma}[proposition]{Lemma}

\newtheorem{conjecture}[proposition]{Conjecture}

\newtheorem{definition}[proposition]{Definition}
\numberwithin{equation}{section}

\theoremstyle{definition}
\newtheorem{remark}[proposition]{Remark}

\newcommand{\Z}{\mathbb{Z}}
\newcommand{\R}{\mathbb{R}}
\newcommand{\Q}{\mathbb{Q}}

\newcommand{\T}{\mathbb{T}}

\newcommand{\cL}{\mathcal{L}}
\newcommand{\cC}{\mathcal{C}}
\newcommand{\cA}{\mathcal{A}}

\newcommand{\cD}{\mathcal{D}}
\newcommand{\cE}{\mathcal{E}}
\newcommand{\E}{\mathcal{E}}
\newcommand{\Om}{\Omega}

\def\eps{\varepsilon}
\def\A{\mathbb A}
\def\R{\mathbb R}
\def\T{\mathbb T}
\def\Q{\mathbb Q}
\def\Z{\mathbb Z}
\def\N{\mathbb N}

\def\cM{\mathcal M}

\def\cE{\mathcal E}
\def\cH{\mathcal H}

\def\~{\tilde}

\def\Gm{\Gamma}

\def\d{\delta}

\def\l{\lambda}

\def \sn {{\rm sn}\,}
\def \cn {{\rm cn}\,}

\def \am {{\rm am}\,}

\newcommand \e {\varepsilon}

\newcommand \f {\varphi}

\def\bdef{\begin{definition}}
\def\endef{\end{definition}}
\def\bthm{\begin{theorem}}
\def\ethm{\end{theorem}}
\def\blm{\begin{lemma}}
\def\elm{\end{lemma}}
\def\brm{\begin{remark}}
\def\erm{\end{remark}}
\def\bprop{\begin{proposition}}
\def\eprop{\end{proposition}}
\def\bcor{\begin{corollary}}
\def\ecor{\end{corollary}}
\def\be{\begin{eqnarray}}
\def\ee{\end{eqnarray}}
\def\beal{\begin{aligned}}
\def\enal{\end{aligned}}

\author{Guan Huang }
\address{Yau Mathematical Sciences Center, Tsinghua University, Beijing, China}
\email{huangguan@mail.tsinghua.edu.cn}
\author{Vadim Kaloshin}
\address{Department of Mathematics, University of Maryland, College Park, MD, USA}
\email{vadim.kaloshin@gmail.com}
\author{Alfonso Sorrentino}
\address{Dipartimento di Matematica, Universit\`a degli Studi di Roma ``Tor Vergata'', Rome, Italy.}
\email{sorrentino@mat.uniroma2.it}
\begin{document}

\title[]{Nearly circular domains which are integrable close to the boundary are ellipses}

\maketitle
\bibliographystyle{plain}
\begin{abstract}
The Birkhoff conjecture says that the boundary of a strictly
convex integrable billiard table is necessarily an ellipse. In this article, 
{we consider a stronger notion of integrability, namely integrability {close to the boundary}}, 
and prove a local version of this conjecture: a small  perturbation of an ellipse of small 
eccentricity which preserves integrability near the  boundary, is itself an ellipse. 
This extends the result in \cite{ADK}, where {integrability was assumed on a larger set}. 
In particular, it shows that (local) integrability near the boundary implies global integrability.
{One of the crucial ideas in the proof consists in  analyzing Taylor expansion  of the corresponding
action-angle coordinates with respect to the eccentricity parameter, 
deriving and studying  higher order conditions for the preservation of integrable rational caustics}. 
\end{abstract}

\section{Introduction}A {\it mathematical billiard} is a system describing 
the inertial motion of a point mass inside a domain, with elastic reflections at 
the boundary (which is assumed to have infinite mass). This simple model 
has been first proposed by G.D. Birkhoff as a mathematical playground where 
``{\it the formal side, usually so formidable in dynamics, almost completely 
disappears and only the interesting qualitative questions need to be considered }'', 
\cite[pp. 155-156]{Birkhoff}.\medskip

Since then billiards have captured much attention in many different contexts, becoming 
a very popular subject of investigation. Not only is their law of motion very physical and intuitive, 
but billiard-type dynamics is ubiquitous. Mathematically, they offer models in every subclass of 
dynamical systems (integrable, regular, chaotic, etc.); more importantly, techniques initially devised for 
billiards have often been applied and adapted to other systems, becoming standard tools and 
having ripple effects beyond the field. \medskip

Let us first recall some properties of the billiard map. We refer to \cite{Siburg, Tabach} for a more comprehensive introduction to the study of billiards.\medskip

Let $\Omega$ be a strictly convex domain in $\R^2$ with $C^r$ boundary $\partial \Omega$,
with $r\geq 3$. The phase space $M$ of the billiard map consists of unit vectors
$(x,v)$ whose foot points $x$ are on $\partial \Omega$ and which have inward directions.
The billiard ball map $f:M \longrightarrow M$ takes $(x,v)$ to $(x',v')$, where $x'$
represents the point where the trajectory starting at $x$ with velocity $v$ hits the boundary
$\partial \Omega$ next, and $v'$ is the {\it reflected velocity}, according to
the standard reflection law: the angle of incidence is equal to the angle of reflection (Figure \ref{billiard}).

\begin{remark}
Observe that if $\Omega$ is not convex, then the billiard map is not continuous;
in this article we will be interested only in strictly convex domains (see Remark \ref{Matherglancing}).
Moreover, as pointed out by Halpern \cite{Halpern}, if the boundary is not at
least $C^3$, then the (continuous) Billiard flow might not be complete (or, equivalently, there might be non-trivial orbits with finite total length).
\end{remark}

Let us introduce coordinates on $M$.
We suppose that $\partial \Omega$ is parametri\-zed  by  arc-length $s$ and
let $\gamma:  \frac{\R}{|\partial\Omega| \Z} \longrightarrow \R^2$ denote such a parametrization,
where $|\partial \Omega|$ denotes the length of $\partial \Omega$. Let $\theta$
be the angle between $v$ and the positive tangent to $\partial \Omega$ at $x$.
Hence, $ M$ can be identified with the annulus $\A = \frac{\R}{|\partial\Omega| \Z} \times (0,\pi)$
and the billiard map $f$ can be described as
\begin{eqnarray*}
f:  \A &\longrightarrow& \A\\
(s,\theta) &\longmapsto & (s',\theta').
\end{eqnarray*}

\begin{figure} [h!]
\begin{center}
\includegraphics[scale=0.25]{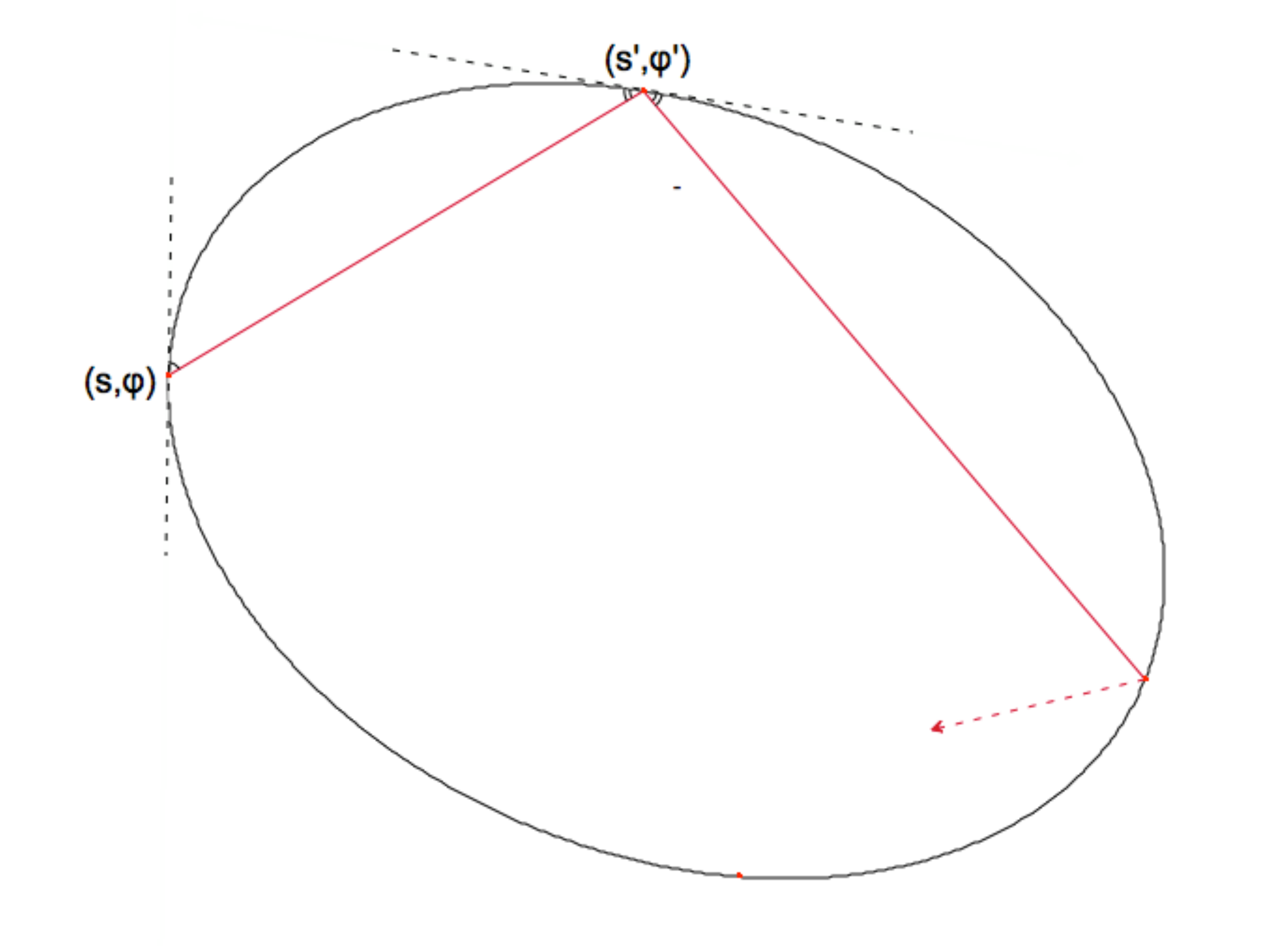}
\caption{}
\label{billiard}
\end{center}
\end{figure}

In particular $f$ can be extended to $\bar{\A}=\frac{\R}{|\partial\Omega| \Z} \times [0,\pi]$ by fixing
$f(s,0)=(s,0)$ and $f(s,\pi)=(s,\pi)$
for all $s$. \medskip

It is easy to check that the billiard map $f$ preserves the area form $\sin\theta \, ds\wedge d\theta$. 
If we denote by
$$
\ell(s,s') := \|\gamma(s) - \gamma(s')\|
$$
the Euclidean distance between two points on $\partial \Omega$, then one can check that
\begin{equation}\label{genfunctbill}
\left\{ \begin{array}{l}
\dfrac{\partial \ell}{\partial s}(s,s') = -\cos \theta \\
\dfrac{\partial \ell}{\partial s'}(s,s') =  \cos \theta'\,.
\end{array}\right.
\end{equation}

\begin{remark}
If we consider the lift  to the universal cover and introduce new coordinates
$(x,y)=(s, -\cos \theta) \in \R \times (-1,1)$, then the billiard map is a twist map
with $\ell$ as generating function and it preserves the area form $dx \wedge dy$. See \cite{Siburg, Tabach}.
\medskip
\end{remark}

Despite the apparently simple (local) dynamics, the qualitative dynamical properties of billiard maps are extremely {non-local}. 
This global influence on the dynamics translates into several intriguing {\it rigidity  phenomena}, which
are at the basis of several unanswered questions and conjectures (see for example \cite{ADK, DKW, HKS, KS, Siburg, SorDCDS,Tabach}). 
Amongst many, in this article we will address the question of classifying  {\it integrable billiards}, also known as {\it Birkhoff conjecture}.


\subsection{Integrable billiards and Birkhoff conjecture}

The easiest example of billiard is given by a billiard in a disc $\cD=\cD_R$  of radius $R$. It is easy to check in this case that 
the angle of reflection remains constant at each reflection (see also  \cite[Chapter 2]{Tabach}). If we denote by $s$  the arc-length 
parameter ({\it i.e.}, $s\in {\R}/  {\tiny 2\pi R} \Z$) and by $\theta \in (0,\pi)$ the angle of reflection, then the billiard map  
has a very simple form:
$$
f(s,\theta) = (s + 2R\, \theta,\; \theta).
$$
In particular, $\theta$ stays constant along the orbit and it represents an {\it integral of motion} for the map.
Moreover, this billiard enjoys the peculiar property of
having  the phase space -- which is topologically a cylinder --  completely foliated by homotopically non-trivial invariant curves 
${\Gamma}_{\theta_0}=\{\theta\equiv \theta_0\}$. These curves correspond to concentric circles of radii 
$\rho_0= R\cos \theta_0$ and are examples of what are called {\it caustics}, {which are defined as follows:}
\medskip

{\it A smooth convex curve $\Gm \subset \Omega$ is called {\em a caustic}, if whenever a trajectory is tangent to it, 
then it  remains tangent after each reflection (see figure \ref{circle-billiard}). }
\medskip

Notice that in the circular case, each caustic $\Gamma$ corresponds to an invariant curve of the associated billiard map $f$ 
and, therefore, has a well-defined rotation number.

\begin{figure} [h!]
\begin{center}
\includegraphics[scale=0.27]{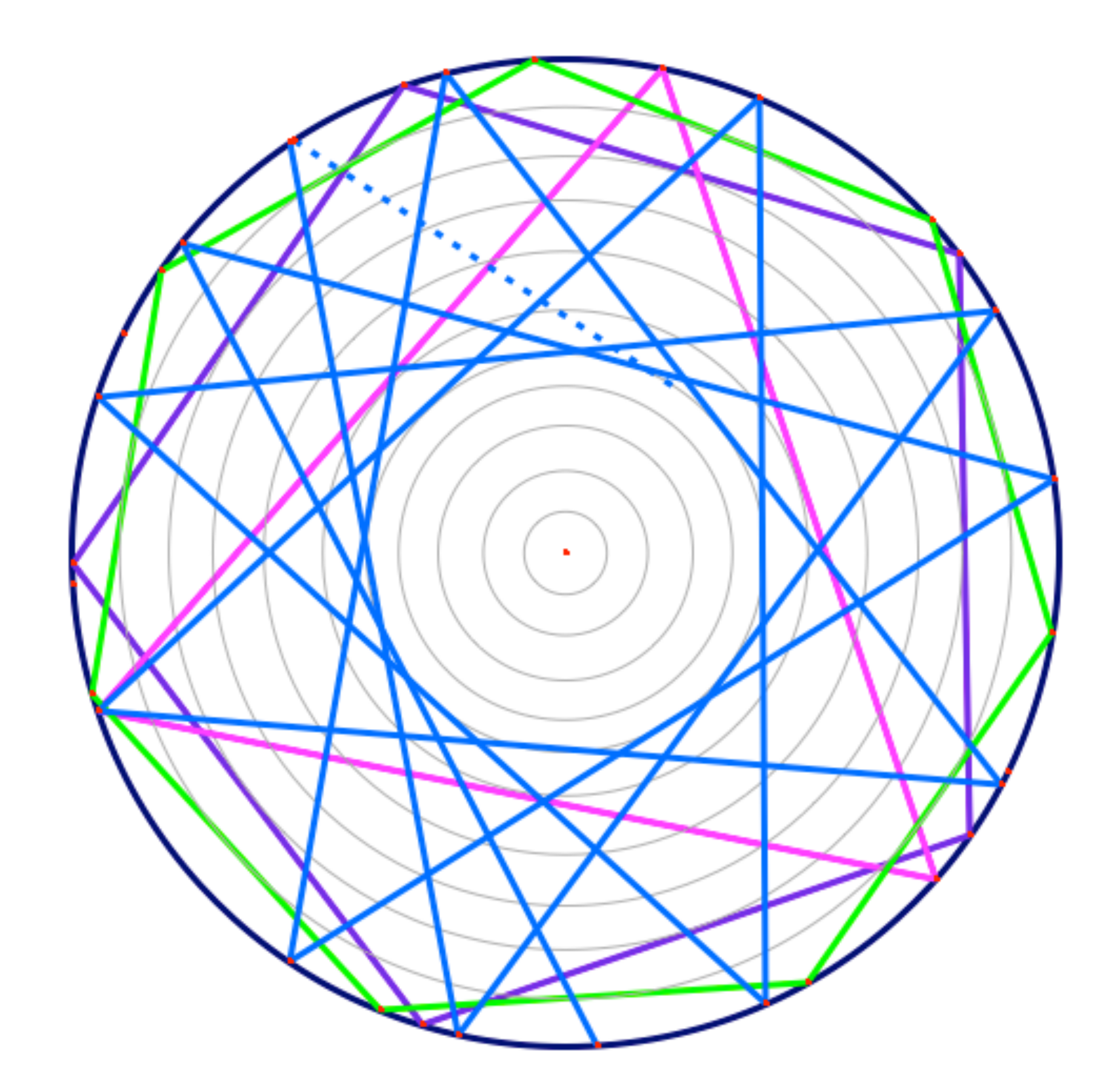}
\caption{Billiard in a disc}
\label{circle-billiard}
\end{center}
\end{figure}

A billiard in a disc is an example of an  {\it integrable billiard}. There are different  ways to define global/local integrability for billiards 
(the equivalence of these notions is an interesting problem itself):
\begin{itemize}
\item[-] either through the existence of an integral of motion, globally or locally near the boundary (in the circular case  
an integral of motion is given by  $I(s,\theta)=\theta$), 
\item[-] or through the existence of a (smooth) foliation {of the whole phase space or  of an open subset (for example, of 
a neighborhood of the boundary $\{\theta=0\}$)}, consisting of invariant  curves of the billiard map; for example, in 
the circular case these are given by ${\Gamma}_{\theta}$. This property translates (under suitable assumptions) 
into the existence of a (smooth) family of caustics, globally or locally near the boundary 
(in the circular case, the concentric circles of radii $R\cos \theta$).
\medskip
\end{itemize}

In \cite{Bialy},  Misha Bialy proved the following  result concerning global integrability (see also \cite{Woi}):
\medskip

\noindent {\bf Theorem (Bialy).}
\noindent {\it If the phase space of the billiard ball map is globally foliated by continuous
invariant curves which are not null-homotopic, then  it corresponds to  a  billiard in a disc.}
\medskip

However, while circular billiards are the only examples of 
global integrable billiards, {non-global} integrability itself is 
still an intriguing open question. One could  consider a  billiard 
in an ellipse: this is in fact  integrable, yet the dynamical picture 
is very distinct from the circular case: as it is showed in 
figure \ref{ellipse-billiard}, each trajectory which does not pass 
through a focal point, is always tangent to precisely one confocal 
conic section, either a confocal ellipse or the two branches of 
a confocal hyperbola (see for example \cite[Chapter 4]
{Tabach}). Thus, the confocal ellipses inside an elliptic 
billiard are convex caustics, but they do not foliate the whole 
domain: the segment between the two foci is left out (describing 
the dynamics explicitly is much more complicated than in the circular case: see for example \cite{Taba}). 
\medskip  

\begin{figure} [h!]
\begin{center}
\includegraphics[scale=0.18]{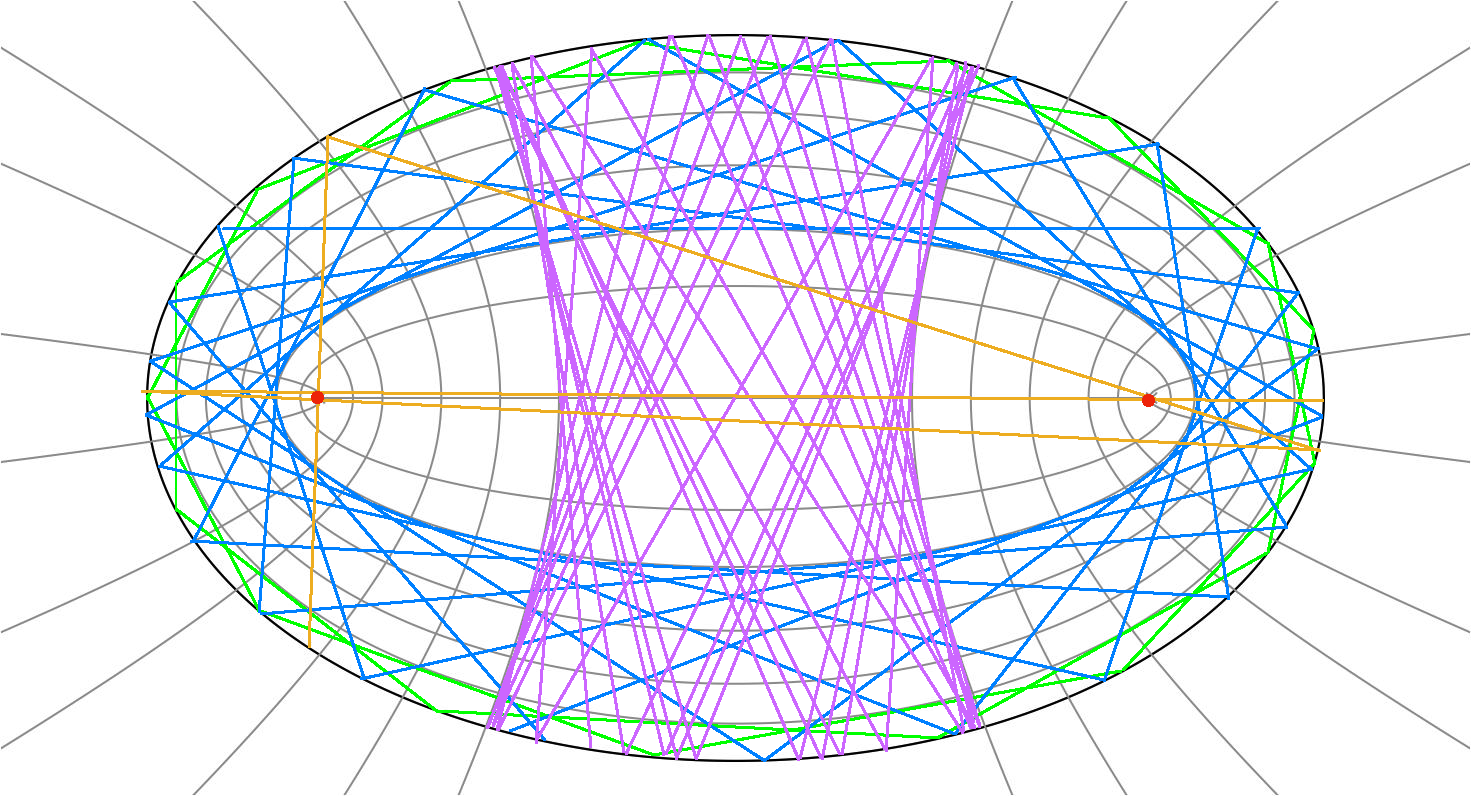}
\caption{Billiard in an ellipse}
\label{ellipse-billiard}
\end{center}
\end{figure}

\noindent{\bf Question (Birkhoff).} {\it Are there other examples 
of {integrable} billiards?}
\medskip

\begin{remark}
Although some vague indications of this question can be found 
in \cite{Birkhoff}, to the best of our  knowledge, its first 
appearance as a conjecture was in a paper by Poritsky 
\cite[Section 9]{Po},\footnote{In \cite[Footnote 1]{Po} Poritsky acknowledged that the results in the paper
were obtained in 1927--29 while he was National Research Fellow in Mathematics at Harvard University, 
presumably under the supervision of Birkhoff.
Although the author does not attribute  this conjecture explicitly to Birkhoff, yet he cites many of his papers on the topic, hence  it is reasonable to surmise Birkhoff's influence behind it.}
which was published several years after Birkhoff's death.
Thereafter, references to this conjecture (either as {\it Birkhoff conjecture} or {\it Birkhoff-Poritsky conjecture}) repeatedly appeared in the literature: 
see, for example, Gutkin \cite[Section 1]{Gutkin}, Moser \cite[Appendix A]{Mo}, Tabachnikov \cite[Section 2.4]{Tabach1}, etc.
\end{remark}

\begin{remark}\label{Matherglancing}
In \cite{Mather82} Mather  proved the non-existence
of caustics (hence, the non-integrability) if the curvature of the boundary vanishes at least at a point.
This observation justifies the restriction of our attention  to strictly convex domains.
\medskip
\end{remark}

\begin{remark}
Interestingly, Treschev in \cite{Tre} gives indication that
there might exist analytic  billiards, different from ellipses, for which the dynamics 
in a neighborhood of the elliptic period-$2$ orbit is conjugate to a rigid rotation. 
These billiards can be seen as an instance of {\it local integrability}; however, 
this regime is somehow complementary to the one conjectured by Birkhoff:
{one has local integrability in a neighborhood of an elliptic periodic orbit of 
period $2$, while Birkhoff  conjecture is related to  integrability in a neighborhood of 
the boundary. This gives an indication that these two notions of integrability might differ.  }
\end{remark}

\medskip

\begin{remark}
The Birkhoff conjecture can be also thought as an analog, in the case of billiards, of the following task: classifying {\it integrable} (Riemannian) geodesic flows on $\T^2$. The complexity of this question, of course, depends on the notion of integrability that one considers.
If one assumes that the whole space space is foliated by invariant Lagrangian graphs ({\it i.e.}, the system is {\it $C^0$-integrable}), then it follows from Hopf conjecture \cite{BI} that the associated metric must be flat.
However, the question becomes more challenging -- and it is still open -- if one considers integrability only on an open and dense set ({\it global integrability}), or assumes the existence of an open set foliated by invariant Lagrangian graphs ({\it local integrability}).
\medskip
 Example of globally integrable (non-flat) geodesic flows on $\T^2$ are those associated to 
{\it Liouville-type metrics}, namely metrics of the form
$$
ds^2  = (f_1(x_1) + f_2(x_2)) \,(dx_1^2 +dx_2^2).
$$
A folklore conjecture states that  these metrics are the only globally (resp. locally) integrable metrics on $\T^2$.
\medskip
A partial answer to this conjecture (global case) is provided in \cite{BFM}, where  the authors prove it  under the assumption that the system admits an integral of motion which is quadratic in the momenta.
\medskip
The question to which we provide an affirmative answer in this article (local Birkhoff conjecture), can be considered as an analog, in the billiard setting, of the above  conjecture (local case). It is interesting to point out, however,  that - contrarily to what happens with billiards --  there is  evidence  that this local conjecture might be false for geodesic flows (see \cite{CK}).\\
\end{remark}


Despite its long history and the amount of attention that Birkhoff conjecture has captured, 
it remains still open. As far as our understanding of integrable billiards is concerned,
the  most important related results are the above--mentioned theorem by Bialy \cite{Bialy} 
(see also \cite{Woi}), a result by Delshams and Ram\'irez-Ros \cite{DRR} in which they study 
entire perturbations of elliptic billiards and prove that any nontrivial symmetric perturbation of 
the elliptic billiard is not integrable, {a result by Innami\footnote{We are grateful to M. Bialy for pointing out this reference.} 
\cite{Innami}, in which he shows that the existence of caustics for all rotation numbers in $(0,1/2)$ implies that 
the billiard must be an ellipse}, and a more recent result by Avila, De Simoi and Kaloshin \cite{ADK} 
in which they show a perturbative version of this conjecture for  ellipses of small eccentricity, 
{assuming the existence of caustics for all rotation numbers in $(0,1/3]$.}
The latter result was generalised to ellipses of any eccentricity by Kaloshin and Sorrentino 
\cite{KS}. 
\medskip

Let us introduce an important notion for this paper. 
\medskip

\begin{definition}
{\rm (i)} We say that $\Gamma$  is an {\rm{integrable} rational caustic} for the billiard map in $\Omega$,
if the corresponding (non-contractible) invariant curve  consists 
of periodic points; in particular, the corresponding rotation number is rational. \newline
\medskip
{\rm (ii)} Let $q_0\geq2$.  If the billiard map associated to $\Omega$ admits integrable 
rational caustics of rotation number $p/q$ for all $0<p/q <1/q_0$, we say that 
$\Omega$ is {\rm $q_0$--rationally integrable. }
\medskip
\end{definition}

\begin{remark}
A simple sufficient condition for rational integrability is the following (see \cite[Lemma 1]{ADK}).
Let ${\mathcal C}_{\Omega}$ denote the union of all smooth convex caustics of the billiard in  
$\Omega$; if the interior of ${\mathcal C}_{\Omega}$ contains caustics of rotation number 
$p/q$ for all $0<p/q <1/q_0$, then $\Om$ is $q_0$-rationally integrable. 
\end{remark}

Let us denote with $\cE_{e,c} \subset \R^2$ an ellipse of eccentricity $e$ and 
semi-focal distance $c$. We state the following local version of Birkhoff conjecture.
\smallskip

\begin{conjecture}\label{conjecture}
For any integer $q_0\ge 3$, there exist $e_0=e_0(q_0)\in (0,1),$\ $m_0=m_0(q_0),\ n_0=n_0(q_0) \in \N$ 
such that the following holds. For each $0< e \leq e_0$ and $c\geq 0$, there exists $\eps=\e(e,c,q_0) > 0$ such that any $q_0$-rationally 
integrable $C^{m_0}$--smooth domain $\Omega$, whose boundary $\partial \Om$ is 
$C^{n_0}$--\,$\varepsilon$-close to an ellipse $\cE_{e,c}$, is itself an ellipse.
\end{conjecture}

\medskip

In this paper we prove this conjecture in some cases and provide a  proof for the remaining ones 
based on certain non-degeneracy conditions. These non-degeneracy conditions are explicit and computable: in Section \ref{case-general} we provide a  description of how to implement an algorithm to verify 
them by means of symbolic computations.   

More precisely, our main results are the following.\medskip

\begin{theorem} \label{main-thm}
Conjecture \ref{conjecture} holds true for $q_0=2, 3, 4,5$, with $m_0=40q_0$ and $n_0=3q_0$. \medskip
\end{theorem}

\begin{theorem}\label{main-thm2} For any integer $q_0\geq 6$, 
Conjecture \ref{conjecture} holds true  with $m_0=40q_0$ and $n_0=3q_0$,  
provided that the { $q_0-2$ matrices \eqref{matrixAodd}-\eqref{matrixEven}  are non-degenerate. }
\end{theorem}

\smallskip

\begin{remark} 
({\it i}) Case $q_0=2$ was proven in \cite{ADK} {(see also \cite{Innami})}. \medskip

({\it ii}) Notice that $\eps(e,c,q_0)\to 0$ as $e\to 0^+$. 
Non-zero $e$, in fact, produces asymmetry and it is fundamental 
for our argument to work. The less $e$ is, the smaller must be the 
perturbations  that allow one to stay in asymmetric regime.\\
We point out  that $\varepsilon$ does not need to go to $0$ with $e$ for $q_0=2$ (see \cite{ADK}).
\medskip 

({\it iii}) The smoothness exponent  is probably not optimal. 
In the proof of one of  the key lemmata (Lemma \ref{rational-condition}),  we have directly used 
certain $C^1$-estimates from~\cite{ADK}. One may improve the smoothness exponent
by deriving $C^n$~estimates instead. \medskip 

{({\it iv}) Notice that we actually do not need the existence of all caustics of rotation number less than $1/q_0$;  
in fact, we  only use integrable rational 
caustics of rotation 
numbers of the form $j/q<1/q_0$ for $j=1,2,3$. \medskip 

{({\it v}) Analysis of caustics of rotation numbers $\frac{2}{2q+1}$  is fairly 
delicate\,\footnote{The same remark applies to rotation numbers $\frac{3}{2q}$, for $q$  not divisible by $3$.}. 
Either for a domain close to the circle or for an arbitrary sufficiently smooth  domain and large $q$, the condition of preservation 
of caustics of rotation numbers $\frac{2}{2q+1}$ and $\frac{1}{2q+1}$  {are the same} to the leading order! 
Thus, to obtain a new condition from caustics of rotation numbers $\frac{2}{2q+1}$ we need a precise information 
about {\it higher order dependence on the rotation number}. For small eccentricity $e$ this can be extracted from the Taylor
expansion of the action-angle variables with respect to the eccentricity parameter (see Appendix C for details). Without this precise 
information our method would not work! This analysis can be considered as the main novel feature of 
the present paper compared to \cite{ADK} and \cite{KS}.}
\medskip 

({\it vi})  The coefficients of matrices \eqref{matrixAodd}-\eqref{matrixEven} are 
completely determined  by the $e$-expansions of the action-angle parametrization for  
the elliptic billiard map, which, in turn, is explicitly given by  elliptic integrals 
(see  \eqref{action-para} and Appendix \ref{e-expansion-sec}). {In particular, 
the entries of these matrices are either $0$, $1$ or of the form ${\xi}\,{ \cos^{-2j} ( w \pi )  e^{2j}}$, 
where $\xi \in \Q$, $j\in \N$,  $w\in \{\frac{1}{2k+1}, \frac{2}{2k+1}, \frac{1}{2k}, \frac{3}{2k}: \; k>j \}$.
See also Remarks \ref{remarkmatrixAodd} and \ref{remarkmatrixAeven}.
}}
\end{remark}

\noindent{\bf Acknowledgements} GH and  AS would like to thank the FIM-ETH for the hospitality. 
GH acknowledges partial support of the National Natural Science Foundation  of China (Grant no. 11790273). VK acknowledges partial support of the NSF grant DMS-1402164 and 
the hospitality of the ETH Institute for Theoretical Studies and the support of  Dr. Max R\"ossler, 
the Walter Haefner Foundation and the ETH Z\"urich Foundation.
AS acknowledges the partial support of the Italian MIUR research grant: PRIN-2012-74FYK7 
``Variational and perturbative aspects of nonlinear differential problems''. 
VK is grateful to Jacopo De Simoi and Ke Zhang for useful discussions. 
AS would like to thank Pau Martin and Rafael Ram\'irez-Ros 
for useful discussions during his visit to UPC. 
Finally, the authors wish to express their enormous gratitude to the anonymous referee for 
her/his remarkable revision work.
\medskip

\section{The strategy of the proof} \label{secstrategyproof}
Let us consider the ellipse 
$$
{\cE_{e,c}} =\left \{(x,y)\in \R^2:\; \frac{x^2}{a^2} + \frac{y^2}{b^2} = 1\right\},
$$
centered at the origin and with semi-axes of lengths, respectively, $0<b\leq a$; 
in particular $e$ denotes its eccentricity, given by $e=\sqrt{1-\frac{b^2}{a^2}} \in [0,1)$ 
 and $c=\sqrt{a^2-b^2}$ the semi-focal distance.
{Observe that when $e=0$, then $c=0$ and $\cE_{0,0}$ degenerates to a $1$-parameter 
family of circles centered at the origin.} 

The family of confocal elliptic caustics  in $\cE_{e,c}$ is given 
by (see also Figure \ref{ellipse-billiard}):
\begin{equation}\label{caustic}
C_{\l} = \left\{
(x,y)\in \R^2: \; \frac{x^2}{a^2-\l^2} +  \frac{y^2}{b^2-\l^2}=1
\right\} \qquad 0<\l<b.
\end{equation}
Observe that the boundary corresponds to $\l=0$, 
while the limit case $\l=b$ corresponds to the the two foci 
${\mathcal F}_{\pm}=(\pm \sqrt{a^2-b^2},0)$. Clearly, for 
$e=0$ we  recover the family of concentric circles 
described in Figure~\ref{circle-billiard}.\medskip

{Denote $\mathbb{T}:=\mathbb{R}/2\pi\mathbb{Z}$}. A more convenient coordinate frame for addressing 
our question is provided by the so-called 
{\it elliptic-polar coordinates} (or, simply, {\it elliptic coordinates})
$(\mu,\f) \in \R_{\geq 0} \times \mathbb{T}$, given by:
$$
\left\{
\begin{array}{l}
x= c \cosh \mu \, \cos\f \\
y= c \sinh \mu \, \sin\f,
\end{array}
\right.
$$
where $c=\sqrt{a^2-b^2}>0$ represents the semi-focal distance 
(in the case $e=0$, this parametrization degenerates to 
the usual polar coordinates). Observe that for each $\mu_*> 0$, 
the equation $\mu\equiv \mu_*$ represents  a confocal ellipse.\medskip

Therefore, in these elliptic polar coordinates  $\cE_{e,c}$ becomes:
$$
\cE_{e,c} =  \left\{({\mu_0},\f),\; \f \in  \mathbb{T} \right\},
$$
where  $\mu_0 =\mu_0(e):= \cosh^{-1}\left({1}/{e}  \right)$.\medskip
Then, any smooth perturbation $\Omega$ of the ellipse 
$\cE_{e,c}$ can be written in this elliptic-coordinate frame as
$$
\partial\Omega=\{(\mu_0+\mu(\varphi),\varphi): \varphi\in\mathbb{T}\},
$$
where $\mu(\varphi)$ is a  small smooth $2\pi$-periodic  function; hereafter, we will adopt this shorthand notation and write
$$
\partial{\Omega}={\cE}_{e,c}+\mu(\varphi).
$$

\bigskip

{Before describing the strategy of our proof, let us first recall 
the scheme in \cite{ADK}, and then describe the needed 
adjustments.
}
%
%
%
%

\subsection{A preliminary scheme of proving Theorem \ref{main-thm} for $q_0=2$. }

In the case $q_0=2$, Theorem \ref{main-thm} was proven in \cite{ADK} and we now  describe  
the proof therein. {In order to get a clearer idea, let us  start from the simplified case of integrable infinitesimal deformations 
of a circle}.  \medskip



\medskip 
Let $\Omega_0$ be a circle centered at the origin.
Let $\Om_\eps$ be a one-parameter family of deformations,
given in  polar coordinates by
$$
\partial\Om_\eps=\{(\mu_0+\eps
\mu(\varphi)+O(\eps^2),\varphi), 
\; \varphi\in \T\}.
$$
Fix a parametrization of the boundary $\varphi:\T \to \T$.
Consider the Fourier expansion of $\mu \circ \varphi$ :
\begin{align*}
  \mu \circ \varphi(\theta)=\mu'_0 + \sum_{k > 0} \mu'_{k,\varphi}\sin (k\theta)+ \mu''_{k,\varphi} \cos (k\theta).
\end{align*}

\begin{theorem}[Ram\'irez-Ros \cite{RR}] \label{thm:RR} 
If, for any sufficiently small $\eps$, $\Om_\eps$ has an integrable rational caustic $\Gm^\eps_{1/q}$ 
of rotation number $1/q$, then 
for a certain parametrization of the boundary\footnote{This parametrization of 
the boundary can be found in Lemma \ref{rational-condition}.} $\varphi_{1/q}(\theta)$, 
we have $\mu'_{q,\varphi_{1/q}} =\mu_{q,\varphi_{1/q}}'' = 0$.\\
\end{theorem}

Notice that in the case of circular billiards,
$\varphi_{1/q}(\theta)\equiv \theta$ for all $q>2$; 
however, this stops to be true away from the circle (see \eqref{action-para} for the more general 
case of elliptic billiards). \\

This more general framework
allows us to explain our strategy better.\medskip 

Let us now assume that the domains $\Om_\eps$ are $2$-rationally integrable for 
all sufficiently small $\eps$ and ignore for a moment dependence on the parametrization; 
then, the above theorem implies that $\mu'_k = \mu''_k = 0$ for $k > 2$, {\it i.e.},
\begin{align*}
  \mu(\varphi)&=\mu'_0+\mu'_1\cos \varphi+\mu_1''\sin \varphi
  + \mu'_2\cos 2\varphi+\mu_2''\sin 2\varphi\\
          &= \mu'_0+\mu_1^*\cos (\varphi-\varphi_1)+
            \mu_2^*\cos 2(\varphi-\varphi_2), 
\end{align*}
where $\varphi_1$ and $\varphi_2$ are appropriately 
chosen phases.
\begin{remark}\label{r_motionDescription}
  Observe that
  \begin{itemize}
  \item $\mu'_0$ corresponds to an homothety;
  \item {$\mu_1^* \cos (\varphi-\varphi_1)$ corresponds to a translation in the direction 
    forming an angle $\varphi_1$ with the polar semi-axis
    $\{\varphi = 0\}$;}
  \item {$\mu_2^*\cos 2(\varphi-\varphi_2)$ corresponds to a deformation of the disc into an ellipse
    of small eccentricity, whose major axis forms an angle $\varphi_2$ with the  polar semi-axis.}
  \end{itemize}
\end{remark}  This implies that, {infinitesimally} (as
  $\eps\to 0$), $2$-rationally integrable deformations of a
  circle are tangent to the $5$-parameter family of
  ellipses.

Observe that in principle in the above theorem one may
need to take $\eps\to0$ as $q\to\infty$.  {However, note that   
the cases we have to deal with correspond to $\eps > 0$  small, 
but not infinitesimal; hence, one cannot use directly the above
scheme to prove the result and a more elaborate strategy needs 
to be adopted. Let us  describe it more precisely.}

\subsection{The actual scheme of the proof of Theorem \ref{main-thm} for $q_0=2$.} 

\medskip

Let $ \cE_{e,c}$ be an ellipse of small eccentricity $e$ and 
 semi-focal distance $c$. Let $(\mu,\varphi)$ be the associated elliptic-coordinate frame. Any domain $\Om$ whose boundary is close to $\cE_{e,c}$, can be written in
 the elliptic-coordinate frame associated to $\cE_{e,c}$ as
\begin{align*}
  \partial\Om=\{( \mu_0+\mu(\varphi),\varphi):\varphi\in\mathbb{T}\},
\end{align*}
where $\mu(\varphi)$ is a (small) smooth $2\pi$-periodic function. 
The  strategy used in \cite{ADK}  proceeds as follows (keep 
in mind that the ellipse $\cE_{e.c}$ admits all integrable rational 
caustics of rotation number $1/q$ for $q > 2$).\medskip%

\noindent\textbf{Step 1:} Derive a quantitative necessary
condition for the preservation of an integrable rational caustic of a given rotation number
(see \cite[Theorem 3]{ADK} or  Lemma~\ref{rational-condition} below).\medskip%

\noindent\textbf{Step 2:} Define the {\it Deformed
Fourier modes} 
$$\{c_0,c_q,c_{-q}\}_{q>0}$$ 
associated to the ellipse $\cE_{e,c}$. They
satisfy the following
properties:
\begin{itemize}
\item {\it (Relation with Fourier Modes)} There exist (see \cite[Lemma~20]{ADK})  
$C^*(e,c)>0$ with $C^*(e{,c})\to0$ as $e\to0^+$, and a properly chosen 
parametrization of the boundary such that 
  $$
  \|c_0-1\|_{C^0}\leq C^*(e{,c})
  $$ 
  and for any $q\ge 1$
\begin{equation}\label{adk-basis1}
\begin{cases} 
\|c_q-\cos ( q\,\cdot)\|_{C^0}\, \leq \,q^{-1}\,C^*(e{,c})\\
\|c_{-q}-\sin ( q\,\cdot)\|_{C^0} \leq q^{-1}\,C^*(e{,c}).
\end{cases}\end{equation}
\item (\emph{Transformations preserving integrability}) 
  The first five functions
\begin{align*}
  c_0,\,c_1,\,c_{-1},\,c_2,\,c_{-2}
\end{align*}
{correspond to infinitesimal generators of deformations preserving the class of ellipses: 
namely,  homotheties, translations and hyperbolic rotations about an arbitrary axis.}
\medskip 

\item ({\it Annihilation of inner products})
 Consider the one-\hskip0pt{}parameter 
family of domains $\Omega_{\varepsilon}$, $\varepsilon\in(-\varepsilon_0,\varepsilon_0)$, 
written in the elliptic-coordinate frame associated to the ellipse $\cE_{e,c}$,
\begin{align*}
  \partial\Om_\eps:=\cE_{e,c}+\eps \mu.
\end{align*}
For any $q > 2$,  if $\Om_\eps$ admits an integrable rational caustic
 of rotation number $1/q$ for all
sufficiently small $\eps$, then
\begin{align}
  \label{eq:vanishing}
  \langle \mu,c_q\rangle =0,\quad   \langle \mu,c_{-q}\rangle = 0,
\end{align}
where $\langle\cdot,\cdot\rangle$ is a suitably weighted $L^2$ inner
product. 

Notice that the functions $c_{\pm q}$ can be explicitly
defined using elliptic integrals via action-angle
coordinates.

\medskip 

\item (\emph{Linear independence and Basis property}) For sufficiently small
  eccentricities,  the
  functions $\{c_0,c_q, c_{-q}:\,q>0\}$ form a (non-orthogonal)  basis of $L^2(\mathbb{T})$.
\medskip
\end{itemize}

\noindent\textbf{Step 3} ({\it Approximation}):  Using the annihilation of 
the inner products, for the domain $\partial\Omega=\cE_{e,c}+\mu$ with 
small eccentricity $e$, one can find  another ellipse $\cE'$ such that 
$$\partial\Omega=\cE'+\mu'\quad \ \text{ and } \ 
\quad \|\mu'\|_{C^1}\leq \frac{1}{2}\|\mu\|_{C^1}.
$$ 
Applying this result to the best approximation of $\Omega$ by an ellipse and
then arguing by contradiction,  allow us  to conclude that $\Omega$ itself must be an ellipse.

\medskip 

\subsection{The adjusted scheme for the case $q_0>2$}
{Now we describe how to modify the above strategy to deal 
with the case $q_0>2$}. 

Fix an ellipse $\cE_{e,c}$ of eccentricity 
$e>0$ and  semi-focal distance $c$. In Section 
\ref{subsec:ellipticdynamics} we will introduce the action-angle 
coordinates associated to the billiard problem in $\cE_{e,c}$  
(it turns out that for $e=0$ these action-angle coordinates 
degenerate to the polar coordinates $(\rho,\theta)$).  
\medskip

\noindent\textbf{Step 1$'$:} For small $e>0$, we study the Taylor expansion, with respect to $e$, of 
the action-angle coordinates. Using 
this expansion, we derive the necessary condition for the preservation of 
 integrable rational caustics, in terms of the Fourier coefficients of the function $\mu$, 
up to the precision of order $e^{2N}$, for some positive integer $N=N(q_0)$. 
See Section \ref{NCFC} and  
equality \eqref{fourier1} for more details.
\medskip

\noindent\textbf{Step 2$'$:} {We define the {\it deformed Fourier 
modes} {$\{\cC_0, \;\cC_{q},\:\cC_{-q}\}_{q>0}$}, similarly to what 
described before.}  Fix some  $r\in\mathbb{N}$; these functions satisfy the following 
properties.
\medskip

\begin{enumerate}
\item {\it (Relation with Fourier mode)} {We have $\cC_0=1$,
$$\cC_q(\cdot)=\mathcal{V}_q(\cdot), \quad 
\cC_{-q}(\cdot)=\mathcal{V}_{-q}(\cdot), \quad 0< q\leq q_0,$$
and there exists $C^*(e)>$ with $C^*(e)\to0$ as $e\to0^+$, such that 
\[\begin{cases}\|\,\cC_q(\cdot)-\mathcal{V}_q(\cdot)\,\|_r\ \leq \ C^*(e)/q,\\
\|\cC_{-q}(\cdot)-\mathcal{V}_{-q}(\cdot)\|_r\leq C^*(e)/q,\end{cases} \quad 
q>q_0,\]
where $\|\cdot\|_{r}$ is the norm in the Sobolev space $H^r(\mathbb{T})$, 
and $\mathcal{V}_q$ are the zero average functions on $\mathbb{T}$, such that 
\[\begin{cases}\mathcal{V}^{(r)}_q(\cdot)=\ \ \cos ( q\;\cdot),& q>0\\
\mathcal{V}^{(r)}_q(\cdot)=-\sin ( q \; \cdot),&q<0,\end{cases}
\]
 where $\mathcal{V}_q^{(r)}$ denotes the $r$-th derivative of $\mathcal{V}_q$. The constant $C^*(e)$ here can be chosen as in \eqref{adk-basis1}.}
\medskip

\item {\it (Linear independence and Basis property)} For small eccentricities, the set of functions 
$\{\cC_0, \; \cC_q,\; \cC_{-q},\,q\in\mathbb{N}_+\}$ form a (non-orthogonal) 
basis of the Hilbert space $H^r$ (see Lemma \ref{basis-proof}).
\medskip

\item {\it (Annihilation of inner products)} From the existence of { integrable 
rational} caustics with rotation numbers $1/q$, 
$q>q_0$,  we deduce the following relations:
$$
\langle \mu, \cC_{\pm q }\rangle_r=O(q^7\|\mu\|_{C^1}^{1+\beta}),\quad \beta>0,
$$
where $\langle\cdot,\cdot\rangle_r$ is the inner product of the Hilbert space 
$H^r(\mathbb{T})$ (see Lemma \ref{key-lemma1}).
\end{enumerate}

\medskip 

Observe that since $q_0\geq3$, with respect to 
the previous scheme we have lost finitely many annihilation 
conditions:
\be \label{eq:vanishing} 
 \langle \mu,\cC_q\rangle_r =  \langle \mu,\cC_{-q}\rangle_r = 0,\quad  3\leq q\leq q_0.
\ee
{Hence, we need to find a way to recover them}. 
Our goal becomes then to show: 
\begin{align}
  \label{eq:2nd-vanishing}
 \langle \mu,\cC_q\rangle_r = O(e^2)\quad , \quad \langle \mu,\cC_{-q}\rangle_r = O(e^2),
  \quad3\leq q\leq q_0.
\end{align}

\medskip

{In particular, we manage to prove them in the following way.}\medskip

\begin{itemize}
\item {\bf Case $q_0=3$:}
We lose  a pair of conditions \eqref{eq:vanishing}, corresponding to  $q=3$. In Section \ref{case3} 
we study the necessary conditions for the existence of {integrable rational} caustics 
of rotation 
numbers $1/5,1/7,2/7$. We use the expansions, with respect to $e$,  of the resulting equalities, up to 
the precision $O(e^6)$, to derive a system of linear equations (see \eqref{linear-system-3}) 
for the $3^{rd}$, $5^{th}$, $7^{th}$ Fourier coefficients. Solving this linear system will provide
 us with \eqref{eq:2nd-vanishing} for $q=3$.\medskip

\item {\bf Case $q_0=4$:}
In this case we lose two pairs of conditions \eqref{eq:vanishing}, corresponding to  {$q=3,4$}. 
In Section \ref{case4} we derive \eqref{eq:2nd-vanishing} for $q=3,4$; this will be achieved in two steps:
{
\begin{itemize}
\item[-] To recover \eqref{eq:2nd-vanishing} for $q=3$, we study the  necessary conditions for the existence of 
{ integrable rational} caustics of  rotation numbers 
 $1/5,\,1/7,\,1/9, \,2/9$, written in terms of the Fourier coefficients of $\mu$, 
and considering their expansions, with respect to $e$, up to order $O(e^{8})$. We then  derive a linear system for 
the $3^{rd}$, $5^{th}$, $7^{th}$, $9^{th}$ Fourier coefficients, whose solution will provide us with  \eqref{eq:2nd-vanishing} for $q=3$.
\item[-]
To recover \eqref{eq:2nd-vanishing} for $q=4$,  we study the necessary  
conditions for the existence of { integrable rational } caustics of 
rotation numbers $1/6,\ 1/8,\ 1/10,\ 1/12,\ 1/14,\ 3/14$, which give rise to 
a system of linear equation for the $4^{th}$, $6^{th}$, $8^{th}$, $10^{th}$, $12^{th}$, $14^{th}$  
Fourier coefficients; similarly as above, the solution of this system will prove  \eqref{eq:2nd-vanishing} for $q=4$.\medskip
\end{itemize}
}

\item {{\bf Case $q_0=5$ and the general case:}}
{Along the same lines described in the previous two items, 
the case $q_0=5$ will be discussed in Section \ref{case5}.
Moreover, in Section \ref{case-general} we will outline 
a general (conditional) procedure to derive \eqref{eq:2nd-vanishing} for  any  $q_0\geq 6$; the implementation of 
this scheme is based on the assumption that  certain explicit 
non-degeneracy conditions for the corresponding linear systems hold (see Remarks \ref{remarkmatrixAodd} and \ref{remarkmatrixAeven}). 
}
\end{itemize}
\smallskip 

\noindent\textbf{Step 3$'$:} {Finally, once the previous steps 
are completed, we adapt the approximation arguments  
from \cite{ADK} and show that $\Omega$ must be an ellipse; see  Section \ref{main-proof} for more details.}\\

\medskip

 \section{Necessary conditions for the existence of a caustic with rational rotation number}\label{NCFC}

\subsection{Elliptic billiard dynamics and caustics} 
\label{subsec:ellipticdynamics}
Now we want to provide a more precise description of 
the billiard dynamics in $\cE_{e,c}$.
We rely on notations of Section \ref{secstrategyproof}. 
{In addition, we need the following notations.

Let $0\leq k<1$, we define  elliptic integrals and 
Jacobi Elliptic functions:
\begin{itemize}
\item Incomplete elliptic integral of the first kind:
\[F(\varphi;k):=\int_{0}^{\varphi}\frac{1}{\sqrt{1-k^2\sin^2\tau}}d\tau.\]
\item Complete elliptic integral of the first kind:
\[K(k):=F\left(\frac{\pi}{2};k\right).\]
\item Jacobi Elliptic functions are obtained by inverting 
incomplete elliptic integrals of the first kind. Precisely, if 
\[
u:=F(\varphi;k)=\int_0^{\varphi}\frac{1}{\sqrt{1-k^2\sin^2\tau}}\,d\tau,
\]
then we define
$$
\varphi:=\text{am}(u;k).
$$ 
The Jacobi elliptic functions are given by:
\[\begin{split}&\sn(u;k):=\sin(\text{am}(u;k))=\sin(\varphi),\\
& \cn(u;k):=\cos(\text{am}(u;k))=\cos(\varphi).\end{split}\] \\
\end{itemize}
}

The following result has been proven in \cite{CF} (see also \cite[Lm. 2.1]{DCR}).\\

\begin{proposition}\label{prop1}
Let  $\l \in (0, b)$ and let
$$
k_\l^2:= \frac{a^2-b^2}{a^2-\l^2} \quad {\rm and} \quad
\d_\l := 2\, F( \arcsin (\l/b); k_\l).
$$
Let us denote, in cartesian coordinates, $q_\l(t) := (a\, \cn (t;k_\l), b\, \sn (t;k_\l))$.  
Then, for every $t\in [0, 4K(k_\l))$  the segment joining $q_\l(t)$ and $q_\l(t+\d_\l)$ 
is tangent to the caustic $C_\l$, defined in \eqref{caustic}.\medskip
\end{proposition}

Observe that:
\begin{itemize} 
\item $k_\l$ is a strictly increasing function of $\l \in (0,b)$; in particular 
$k_\l \rightarrow e$ as $\l \rightarrow 0^+$, while $k_\l \rightarrow1$ as $\l\rightarrow b^-$.
Observe that $k_\l$ represents the eccentricity of the ellipse $C_\l$.
\item
$\delta_\l$ is also a strictly increasing function of $\l \in (0,b)$; in fact, 
$F(\f;k)$ is clearly strictly increasing in both $\f$ and $k \in [0,1)$. Moreover, 
$\d_\l \rightarrow 0$ as $\l\rightarrow 0^+$, and $\d_\l \rightarrow +\infty$ as $\l\rightarrow b^-$.
\end{itemize}


\bigskip

Let us now consider the parametrization of the boundary induced by the dynamics corresponding to  the caustic $C_\l$:
\begin{eqnarray*} 
Q_\l : \R/2\pi\Z &\longrightarrow& \R^2 \nonumber \\
\theta &\longmapsto& q_\l \left(\frac{4K(k_\l)}{2\pi} \, \theta \right). \nonumber
\end{eqnarray*}

We define the {\it rotation number} associated to the caustic $C_\l$ to be
\begin{equation}\label{rot-func}
\omega(\l,e):= \frac{\d_\l}{4K(k_\l)} = 
\frac{F(\arcsin (\l/b); k_\l)}{2K(k_\l)}.
\end{equation}
In particular $\omega(\l,e)$ is strictly increasing in $\l$ and $\omega(\l,e) \longrightarrow 0$ as 
$\l\rightarrow 0^+$, while $\omega(\l,e) \rightarrow \frac{1}{2}$ as $\l\rightarrow b^-$.
In addition, for every $\theta\in \T$, 
the orbit starting at $Q_\l(\theta)$ and tangent to $C_\l$, hits the boundary at 
$Q_\l(\theta + 2\pi\, \omega_\l)$. 
\medskip

Denote the inverse of the function $\omega(\lambda,e)$, by 
$\lambda_{\omega}=\lambda(e,\omega).$  Notice  that {in the Taylor expansion of
$\omega(\lambda,e)$ only  even powers of $e$ appear,  hence the same holds for 
$\lambda(e,\omega)$.} Moreover, 
$$
\omega(\lambda,0)=\frac{\arcsin(\lambda/b)}{\pi},
$$  
and { it is straightforward to show that } the following estimate holds.

\begin{lemma} \label{rot-no1}There exists  $C>0$ such that for each $e\in[0,\frac{1}{2}]$ 
and $\omega\in(0,1/2)$, we have
$$|\lambda(e,\omega)-b\sin \omega \pi|\leq Ce^2.$$
\end{lemma}

\medskip

We want to write the boundary parametrization induced by the caustic $C_\l$, 
expressed in elliptic coordinates $(\mu, \f)$, namely determine the function 
$S_\l(\theta)=(\mu_\l(\theta),\f_\l(\theta)) = (\mu_0,\f_\l(\theta))$ such that 
the orbit starting at $S_\l(\theta)$ (in elliptic coordinates) and tangent to $C_\l$, 
hits the boundary at $S_\l(\theta + 2\pi\, \omega_\l)$.

It is easy to deduce  from the above expression that
\begin{equation}\label{action-para}
\varphi_{\lambda}(\theta)=\varphi(\theta,\lambda,e) := \am \left( \frac{4K(k_\l)}{2\pi}\,\theta ; k_\l \right).
\end{equation}
Therefore, we have 
$S_\l(\theta)=\left(\mu_0, \am \left( \frac{4K(k_\l)}{2\pi}\,\theta; k_\l \right) \right)$.  
The parametrization given in  \eqref{action-para} is called the {\it action-angle 
parametrization} of the boundary, associated to the caustic $C_{\lambda}$. 

Below, if we need to emphasize the rotation number of the associated caustic,  
we will write $C_{\lambda_\omega}$.
\medskip

Fix a positive integer $m$  and consider a perturbation of 
the ellipse~$\cE_{e,c}$, denoted $\Omega_\e$, {\it i.e.}, in the elliptic coordinates
$$
\partial \Omega_\e=\{(\mu,\varphi):\; \mu={\mu_0+}\mu_{\varepsilon}(\varphi),\;\varphi\in[0,2\pi)\},
$$
where $\mu_{\varepsilon}(\varphi)$ is a $2\pi$-periodic  $C^m$-function with 
$\|\mu_{\varepsilon}\|_{C^m}\leq M$ and $\|\mu_{\varepsilon}\|_{C^1}\le \varepsilon$, 
{for some sufficiently small $\e>0$}.\\

 \begin{lemma}\label{rational-condition}
For any rational number $p/q\in(0,1/2)$ in  lowest terms, 
if the billiard inside $\Omega_\e$ admits 
{an integrable rational caustic $C^{\e}_{\lambda_{p/q}}$}
of rotation number $p/q$,  then there exist constants  
$c_{p/q}$ and $C=C(m)$ such that
\begin{equation}\label{con1}
\lambda_{p/q}
\sum_{k=1}^q
\mu_{\varepsilon}(\varphi_{\lambda_{p/q}}
(\theta+\frac{kp}{q}2\pi))=
c_{p/q}+\Upsilon_{p/q}({\theta}), 
\end{equation}
where\footnote{We drop
dependence on $\varepsilon$ in the notations.} 
$\Upsilon_{p/q}\in C^{m-1}(\mathbb{T})$ and
$$
\|\Upsilon_{p/q}\|_{C^{m-1}}\leq Cq \qquad
\|\Upsilon_{p/q}\|_{C^0}\leq Cq^7\|\mu_{\varepsilon}\|_{C^1}^2.
 $$
\end{lemma}

 \medskip

\brm 
We need the higher regularity estimates from this lemma to prove Lemma \ref{key-lemma1}. \\
\erm 
 
 \begin{proof}
 Let $\varphi_{k}$, $k=0,\dots, q-1$, be the vertices of 
 the maximal $(p,q)$-gon inscribed in the ellipse $\cE_{e,c}$,
 tangent to the caustic $C^0_{\l_{p/q}}$ of the billiard map in 
 $\cE_{e,c}$, {with $\varphi_0$ being $\theta$-dependent.}  
 Let $\varphi^{\varepsilon}_{k}$, $k=0,\dots,q-1$, be 
 the vertices of the maximal $(p,q)$-gon inscribed in $\Omega_\e$, 
tangent to the caustic $C^{\e}_{\l_p/q}$, with $\varphi^{\varepsilon}_0=\varphi_0$. Then, by Lemma 5 in 
\cite{ADK}, we have that there exists $C>0$, independent 
of $q$, such that
 \begin{equation}\label{bound1}
 |\varphi_k-\varphi_k^{\varepsilon}|\leq Cq^3\|\mu_{\varepsilon}\|^2_{C^1},\quad k=0,\dots, q-1.
 \end{equation}
 For $\|\mu_{\epsilon}\|_{C^1}$ small enough,  the generating function of the billiard dynamics inside $\Omega_{\varepsilon}$ is given by
 $$h_{\varepsilon}(\varphi,\varphi')=h_0(\varphi,\varphi')+h_1(\varphi,\varphi')+h_2(\varphi,\varphi'),
 $$
 where $h_0(\varphi,\varphi')$ is the generating function of the billiard dynamics inside the ellipse 
 $\cE_{e,c}$ and
 $$\|h_1(\varphi,\varphi')\|_{C^1}\leq 2\|\mu_{\varepsilon}\|_{C^1} \qquad
 \|h_2(\varphi,\varphi')\|_{C^0}<C\|\mu_{\varepsilon}\|_{C^1}^2,
 $$
 and
 $$\|h_1\|_{C^m}+\|h_2\|_{C^m}\leq C.$$
 Using \cite[Proposition 4.1]{PR}, we deduce
\begin{equation}\label{key1}
\sum_{k=0}^{q-1}h_1(\varphi_k,\varphi_{k+1})=2\lambda_{p/q}\sum_{k=0}^{q-1}\mu_{\varepsilon}(\varphi_k).
\end{equation}
By the existence of { an integrable rational }
caustic with rotation number $p/q$ for the billiard dynamics inside $\Omega_\eps$, we have
\begin{equation}\label{const1}L(\varphi):=\sum_{k=0}^{q-1}h_{\varepsilon}(\varphi_k^{\varepsilon},\varphi_{k+1}^{\varepsilon})=const., \quad \mbox{{where $\varphi^\varepsilon_0=\varphi$}}.\end{equation}
Since
\[\begin{split}
&\ \ \ \sum_{k=1}h_0(\varphi_k^{\varepsilon},\varphi_{k+1}^{\varepsilon})=\sum_{k=0}^{q-1} \big[h_0(\varphi_k^{\varepsilon},\varphi_{k+1}^{\varepsilon})-h_0(\varphi_k,\varphi_{k+1})+h_0(\varphi_k,\varphi_{k+1})\big] \\
&=\sum_{k=1}^{q-1}\big[(\partial_1h_0(\varphi_{k},\varphi_{k+1})+\partial_2h_0(\varphi_{k-1},\varphi_k))(\varphi_k^{\varepsilon}-\varphi_k)+h_0(\varphi_k,\varphi_{k+1})
\\
& \qquad +\;O(|\varphi_k^{\varepsilon}-\varphi_k|^2)\big] \\
&=\sum_{k=0}^{q-1}h_0(\varphi_k,\varphi_{k+1})+\Upsilon_{p/q}^0,
\end{split}\]
with $$\|\Upsilon_{p/q}^0\|_{C^0}=O(q^7\|\mu_{\varepsilon}\|_{C^1}^2),\quad \|\Upsilon_{p/q}^0\|_{C^m}\leq q\,C_0(m),$$
and
$$
\sum_{k=0}^{q-1}h_1(\varphi_k^{\varepsilon},\varphi_{k+1}^{\varepsilon})=
\sum_{k=0}^{q-1}h_1(\varphi_k,\varphi_{k+1})+\Upsilon_{p/q}^1,$$
with $$\|\Upsilon_{p/q}^1\|_{C^0}=O(q^7\|\mu_{\varepsilon}\|_{C^1}^2),\quad \|\Upsilon_{p/q}^1\|_{C^{m-1}}\leq q\, C_1(m),$$
using  the fact that 
$$\varphi_k =\varphi_{\lambda_{p/q}}\left(\theta+\frac{pk}{q}2\pi\right),\quad k=0,\dots, q-1,$$
 the assertion of the lemma follows from \eqref{key1} and \eqref{const1}.
 \end{proof}
 
 \medskip
 Let us consider the Fourier series of $\mu_{\varepsilon}(\varphi)$,
 \[
 \mu_{\varepsilon}(\varphi)=\mu'_0+
 \sum_{k=1}^{+\infty}a_k\cos (k\varphi) +b_k\sin (k\varphi).
 \]
 Now substitute into $\mu_{\varepsilon}(\varphi)$ the action-angle parametrization 
 $\varphi=\varphi_{\lambda}(\theta)$, and expand it with respect to the eccentricity $e$. {By Lemma \ref{lm:mu-expansion}}, we obtain, 
 for any positive integer $N\in \mathbb{N}$, $N\leq m-1$, that
 \begin{equation}\label{expan1}
 \mu_{\varepsilon}(\varphi(\theta,\lambda,e))=\mu_{\varepsilon}(\theta)+\sum_{n=1}^N P_n(\theta)\frac{a^ne^{2n}}{(a^2-\lambda^2)^n}+O(\|\mu_{\varepsilon}\|_{C^{N+1}}e^{2N+2}),\end{equation}
 where
 $$
 P_n(\theta)=\sum_{k=1}^{+\infty}\sum_{l=-n}^{n}\xi_{n,l}(k)\big(a_k\cos((k+2l)\theta) +b_k\sin((k+2l)\theta)\big),
 $$
and  $\xi_{n,l}(k)$ are polynomials in $k$ (see Appendix 
\ref{e-expansion-sec}). Let us  now recall  the following elementary 
identities {(we leave the proof to the reader)}.
\begin{lemma}
Let  $0<p/q\in\mathbb{Q}$ in  lowest terms. \  If 
 $ n\in\mathbb{N}\setminus q\mathbb{N}$, then
{\small $$
\sum_{m=1}^q\cos \left(n\left(\theta+\frac{pm}{q}2\pi\right)\right)\equiv 0\quad 
 \mbox{and} \quad
 \sum_{m=1}^q\sin\left(n\left(\theta+\frac{pm}{q}2\pi\right)\right)\equiv 0.
$$
}
If 
$n\in q\mathbb{N}$, then
{\small $$\sum_{m=1}^q\cos\left(n\left(\theta+\frac{pm}{q}2\pi\right)\right)\equiv 
q\cos n\theta \ \ 
,
\ \  \sum_{m=1}^q\sin\left(n\left(\theta+\frac{pm}{q}2\pi\right)\right)\equiv 
q\sin n\theta.$$
}
\end{lemma}

\medskip

If we apply the above equalities to \eqref{con1} and \eqref{expan1}, we obtain
\[\begin{split}&\sum_{j=1}^{+\infty}a_j\cos (jq\,\theta) +b_j \sin (jq\,\theta)\\
&+\sum_{n=1}^N\sum_{l=-n}^{n}\xi_{n,l}(jq-2l)\Big(a_{jq-2l}\cos (jq\,\theta) +b_{jq-2l}\sin (jq\,\theta)\Big)\frac{a^ne^{2n}}{(a^2-\lambda_{p/q}^2)^n}\\
&=O(\|\mu_{\varepsilon}\|_{C^{N+1}}e^{2N+2}+\lambda_{p/q}^{-1}q^7\|\mu_{\varepsilon}\|_{C^1}^2)+\frac{c_{p/q}}{q}-\mu'_0.\end{split}\]

Multiplying both sides by $\cos (q\,\theta)$ and integrating with respect to $\theta$ from $0$ to $2\pi$, we get the following proposition.\medskip

\begin{proposition} {Let $0<p/q \in \Q\cap (0,1)$ and  assume that $\Omega$ admits an integrable rational 
caustic of rotation number $p/q$. }Let  $N\in \N$ such that $q>2N$. Then:
\begin{equation}\label{fourier1}\begin{split}
&a_q+\sum_{n=1}^N\ 
\sum_{|l|\le n}\xi_{n,l}(q-2l) \,a_{q-2l}\,\frac{a^ne^{2n}}{(a^2-\lambda_{p/q}^2)^n} \\
&=O(e^{2N+2}\|\mu_{\varepsilon}\|_{C^{N+1}}+
\lambda_{p/q}^{-1}q^7\|\mu_{\varepsilon}\|_{C^1}^2).\end{split}\end{equation}
\end{proposition}

\medskip

Similarly, if we multiply both sides by $\sin q\theta$ and integrate with respect to $\theta$ 
from $0$ to $2\pi$, we obtain the analogous equality for $b_q$.\\

\medskip

\section{The case $q_0=3$}\label{case3}
In this section we consider a $3$-rationally integrable domain $\Omega$, whose boundary is 
${C^3}$--close to  an ellipse $\cE_{e,c}$, {\it i.e.}, for a {$C^3$}--small function $\mu(\varphi)$ we have  
$$\partial \Omega=\cE_{e,c}+\mu(\varphi).$$
Let 
$$\mu(\varphi)=\mu'_0+\sum_{k=1}^{+\infty}a_k\cos (k\varphi) +b_k\sin (k\varphi),$$
and  assume 
\be \label{3:small}
\|\mu\|_{C^1}\leq e^6.
\ee
We will show that {the higher order relations} on the existence of  
 integrable rational caustics of rotation numbers  {$\frac{1}{2k+1}, k \geq 1$, 
and $\frac{2}{7}$} imply that 
\be \label{bnd:q3-a3} 
a_3=O(e^2\|\mu\|_{C^3})\quad ,\quad b_3=O(e^2\|\mu\|_{C^3}).\\
\ee

\brm \label{q0=3}
The proof in this case consists of one step and does not require {any other iteration of the same argument (compare also with 
Remarks \ref{q0=4} and \ref{q0=5})}.\\
\erm 

For simplicity, we assume  that the semi-major axis of $\cE_{e,c}$ equals to $1$, 
{\it i.e.}, $c=e$, {and we denote it simply by $\cE_e$}. 

{Let us start by observing the following lemma, which is a special case of  Lemma 
\ref{fourier41} (and of Lemma \ref{Fourier5} for $k_0=2$).  }
\medskip

\begin{lemma}
\label{fourier2}$a_5=O(e^2\|\mu\|_{C^1})$,  $a_7,\; a_9,\; a_{11}=O(e^4\|\mu\|_{C^2})$.\\
\end{lemma}

\brm {{Although we do not provide a direct proof of this lemma, let us point out that 
it exploits the}}
existence of 
 integrable rational caustics of rotation numbers  
$\frac{1}{5}, \,\frac 17,\, \frac 19, \frac 1{11}$, and $\frac{1}{13}$.
\erm 

\medskip


{Let us now show how property  \eqref{bnd:q3-a3} follows 
from this lemma.
\begin{itemize}
\item
From the existence of { an integrable rational}
caustic with rotation number $1/5$, using \eqref{fourier1} {and \eqref{3:small}} 
with $N=1$, we deduce that
\[
a_5+\big[\xi_{1,-1}(7)\,a_7+\xi_{1,1}(3)\,a_3\,\big]\frac{e^2}{1-\lambda_{1/5}^2}+O(e^4\|\mu\|_{C^2})=0,
\]
where $\xi_{1,\pm1}(k)=\pm\frac{k}{16}$ 
(see Appendix \ref{e-expansion-sec}). Hence, it follows from Lemmata \ref{rot-no1} and  \ref{fourier2} that
\begin{equation}\label{5-caustic}
a_5+\frac{3a_3}{16}\frac{e^2{(1+O(e^2))}}{\cos^2\frac{\pi}{5}}=
O(e^4\|\mu\|_{C^2}).
\end{equation}
which implies
\begin{equation}\label{5-caustic}
a_5+\frac{3a_3}{16}\frac{e^2}{\cos^2\frac{\pi}{5}}=
O(e^4\|\mu\|_{C^2}).
\end{equation}

\item From the existence of { an integrable rational } 
caustic with rotation number $1/7$, using \eqref{fourier1} {and \eqref{3:small}} 
with $N=2$, we obtain that
\[
a_7+\sum_{n=1}^2\sum_{l=-n}^{n}\xi_{n,l}(7-2l)
\,a_{7-2l}\,\frac{e^{2n}}{(1-\lambda_{1/7}^2)^n}+
O(e^6\|\mu\|_{C^3})=0,
\]
where  $\xi_{2,2}(k)=\frac{k^2+k}{512}$ (see Appendix \ref{e-expansion-sec}).

Observe that, as  
it follows from Lemma~\ref{rot-no1}, 
\begin{equation}\label{estimate1minuslambda}
\frac{1}{(1-\lambda_{\omega}^2)^n}=\frac{1+O(e^2)}{{\cos^{2n}}(\pi \omega)}.
\end{equation}

{Hence, we obtain:
\[
a_7+\sum_{n=1}^2 \sum_{l=-n}^{n} \xi_{n,l}(7-2l)
\,a_{7-2l}\,\frac{ (1+O(e^2))}{\cos^{2n}(\frac{\pi}{7})} e^{2n}
=
O(e^6\|\mu\|_{C^3}),
\]
which implies, using the estimates in Lemma \ref{fourier2},  that
 \begin{equation}\label{7-caustic1}a_7+
\frac{5a_5}{16}\frac{e^2}{\cos^2(\frac{\pi}{7})}+
\frac{12a_3}{512}\frac{e^4}{\cos^4(\frac{\pi}{7})}=O(e^6\|\mu\|_{C^3}).
\end{equation}
In fact, for $n=2$ the terms $e^{2n}O(e^2) = O(e^6)$. For $n=1$, the same is true, observing that 
$a_5=O(e^2\|\mu\|_{C^1})$,  $a_7=O(e^4\|\mu\|_{C^2})$ and $a_9=O(e^4\|\mu\|_{C^2})$, as it follows from Lemma \ref{fourier2} (see also Sections \ref{subsecodd} and \ref{subseceven} for  more precise computations).
}
\\
\item Similarly, from  the existence of  { an integrable rational} caustic with 
rotation number $2/7$, we get
\begin{equation}\label{7-caustic2}
a_7+\frac{5a_5}{16}\frac{e^2}{\cos^2 \left(\frac{2\pi}{7}\right)}+
\frac{12a_3}{512}\frac{e^4}{\cos^4 \left(\frac{2\pi}{7}\right)}=O(e^6\|\mu\|_{C^3}).
\end{equation}
\medskip
\item 
Combining \eqref{5-caustic}, \eqref{7-caustic1} and \eqref{7-caustic2}, we obtain the linear system
\begin{equation}\label{linear-system-3}
\left(\begin{array}{ccc}\frac{3e^2}{16\cos^2 \left(\frac{\pi}{5}\right)} & 1 & 0 \\
\frac{12e^4}{512\cos^4\left(\frac{\pi}{7}\right)} & \frac{5e^2}{16\cos^2\left(\frac{\pi}{7}\right)} 
& 1 \\ \frac{12e^4}{512\cos^4\left(\frac{2\pi}{7}\right)} & \frac{5e^2}{16\cos^2\left(\frac{2\pi}{7}\right)} 
& 1\end{array}\right)\left(\begin{array}{c}a_3 \\
a_5 \\ a_7\end{array}\right)=
\left(\begin{array}{c}O(e^4\|\mu\|_{C^2}) \\ 
O(e^6\|\mu\|_{C^3})\\ O(e^6\|\mu\|_{C^3})\end{array}\right).
\end{equation}
 {Observe that coefficient matrix is invertible\footnote{By means of Mathematica, one can compute that 
its determinant is $-4.182\times 10^{-4}e^4.$}; moreover, using Theorem \ref{thm:inv-matrix} in Appendix \ref{appendix-matrix}
we can compute the first row of its inverse, which has the form 
\begin{equation}\label{hierstruc1}
\left(\begin{array}{ccc}O(e^{-2}) & O(e^{-4}) & O(e^{-4})\end{array}\right).
\end{equation}}
This allows us to conclude that
$a_3=O(e^2\|\mu\|_{C^3}).$\\
\end{itemize}
\medskip
Similarly,
one can prove that $b_3=O(e^2\|\mu\|_{C^3}).$\\
 \medskip
}

\section{The case $q_0=4$}\label{case4}
In this section we consider a $4$-rationally integrable domain $\Omega$, whose boundary is 
$C^6$--close to  an ellipse {$\cE_{e}$} (also here we assume $c=e$),
{\it i.e.} 
$$\partial \Omega=\cE_{e}+\mu(\varphi),$$
where $\mu$ is  a $C^6$--small function.
Let 
$$\mu(\varphi)=\mu'_0+\sum_{k=1}^{+\infty}a_k\cos (k\varphi) +b_k\sin (k\varphi),$$
and  assume 
\begin{equation}\label{3:smallbis} 
\|\mu\|_{C^1}\leq e^{12}.
\end{equation}
We will show that the {higher order conditions} on the  existence of 
integrable rational caustics of {rotation numbers 
$\frac{1}{2k+1},\ \frac{1}{2k+2},\ k\geq 2,$ along with $\frac{2}{9}$ and $\frac{3}{14}$, }
imply that 
\be \label{bnd:q4-a34} 
a_k=O(e^2\|\mu\|_{C^6}) \quad ,\quad 
b_k=O(e^2\|\mu\|_{C^6}),\quad k=3,4.
\ee

{
\brm \label{q0=4}
The proof in this case consists of two steps, {related to the} odd and even cases, and 
does not require {any iteration (compare also with Remarks \ref{q0=3} and \ref{q0=5})}.
We give a detailed account of the odd case below.
\erm }

{Let us start by stating the following lemma, which is also a special case of Lemma \ref{Fourier5}
(with $k_0=2$)}.\\


\begin{lemma}\label{fourier41} 
$$a_5=O(e^2{\|\mu\|_{C^3}}),\;a_7
=O(e^4 {\|\mu\|_{C^3}}),\; 
a_9,\; a_{11},\;a_{13}=O(e^6 {\|\mu\|_{C^3}}).$$
\end{lemma}

\medskip

{
\brm 
{Although an independent proof of} this lemma is not given, {let us} observe that 
{it exploits the} existence of  integrable rational caustics of rotation numbers  
{$\frac{1}{5}, \,\frac 17,\, \frac 19,\, \frac 1{11},\,\frac 1{13},\, \frac{1}{15}$, 
and $\frac{1}{17}$.}
\erm }

\medskip

Let us now describe how to  prove \eqref{bnd:q4-a34}.\\

Let us first show that 
$
a_3=O(e^2\|\mu\|_{C^4}).
$

{
\begin{itemize}
\item
From the existence of an {integrable rational} caustic with rotation number $1/5$, 
using equality \eqref{fourier1} (with $N=1$) together with 
Lemmata \ref{rot-no1} and \ref{fourier41}, we deduce that
{(see also the analogous discussion for \eqref{5-caustic})}
$$
a_5+a_3\,\xi_{1,1}(3)\,\frac{e^2}{\cos^2(\frac{\pi}{5})}(1+O(e^2))=O(e^4\|\mu\|_{C^2}).
$$
{Since $|a_3|\leq \|\mu\|_{C^2}$, the term $a_3\xi_{1,1}(3)e^2O(e^2)$ 
could be put into the error term $O(e^4\|\mu\|_{C^2})$ on the right-hand side.}\\

\item {Similarly to what done in \eqref{7-caustic1},  from  the existence of  { integrable rational} 
caustic with rotation numbers $1/7$, we get 
$$
a_7+a_5\,\frac{\xi_{1,1}(5)e^2}{\cos^2(\frac{\pi}{7})}+
a_3\,\frac{\xi_{2,2}(3)e^4}{\cos^4(\frac{\pi}{7})}=O(e^6\|\mu\|_{C^3}).
$$
}

\item {From  the existence of  {integrable rational} 
caustic with rotation number $1/9$, using \eqref{fourier1} (with $N=3$), 
\eqref{estimate1minuslambda} and \eqref{3:smallbis}, we obtain that
\[
a_9+\sum_{n=1}^3\sum_{l=-n}^{n}\xi_{n,l}(9-2l)
\,a_{9-2l}\,\frac{(1+O(e^2))}{\cos^{2n}(\frac{\pi}{9})} e^{2n}
+
O(e^8\|\mu\|_{C^4})=0,
\]
which implies, using the estimates in Lemma \ref{fourier41},  that
$$
a_9+a_7\,\frac{\xi_{1,1}(7)e^2}{\cos^2(\frac{\pi}{9})}+
a_5\,\frac{\xi_{2,2}(5)e^4}{\cos^4(\frac{\pi}{9})}+
a_3\,\frac{\xi_{3,3}(3)e^6}{\cos^6(\frac{\pi}{9})}=O(e^8\|\mu\|_{C^4}).
$$
In fact, for $n=3$ the terms $e^{2n}O(e^2) = O(e^8)$. For $n=1,2$, the same is true, observing that 
$a_5=O(e^2\|\mu\|_{C^3})$,  $a_7=O(e^4\|\mu\|_{C^3})$,  $a_9=a_{11}=a_{13}=O(e^6\|\mu\|_{C^3})$, 
as it follows from Lemma \ref{fourier41}
(see also Sections \ref{subsecodd} and \ref{subseceven} for  more precise computations).\\
}

\item {Similarly, from  the existence of  { an integrable rational} caustic with 
rotation number $2/9$, we get
$$
a_9+a_7\,\frac{\xi_{1,1}(7)e^2}{\cos^2(\frac{2\pi}{9})}+
a_5\,\frac{\xi_{2,2}(5)e^4}{\cos^4(\frac{2\pi}{9})}+
a_3\,\frac{\xi_{3,3}(3)e^6}{\cos^6(\frac{2\pi}{9})}=O(e^8\|\mu\|_{C^4}).
$$
}

\item 
Therefore, we obtain the following system of linear equations in the variables $a_3,\dots, a_9$:
{\small
\be \label{linear-system-4}
\left(\begin{array}{cccc}\dfrac{\xi_{1,1}(3)e^2}{\cos^2(\frac{\pi}{5})} & 1 & 0 & 0
\\
\dfrac{\xi_{2,2}(3)e^4}{\cos^4(\frac{\pi}{7})} & \dfrac{\xi_{1,1}(5)e^2}{\cos^2(\frac{\pi}{7})}& 1 & 0 
\\
\dfrac{\xi_{3,3}(3)e^6}{\cos^6(\frac{\pi}{9})} &\dfrac{\xi_{2,2}(5)e^4}{\cos^4(\frac{\pi}{9})} & 
\dfrac{\xi_{1,1}(7)e^2}{\cos^2(\frac{\pi}{9})} & 1 
\\
\dfrac{\xi_{3,3}(3)e^6}{\cos^6(\frac{2\pi}{9})} &\dfrac{\xi_{2,2}(5)e^4}{\cos^4(\frac{2\pi}{9})} & 
\dfrac{\xi_{1,1}(7)e^2}{\cos^2(\frac{2\pi}{9})} & 1 \end{array}\right)\left(\begin{array}{c}a_3 
\\
a_5 \\ a_7 \\ a_9\end{array}\right)=
\left(\begin{array}{c}O(e^4\|\mu\|_{C^2}) \\ O(e^6\|\mu\|_{C^3}) \\ O(e^8\|\mu\|_{C^4}) 
\\ O(e^8\|\mu\|_{C^4})\end{array}\right).\ \ 
\ee
}
{Observe that the coefficient matrix of this linear system is invertible\footnote{By means of Mathematica, one can compute that its determinant is $-4.02\times10^{-6}e^6$}; moreover, using Theorem \ref{thm:inv-matrix} in Appendix \ref{appendix-matrix}
we can compute the first row of its inverse, which has the form
\begin{equation}\label{hierstruc2}
\left(
\begin{array}{cccc}O(e^{-2}) & O(e^{-4}) & O(e^{-6}) & O(e^{-6})\end{array}
\right).
\end{equation}
}
All of this is enough to conclude that 
$$a_3=O(e^2\|\mu\|_{C^4}).$$
\end{itemize}
}

Let us now show that 
$
a_4=(e^2\|\mu\|_{C^6}).
$
In the same way as before, from the existence of  integrable rational
caustics with rotation numbers $1/6,\ 1/8,\ 1/10,\ 1/12,\ 1/14$ and $3/14$, 
we obtain a linear system of equations in the variables $a_4$, $a_6,\dots, a_{14}$.
\be \label{linear-system-5}
\begin{split}&\left(\begin{array}{cccccc}
\dfrac{\xi_{1,1}(4)e^2}{\cos^2(\frac{\pi}{6})} & 1 & 0 & 0 & 0 & 0 
\\
\dfrac{\xi_{2,2}(4)e^4}{\cos^4(\frac{\pi}{8})} & \dfrac{\xi_{1,1}(6)e^2}{\cos^2(\frac{\pi}{8})} & 1 & 0 & 0 & 0 
\\
\dfrac{\xi_{3,3}(4)e^6}{\cos^6(\frac{\pi}{10})} & \dfrac{\xi_{2,2}(6)e^4}{\cos^4(\frac{\pi}{10})} & \dfrac{\xi_{1,1}(8)e^2}{\cos^2(\frac{\pi}{10})} & 1 & 0 & 0 
\\
\dfrac{\xi_{4,4}(4)e^8}{\cos^8(\frac{\pi}{12})} & \dfrac{\xi_{3,3}(6)e^6}{\cos^6(\frac{\pi}{12})} & \dfrac{\xi_{2,2}(8)e^4}{\cos^4(\frac{\pi}{12})} & \dfrac{\xi_{1,1}(10)e^2}{\cos^2(\frac{\pi}{12})} & 1 & 0 
\\
\dfrac{\xi_{5,5}(4)e^{10}}{\cos^{10}(\frac{\pi}{14})} & \dfrac{\xi_{4,4}(6)e^8}{\cos^8(\frac{\pi}{14})} & \dfrac{\xi_{3,3}(8)e^6}{\cos^6(\frac{\pi}{14})} & \dfrac{\xi_{2,2}(10)e^4}{\cos^4(\frac{\pi}{14})} & \dfrac{\xi_{1,1}(12)e^2}{\cos^2(\frac{\pi}{14})} & 1 
\\ 
\dfrac{\xi_{5,5}(4)e^{10}}{\cos^{10}(\frac{3\pi}{14})} & \dfrac{\xi_{4,4}(6)e^8}{\cos^8(\frac{3\pi}{14})} & \dfrac{\xi_{3,3}(8)e^6}{\cos^6(\frac{3\pi}{14})} & \dfrac{\xi_{2,2}(10)e^4}{\cos^4(\frac{3\pi}{14})} & \dfrac{\xi_{1,1}(12)e^2}{\cos^2(\frac{3\pi}{14})} & 1\end{array}\right)
\\
&\times\left(\begin{array}{c}a_4 \\ a_6 \\ a_8 \\ 
a_{10} \\ a_{12} \\ a_{14}\end{array}\right)=\left(\begin{array}{c}O(e^4\|\mu\|_{C^2}) 
\\ O(e^6\|\mu\|_{C^3}) \\ O(e^8\|\mu\|_{C^4}) \\ O(e^{10}\|\mu\|_{C^5}) 
\\ O(e^{12}\|\mu\|_{C^6}) \\ O(e^{12}\|\mu\|_{C^6})\end{array}\right).
\end{split}
\ee
{Also in this case the coefficient matrix is non-degenerate\footnote{By means of Mathematica, one can compute that 
its determinant is $7.1437\times 10^{-5}e^{10}$}; moreover, using Theorem \ref{thm:inv-matrix} in Appendix \ref{appendix-matrix}
we can compute the first row of its inverse, which has  the form
\begin{equation}\label{hierstruc3}
\left(\begin{array}{cccccc}
O(e^{-2}) & O(e^{-4}) & O(e^{-6}) & 
O(e^{-8}) & O(e^{-10}) & O(e^{-10})\end{array}\right).
\end{equation}}
Hence, we can conclude that
$$a_4=O(e^2\|\mu\|_{C^6}).$$

\medskip

Repeating the same arguments as before, one can show the analogous equalities for $b_k$'s, namely
$$b_3=O(e^2\|\mu\|_{C^4}),\quad b_{4}=O(e^2\|\mu\|_{C^6}).
$$
\medskip

\section{The case $q_0=5$  }\label{case5}
In this section we consider a $5$-rationally integrable 
domain $\Omega$, whose boundary is $C^7$--close to  
an ellipse $\cE_e$ {(we continue to assume that $c=e$)}, 
{\it i.e.}, for a $C^7$--small function $\mu(\varphi)$ we have  
$$
\partial \Omega=\cE_e+\mu(\varphi).
$$
Let 
$$
\mu(\varphi)=\mu'_0+\sum_{k=1}^{+\infty}a_k\cos (k\varphi) +b_k\sin (k\varphi),
$$
and  assume \begin{equation}\label{3:smalltris}\|\mu\|_{C^1}\leq e^{14}.\end{equation}

\medskip

We will show that the {higher order conditions} 
on the  existence of integrable rational caustics of {rotation numbers 
$\frac{1}{2k+1},\ \frac{1}{2k+2},\ k\geq 2,$ along with 
$\frac{2}{11},\ \frac{2}{13}$ and 
$\frac{3}{16}$  }
imply that 
 $$
 a_k, b_k=O(e^2\|\mu\|_{C^7}),\quad k=3,4,5.\\
 $$

\medskip

{
\brm \label{q0=5}
The proof in this case consists of {three} steps: we start by analyzing 
the odd and even cases and in the {odd} case we  {need to iterate the argument once} (inductive step).
In the general case $q_0>5$, we will need the number of inductive 
steps {to be} $[q_0/2]-1$; see the beginning of Section \ref{case-general}.  \erm }
\smallskip

{Let us start by stating the following lemma, which is also a special case of Lemma \ref{Fourier5} (for $k_0=3$).}

\begin{lemma}\label{fourier3}   
$$a_7=O(e^2\|\mu\|_{C^4}),\ \ a_9=O(e^4\|\mu\|_{C^4}),\ \ 
    a_{11},a_{13},a_{15}=O(e^6\|\mu\|_{C^4}).$$
\end{lemma}

\medskip

\brm 
{Although an independent proof of} this lemma is not given, {let us} observe that 
{it exploits the} existence of integrable rational caustics of rotation numbers  
{$\frac{1}{5}, \,\frac 17,\, \frac 19, \frac 1{11}, 
\frac 1{13}, \frac{1}{15},\, \frac{1}{17}$, and $\frac{1}{19}$.}\\
\erm 

\medskip

Let us now show how property  \eqref{hainan} follows 
from this lemma.

\begin{itemize}

\item {From the existence of an {integrable rational} caustic with rotation number $1/7$, using 
 equality  \eqref{fourier1} with $N=1$, we have
\[
a_7+\Big(\xi_{1,1}(5)a_5 +\xi_{1,-1}(9)a_9 \Big) \frac{e^2}{1-\lambda^2_{1/7}}+O(e^4\|\mu\|_{C^2})=0.
\]
By Lemmata \ref{rot-no1} and  \ref{fourier3}, it follows that
\[
a_7+\xi_{1,1}(5)a_5\frac{e^2}{\cos^2(\pi/7)}=O(e^4\|\mu\|_{C^2}).
\]
}

\item {From the existence of  an integrable rational  
caustic with rotation number $1/9$, using \eqref{fourier1} (with $N=2$) {and \eqref{3:smalltris}}, we obtain that
\[
a_9+\sum_{n=1}^2\sum_{l=-n}^{n}\xi_{n,l}(9-2l)
\,a_{9-2l}\,\frac{e^{2n}}{(1-\lambda_{1/9}^2)^n}+
O(e^6\|\mu\|_{C^3})=0.
\]
Using \eqref{estimate1minuslambda}, we obtain
\[
a_9+\sum_{n=1}^2 \sum_{l=-n}^{n} \xi_{n,l}(9-2l)
\,a_{9-2l}\,\frac{ (1+O(e^2))}{\cos^{2n}(\frac{\pi}{9})} e^{2n}
=
O(e^6\|\mu\|_{C^3}),
\]
which implies, using the estimates in Lemma \ref{fourier3},  that
\[
a_9+a_7\frac{\xi_{1,1}(7)e^2}{\cos^2(\pi/9)}+a_5\frac{\xi_{2,2}(5)e^4}{\cos^4(\pi/9)}=O(e^6\|\mu\|_{C^3}).
\]
In fact, clearly for $n=2$ the terms $e^{2n}O(e^2) = O(e^6)$. For $n=1$, the same is true, observing that 
$a_7=O(e^2\|\mu\|_{C^1})$,  $a_9=O(e^4\|\mu\|_{C^2})$, and  $a_{11}, a_{13}=O(e^6\|\mu\|_{C^3})$, as it follows from Lemma \ref{fourier3} (see also Sections \ref{subsecodd} and \ref{subseceven} for  more precise computations).\\
}

\item {From  the existence of  {integrable rational} 
caustic with rotation number 
$1/11$, using \eqref{fourier1} (with $N=3$), \eqref{3:smalltris} and \eqref{estimate1minuslambda}, we obtain that
\[
a_{11}+\sum_{n=1}^3\sum_{l=-n}^{n}\xi_{n,l}(11-2l)
\,a_{11-2l}\,\frac{(1+O(e^2))}{\cos^{2n}(\frac{\pi}{11})} e^{2n}
+
O(e^8\|\mu\|_{C^4})=0,
\]
which implies, again using the estimates in Lemma \ref{fourier3},  that  
\[a_{11}+a_9\frac{\xi_{1,1}(9)e^2}{\cos^2(\pi/11)}+
a_7\frac{\xi_{2,2}(7)e^4}{\cos^4(\pi/11)}+
a_5\frac{\xi_{3,3}(5)e^6}{\cos^6(\pi/11)}=O(e^8\|\mu\|_{C^4});
\]
see also Sections \ref{subsecodd} and \ref{subseceven} for  more precise computations.\\
}

\item {Similarly, from the existence of an { integrable rational} 
caustic with rotation number  $2/11$, we obtain:
\[a_{11}+a_9\frac{\xi_{1,1}(9)e^2}{\cos^2(2\pi/11)}+
a_7\frac{\xi_{2,2}(7)e^4}{\cos^4(2\pi/11)}+
a_5\frac{\xi_{3,3}(5)e^6}{\cos^6(2\pi/11)}=O(e^8\|\mu\|_{C^4}).
\]}


\item
{Putting all of this information together, we obtain a system 
of linear equations with unknowns $a_5, a_7, a_9, a_{11}$ and the following 
coefficient matrix:}
\[\left(\begin{array}{cccc}\dfrac{\xi_{1,1}(5)e^2}{\cos^2(\pi/7)} & 1 & 0 & 0 
\\
\dfrac{\xi_{2,2}(5)e^4}{\cos^4(\pi/9)} & \dfrac{\xi_{1,1}(7)e^2}{\cos^{2}(\pi/9)} & 1 & 0 
\\
\dfrac{\xi_{3,3}(5)e^6}{\cos^6(\pi/11)} & \dfrac{\xi_{2,2}(7)e^4}{\cos^4(\pi/11)} 
& \dfrac{\xi_{1,1}(9)e^2}{\cos^2(\pi/11)} & 1 
\\
\dfrac{\xi_{3,3}(5)e^6}{\cos^6(2\pi/11)} & \dfrac{\xi_{2,2}(7)e^4}{\cos^4(2\pi/11)} 
& \dfrac{\xi_{1,1}(9)e^2}{\cos^2(2\pi/11)} & 1\end{array}\right).\]
This matrix is invertible\footnote{By means of Mathematica, one can compute that its determinant is 
$1.4\times 10^{-5}e^{6} $.} and, using Theorem \ref{thm:inv-matrix} in Appendix \ref{appendix-matrix}, we can compute 
the first row of its inverse, which has the form
\begin{equation}\label{hierstruc4}
\big(O(e^{-2}),\; O(e^{-4}),\; O(e^{-6}),\; O(e^{-6})\big).
\end{equation}
Hence, we conclude that 
\begin{equation}\label{hainan}
a_5=O(e^2\|\mu\|_{C^4}).
\end{equation}
\end{itemize}
\medskip

Now, using this new estimate,  we obtain the following improvement of Lemma~\ref{fourier3}. 
{This is the second inductive step aforementioned in Remark \ref{q0=5}. 
This Lemma follows from \eqref{assumption-i} for $k_0=3$ and $m=2$.}

\begin{lemma}\label{fourier4}
$$a_7=O(e^4\|\mu\|_{C^4}),\;a_{9}=O(e^6\|\mu\|_{C^4}),\; a_{11}=O(e^8\|\mu\|_{C^4}),\; a_{13}=O(e^{10}\|\mu\|_{C^5})$$
and $$a_{15},a_{17},a_{19},a_{21}=O(e^{10}\|\mu\|_{C^5}).$$
\end{lemma}

\medskip

Proceeding as before, using Lemmata \ref{rot-no1} and \ref{fourier4} and the equality \eqref{fourier1}, 
from {the higher order relations} on the existence of { integrable rational} caustics of rotation numbers $1/7$, $1/9$, 
$1/11$, $2/11$, $1/13$, and $2/13,$ we obtain the following linear system
{(see  Sections \ref{subsecodd} and \ref{subseceven} for  more precise computations)}:

\begin{equation}\label{big1-section6}
\begin{split}&\left(
\begin{array}{cccccc}\frac{\xi_{2,2}(3)e^4}{\cos^4(\pi/7)}&\frac{\xi_{1,1}(5)e^2}{\cos^2(\pi/7)} & 1 & 0 & 0&0 \\ 
\frac{\xi_{3,3}(3)e^6}{\cos^6(\pi/9)}&\frac{\xi_{2,2}(5)e^4}{\cos^4(\pi/9)} & \frac{\xi_{1,1}(7)e^2}{\cos^{2}(\pi/9)} & 1 & 0&0 \\ 
\frac{\xi_{4,4}(3)e^8}{\cos^8(\pi/11)}&\frac{\xi_{3,3}(5)e^6}{\cos^6(\pi/11)} & \frac{\xi_{2,2}(7)e^4}{\cos^4(\pi/11)} & \frac{\xi_{1,1}(9)e^2}{\cos^2(\pi/11)} & 1 &0 \\ 
\frac{\xi_{4,4}(3)e^8}{\cos^8(2\pi/11)}&\frac{\xi_{3,3}(5)e^6}{\cos^6(2\pi/11)} & \frac{\xi_{2,2}(7)e^4}{\cos^4(2\pi/11)} & \frac{\xi_{1,1}(9)e^2}{\cos^2(2\pi/11)} & 1&0 \\
\frac{\xi_{5,5}(3)e^{10}}{\cos^{10}(\pi/13)}&\frac{\xi_{4,4}(5)e^8}{\cos^8(\pi/13)}&\frac{\xi_{3,3}(7)e^6}{\cos^6(\pi/13)}&\frac{\xi_{2,2}(9)e^4}{\cos^4(\pi/13)}&\frac{\xi_{1,1}(11)e^2}{\cos^2(\pi/13)}&1 \\
\frac{\xi_{5,5}(3)e^{10}}{\cos^{10}(2\pi/13)}&\frac{\xi_{4,4}(5)e^8}{\cos^8(2\pi/13)}&\frac{\xi_{3,3}(7)e^6}{\cos^6(2\pi/13)}&\frac{\xi_{2,2}(9)e^4}{\cos^4(2\pi/13)}&\frac{\xi_{1,1}(11)e^2}{\cos^2(2\pi/13)}&1
\end{array}\right)\\ 
&\times\left(\begin{array}{c}a_3 \\ a_5 \\ a_7 \\ a_9 \\ a_{11} \\ a_{13}\end{array}\right)=
\left(\begin{array}{c}O(e^6\|\mu\|_{C^4}) \\ O(e^8\|\mu\|_{C^4}) 
\\ O(e^{10}\|\mu\|_{C^5}) \\ O(e^{10}\|\mu\|_{C^5}) 
\\ O(e^{12}\|\mu\|_{C^6}) \\ O(e^{12}\|\mu\|_{C^6})
\end{array}
\right).
\end{split}
\end{equation}
{This matrix of coefficients is invertible\footnote{By means of Mathematica, one can compute that its determinant is 
$6.86498\times10^{-15}e^{16} $.} and, using Theorem \ref{thm:inv-matrix} in Appendix \ref{appendix-matrix}, we can compute 
the first two rows of its inverse, which  have the form
\begin{equation}\label{hierstruc5}
\left(\begin{array}{cccccc}O(e^{-4}) & O(e^{-6}) & O(e^{-8}) & O(e^{-8}) & O(e^{-10}) & O(e^{-10}) 
\\ O(e^{-2})&O(e^{-4}) & O(e^{-6}) & O(e^{-6}) & O(e^{-8}) & O(e^{-8})\end{array}\right).
\end{equation}}
Therefore, we conclude that
$$a_3=O(e^2\|\mu\|_{C^6}),\quad a_5=O(e^4\|\mu\|_{C^6}).$$

Then, we  want to show that $$a_4=O(e^2\|\mu\|_{C^7}).$$
For this,  we exploit relations \eqref{fourier1}, obtained from {the higher order conditions on} the existence  caustics with rotation numbers 
$1/6,\ 1/8,\ 1/10$, $1/12,\ 1/14,\ 1/16$ and $3/16$. 

In the same spirit as before {(see  Sections \ref{subsecodd} and \ref{subseceven} for  more precise computations)}, we get the following  linear system:
\[\begin{split}&\left(\begin{array}{ccccccc}\frac{\xi_{1,1}(4)e^2}{\cos^2(\pi/6)} & 1 & 0 & 0 & 0 & 0 &0 \\  \frac{\xi_{2,2}(4)e^4}{\cos^4(\pi/8)} & \frac{\xi_{1,1}(6)e^2}{\cos^2(\pi/8)} & 1 & 0 & 0 & 0 & 0 \\ \frac{\xi_{3,3}(4)e^6}{\cos^6(\pi/10)} & \frac{\xi_{2,2}(6)e^4}{\cos^4(\pi/10)} & \frac{\xi_{1,1}(8)e^2}{\cos^2(\pi/10)} & 1 & 0 & 0 & 0 \\ \frac{\xi_{4,4}(4)e^8}{\cos^8(\pi/12)} & \frac{\xi_{3,3}(6)e^6}{\cos^6(\pi/12)} & \frac{\xi_{2,2}(8)e^4}{\cos^4(\pi/10)} & \frac{\xi_{1,1}(10)e^2}{\cos^2(\pi/12)} & 1 & 0 & 0 \\ \frac{\xi_{5,5}(4)e^{10}}{\cos^{10}(\pi/14)} & \frac{\xi_{4,4}(6)e^8}{\cos^8(\pi/14)} & \frac{\xi_{3,3}(8)e^6}{\cos^6(\pi/14)} & \frac{\xi_{2,2}(10)e^4}{\cos^4(\pi/14)} & \frac{\xi_{1,1}(12)e^2}{\cos^2(\pi/14)} & 1 & 0 \\ \frac{\xi_{6,6}(4)e^{12}}{\cos^{12}(\pi/16)} & \frac{\xi_{5,5}(6)e^{10}}{\cos^{10}(\pi/16)} & \frac{\xi_{4,4}(8)e^8}{\cos^8(\pi/16)} & \frac{\xi_{3,3}(10)e^6}{\cos^6(\pi/16)} & \frac{\xi_{2,2}(12)e^4}{\cos^4(\pi/16)} & \frac{\xi_{1,1}(14)e^2}{\cos^2(\pi/16)} & 1 \\ 
\frac{\xi_{6,6}(4)e^{12}}{\cos^{12}(3\pi/16)} & \frac{\xi_{5,5}(6)e^{10}}{\cos^{10}(3\pi/16)} & \frac{\xi_{4,4}(8)e^8}{\cos^8(3\pi/16)} & \frac{\xi_{3,3}(10)e^6}{\cos^6(3\pi/16)} & \frac{\xi_{2,2}(12)e^4}{\cos^4(3\pi/16)} & \frac{\xi_{1,1}(14)e^2}{\cos^2(3\pi/16)} & 1\end{array}\right)\\ 
&\times\left(\begin{array}{c}a_4 \\ a_6 \\ a_8 \\ a_{10} \\ a_{12} \\ a_{14} \\ a_{16}\end{array}\right)=\left(\begin{array}{c}O(e^4\|\mu\|_{C^2}) \\ O(e^6\|\mu\|_{C^3}) \\ O(e^8\|\mu\|_{C^4}) \\ O(e^{10}\|\mu\|_{C^5}) \\ O(e^{12}\|\mu\|_{C^6}) \\ O(e^{14}\|\mu\|_{C^7}) \\ O(e^{14}\|\mu\|_{C^7})\end{array}\right).
\end{split}\]
{This matrix is invertible\footnote{By means of Mathematica, one can compute that its determinant is 
$-2.5\times 10^{-6}e^{12} $.} and, using Theorem \ref{thm:inv-matrix} in Appendix \ref{appendix-matrix}, we can compute 
the first  row of its inverse, which has the form
\begin{equation}\label{hierstruc6}
{\small \left(\begin{array}{ccccccc}O(e^{-2}) & O(e^{-4}) & O(e^{-6}) & O(e^{-8}) & O(e^{-10}) & O(e^{-12}) & O(e^{-12})\end{array}\right).}
\end{equation}}
Hence, we obtain 
$$a_4=O(e^2\|\mu\|_{C^7}).\medskip$$

Similarly {(see again Sections \ref{subsecodd} and \ref{subseceven} for  more precise computations)}, one can prove that
$$
b_k=O(e^2\|\mu\|_{C^7}),\quad k=3,4,5.
$$\medskip

\section{The general cases}\label{case-general}
{In the previous sections, we have described how to recover 
the missing relations in the  cases $q_0=3,4,5$. Clearly, 
the same set of ideas can be implemented  for any $q_0 \geq 6$. 
In this section we aim to outline the procedure for proving 
these results in the general case.}

Let $q_0\geq 6$ and let us consider  a $q_0$-rationally integrable domain $\Omega$, 
whose boundary is close to an ellipse $\cE_e$ (we use the normalization $c=e$)
$$
\partial \Omega=\cE_e+\mu(\varphi).
$$
Let
$$\mu(\varphi)=\mu'_0+\sum_{k=1}^{+\infty}a_k\cos (k\varphi) +b_k\sin (k\varphi),$$
and assume 
$$
\|\mu\|_{C^1}\leq e^{6q_0}.\medskip
$$

 Without loss of generality, we assume that $q_0$ is an even integer, {\it i.e.}, {$$q_0=2k_0,\quad \text{with}\quad k_0\geq 3.$$}

 Let us  outline {this inductive procedure} to show that
$$a_k, b_k=O(e^2\|\mu\|_{C^{5k_0}}), \quad k=3, \dots, q_0=2k_0.$$

{The proof of this claim will be detailed in the following subsections (see Proposition \ref{basic-estimate} for a more precise summarizing statement). }\medskip

\medskip

Let us start with the following Lemma.

\begin{lemma}\label{Fourier5} 
$$
a_{2k+1}=\left\{
\begin{array}{ll}
O(e^{2(k-k_0)+2}\|\mu\|_{C^{{k_0+1}}}) & \mbox{if}\;k=k_0,\dots, 2k_0-1,\\ 
O(e^{2k_0+2}\|\mu\|_{C^{{k_0+1}}}) & \mbox{if }\; k={2k_0},\dots {3k_0},\\
{O(e^{2(4k_0-k)+2}\|\mu\|_{C^{k_0+1}})} & {\mbox{if}\;k=3k_0+1,\dots, 4k_0}.
\end{array}
\right.
$$
and   
$$
a_{2k}=\left\{
\begin{array}{ll}
O(e^{{2(k-k_0)}}\|\mu\|_{C^{{2k_0+1}}}) & \mbox{if}\;k=k_0+1,\dots, 3k_0+1,\\
O(e^{4k_0+2}\|\mu\|_{C^{{2k_0+1}}}) & \mbox{if }\; k=3k_0+2,\dots, 6k_0+1,\\
{O(e^{2(8k_0-k)+4} \|\mu\|_{C^{{2k_0+1}}})} & {\mbox{if}\; k=6k_0+2,\dots, 8k_0+1}.
\end{array}
\right.$$
\end{lemma}

\medskip

{
\brm
Notice that only the first two items in each bracket are really used for our proof. However, we report also the others for the sake of completeness.
\erm}

\medskip

%
%

\begin{proof}
{Let us start by proving the estimates for the Fourier coefficients of odd order. 
The proof consists in an iterative application of equality \eqref{fourier1}.}

From the existence of { integrable rational} 
caustics with  rotation numbers $\frac{1}{2k+1}$, with $k=k_0,\dots, 4k_0$ 
{(observe that this choice ensures that $\frac{1}{2k+1}<\frac{1}{q_0}$)},
using equality \eqref{fourier1} with $N=0$, we easily get:
\begin{equation}\label{iteration1_lemma7.1}
a_{2k+1}=O(e^2\|\mu\|_{C^1}),\quad k=k_0,\dots, 4k_0.\\
\end{equation}

\smallskip

{Let us now consider \eqref{fourier1} with $N=1$, for  rotation numbers $\frac{1}{2k+1}$, where $k=k_0+1,\ldots, 4k_0-1$:}
{\begin{eqnarray*}
a_{2k+1} &=& - \sum_{l= - 1}^1\xi_{1,l}(2(k-l)+1) \,a_{2(k-l)+1}\,\frac{e^{2}}{1-\lambda_{1/{2k+1}}^2} +O(e^{4}\|\mu\|_{C^2})   \\
&=& - \sum_{l= - 1}^1\xi_{1,l}(2(k-l)+1) \,a_{2(k-l)+1}\,\frac{e^{2}}{\cos^2\left({\frac{\pi}{2k+1}}\right)} +O(e^{4}\|\mu\|_{C^2}), 
\end{eqnarray*}
where in the last equality we have used Lemma \ref{rot-no1} 
which implies
\begin{equation}\label{expansionlambda}
1-\lambda_{1/{2k+1}}^2 = \cos^2{\frac{\pi}{2k+1}} + O(e^2).
\end{equation}
Observe now that, since $|l|\leq 1$, then $ k_0\leq k-l \leq 4k_0$, hence we can use estimates \eqref{iteration1_lemma7.1} and obtain:
$$
a_{2k+1}=O(e^4\|\mu\|_{C^2}),\quad k=k_0+1,\dots, 4k_0-1.\\
$$
}

{In order to prove the claim, we need to iterate the same argument until $N= k_0$.}\\

{ Let us describe how the inductive procedure works.
Suppose that we have already iterated the same argument for $N=1,\ldots,N_0 < k_0$; then, we have obtained:
\begin{equation}\label{induction1_lemma7.1}
a_{2k+1}=O(e^{2N_0+2}\|\mu\|_{C^{N_0+1}}),\quad k=k_0+N_0,\dots, 4k_0-N_0.
\end{equation}
Observe that  if $k_0\leq k<k_0+N_0$, then the index $k$ has been involved until the iteration corresponding to $N=k-k_0$; hence
\begin{equation}\label{induction2_lemma7.1}
a_{2k+1}=O(e^{2(k-k_0)+2}\|\mu\|_{C^{(k-k_0)+1}}),\quad k_0 \leq k < k_0+N_0.
\end{equation}
Similary, if $4k_0-N_0< k \leq 4k_0$, then the index $k$ has been involved until the iteration corresponding to $N=4k_0-k$; hence
\begin{equation}\label{induction3_lemma7.1}
a_{2k+1}=O(e^{2(4k_0-k)+2}\|\mu\|_{C^{(4k_0-k)+1}}),\quad 4k_0-N_0< k \leq 4k_0.\\
\end{equation}
}

{Apply now \eqref{fourier1} with $N=N_0+1$ and  rotation numbers 
$\frac{1}{2k+1}$, with $k=k_0+N_0+1,\ldots, 4k_0-N_0-1$. Then:}
{\begin{eqnarray}
a_{2k+1} &=& - \sum_{n=1}^{N_0+1} \sum_{|l|\leq n}\xi_{n,l}(2(k-l)+1) \,a_{2(k-l)+1}\,\frac{e^{2n}}{(1-\lambda_{1/{2k+1}}^2)^n} \nonumber\\
&& \quad +O(e^{2(N_0+1)+2}\|\mu_{\varepsilon}\|_{C^{N_0+2}}).   \label{induction4_lemma7.1}
\end{eqnarray}
We want to show that all terms in this sum can be included in the remainder. Let us distinguish several cases:
\begin{itemize}
\item $l=0$ appears for all $n\geq 1$ and, using \eqref{induction1_lemma7.1}, we conclude:
$$
a_{2k+1}\,\frac{e^{2n}}{(1-\lambda_{1/{2k+1}}^2)^n} = O(e^{2(N_0+1)+2}\|\mu\|_{C^{N_0+1}}).
$$
\item $0<l\leq N_0+1$ appears for all $n\geq l$; 
using \eqref{induction1_lemma7.1}, \eqref{induction2_lemma7.1} and \eqref{expansionlambda}, 
we can conclude:
\begin{eqnarray*}
&& a_{2(k-l)+1}\,\frac{e^{2n}}{(1-\lambda_{1/{2k+1}}^2)^n} \\
&&\quad = \; O(e^{2(k-l-k_0)+2}\|\mu\|_{C^{N_0+1}}) \cdot O(e^{2l}) \cdot \big(1+O(e^2)\big) \\
&&\quad = \; O(e^{2(k-k_0)+2}\|\mu\|_{C^{N_0+1}})\cdot \big(1+O(e^2)\big) \\
&&\quad = \; O(e^{2(N_0+1)+2}\|\mu\|_{C^{N_0+1}}) \cdot \big(1+O(e^2)\big)\\
&&\quad = \; O(e^{2(N_0+1)+2}\|\mu\|_{C^{N_0+1}}),
\end{eqnarray*}
where, in the second-last equality, we have used that $k\geq k_0+N_0+1$.
\item $0<-l\leq N_0+1$ appears for all $n\geq -l$; using \eqref{induction1_lemma7.1}, \eqref{induction3_lemma7.1} and \eqref{expansionlambda}, we can conclude:
\begin{eqnarray*}
&& a_{2(k+l)+1}\,\frac{e^{2n}}{(1-\lambda_{1/{2k+1}}^2)^n} \\
&&\quad = \; O(e^{2(4k_0-k+l)+2}\|\mu\|_{C^{N_0+1}}) \cdot O(e^{-2l}) \cdot \big(1+O(e^2)\big) \\
&&\quad = \; O(e^{2(4k_0-k)+2}\|\mu\|_{C^{N_0+1}})\cdot \big(1+O(e^2)\big) \\
&&\quad = \; O(e^{2(N_0+1)+2}\|\mu\|_{C^{N_0+1}}) \cdot \big(1+O(e^2)\big)\\
&&\quad = \; O(e^{2(N_0+1)+2}\|\mu\|_{C^{N_0+1}}),
\end{eqnarray*}
where, in the second-last equality, we have used that $k\leq 4k_0-N_0-1$.%
\end{itemize}
It follows from these estimates and \eqref{induction4_lemma7.1}  that
$$ a_{2k+1} = O(e^{2(N_0+1)+2}\|\mu_{\varepsilon}\|_{C^{N_0+2}}) \qquad {\rm for}\; k=k_0+N_0+1,\ldots, 4k_0-N_0-1.\\$$
}

{The claim of the theorem then follows by taking $N_0=k_0$ in \eqref{induction1_lemma7.1}, \eqref{induction2_lemma7.1} and \eqref{induction3_lemma7.1}.}\\

Similarly, one proves the relations corresponding to Fourier coefficients of  even order. 
{More specifically, one considers
{ integrable rational}  caustics with  rotation numbers $\frac{1}{2k}$, with $k=k_0+1,\dots, 8k_0+1$. As in the previous part, the proof consists in an iterative application of \eqref{fourier1}; in particular, in this case the number of needed iterations equals $2k_0$ (from which the appearance of the $C^{2k_0+1}$-norm). }\\
\end{proof}

\medskip


{Now, we want to describe how to recover the missing relations. 
We distinguish between Fourier coefficients corresponding to 
Fourier modes of, respectively, odd and even order.}\medskip

\subsection{Fourier coefficients of odd order Fourier modes} \label{subsecodd}
Let us prove that for every integer ${1} \leq m\leq k_0$, we have  that    
\begin{equation}\label{assumption-i}a_{2k+1}=
\begin{cases}O(e^{2(k-m)+2}\|\mu\|_{C^{{3k_0}}}), & k \in[m, 3k_0-m),\\
O(e^{2k_0+4(k_0-m)+2}\|\mu\|_{C^{{3k_0}}}),& k\in[3k_0-m, {6k_0-3m}]. \\
\end{cases}\end{equation}

\medskip

{\begin{remark}
The above estimates hold with sharper choices of the norms $\|\cdot\|_{C^k}$ (see \eqref{fourier1} and  Lemma \ref{Fourier5}). However, for the sake of simplicity we have opted for a common choice that is suitable for all steps involved in the  algorithm (see Remark \ref{choicenorm}).\\
\end{remark}}

{The argument below consists in a finite (backward) induction:
the first step corresponds to $m=k_0$, while the final one to $m=2$.}
{Observe that  Lemma \ref{Fourier5} implies  \eqref{assumption-i} for $m=k_0$. 
 Let us assume that \eqref{assumption-i} holds for a given $ 2\leq m \leq k_0$ (inductive hypothesis) and let us prove it for $m-1$.\\
}

{We denote $N(k):=k-m+1$.}\\

{ Let us fix $k\in \{k_0,\dots, 3k_0-m\}$; observe that for such a choice of $k$, there exists 
 an integrable rational caustic with rotation number  $\frac{1}{2k+1}$. Let us now apply  \eqref{fourier1} with $N=N(k)$:
\begin{equation}\label{eq1sec7.1}
a_{2k+1}+ \sum_{n=1}^{N(k)} \sum_{|l|\leq n} \frac{\xi_{n,l}(2(k-l)+1)\, e^{2n}}{(1-\lambda_{\frac{1}{2k+1}}^2)^n} a_{2(k-l)+1}=O(e^{2N(k)+2}\|\mu\|_{C^{N(k)+1}}).
\end{equation}}

{\begin{remark}\label{choicenorm}
Notice that all estimates involve $\|\mu\|_{C^{k+1}}$ and $\|\mu\|_{C^{N(k)+1}}$, for some $k\leq3k_0-m$ and $m\geq 1$; in particular,  
$N(k)\leq 3k_0-m+1 < 3k_0$.
Hence, we can choose to bound all terms with respect to $\|\mu\|_{C^{3k_0}}$. Hereafter, in order to simplify the notation, we will neglect this term and concentrate on
the part involving powers of the eccentricity $e$.
\end{remark}}

{Now we want to show that in \eqref{eq1sec7.1} the only terms in the sum that are not of the same order as the remainder are the ones corresponding to $l=n$.}\\

{
\begin{itemize}
\item Observe that if $0\leq l \leq n-1$, then 
$$
k-l \geq k-(N(k)-1) = m
$$
and, since $k_0\leq k \leq 3k_0-m$, we also have
$$
k-l \leq 3k_0-m.
$$
Using the inductive hypothesis, the fact that $0\leq l\leq n-1$ and \eqref{expansionlambda}, we 
get:
\begin{eqnarray*}
\frac{ a_{2(k-l)+1}}{(1-\lambda_{\frac{1}{2k+1}}^2)^n}  e^{2n} &=& O(e^{2(k-l-m)+2}) \cdot O(e^{2l+2}) \cdot (1+ O(e^2))\\
&=& O(e^{2(k-m+1)+2}) \cdot (1+ O(e^2)) \\
&=& O(e^{2N(k)+2}).
\end{eqnarray*}
\item Let us now consider negative $l$. \\
\medskip
First observe that if $l=-N(k)$, then clearly  
\begin{eqnarray*}
\frac{ a_{2(k+N(k))+1}}{(1-\lambda_{\frac{1}{2k+1}}^2)^{N(k)}}  e^{2N(k)} = O(e^{2N(k)+2}),
\end{eqnarray*}
where we have used that 
$a_{2(k+N(k))+1}=O(e^2)$, as it follows applying \eqref{fourier1} with $N=0$ (in fact, since $k+N(k)\geq k_0$, there exists by assumption an integrable rational caustic of rotation number $\frac{1}{2(k+N(k))+1}$).\\ 
\medskip
Let us now assume that $-N(k)+1\leq l <0$.\\
If $k+m-3k_0 \leq  l <0$, then $ k_0< k-l \leq 3k_0-m$; hence, using the inductive hypothesis we get:
\begin{eqnarray*}
\frac{ a_{2(k-l)+1}}{(1-\lambda_{\frac{1}{2k+1}}^2)^n}  e^{2n} &=& O(e^{2(k-l-m)+2}) \cdot O(e^{2}) \cdot (1+ O(e^2))\\
&=& O(e^{2(k-m+1)+2}) \cdot (1+ O(e^2)) \\
&=& O(e^{2N(k)+2}).
\end{eqnarray*}
On the other hand, if $ -n\leq l < k+m-3k_0$, then 
$$   k-l> 3k_0-m \quad {\rm and} \quad k-l \leq k+N(k)-1 \leq 6k_0-3m.$$  
Therefore, using the inductive hypothesis we get (we use that $n\geq -l >3k_0-m-k$):
\begin{eqnarray*}
\frac{ a_{2(k-l)+1}}{(1-\lambda_{\frac{1}{2k+1}}^2)^n}  e^{2n} &=& O(e^{2k_0+4(k_0-m)+2}) \cdot O(e^{2n}) \cdot (1+ O(e^2))\\
&=&  O(e^{2k_0+4(k_0-m)+2}) \cdot O(e^{6k_0-2m-2k}) \cdot (1+ O(e^2))\\
&=& O(e^{12k_0-6m-2k+2}) \cdot (1+ O(e^2))\\
&=& O(e^{2N(k)+2}) \cdot (1+ O(e^2))
\end{eqnarray*}
where in the last equality we have used that $m\leq k_0$, $k< 3k_0-m$ and therefore
\begin{eqnarray*}
12k_0-6m-2k+2 &=&  2(3k_0 -m) + 2(3k_0-m-k)  - 2m +2 \\
& \geq &  2 (k-m+1) + 2\\
&=& 2N(k)+2.\\
\end{eqnarray*}
\end{itemize}
} 
 
{Using these estimates, we see that \eqref{eq1sec7.1} becomes: 
\begin{eqnarray*}
a_{2k+1}+\sum_{j=1}^{N(k)}\frac{\xi_{j,j}(2k+1-2j)e^{2j}}{(1-\lambda_{\frac{1}{2k+1}}^2)^j}a_{2(k-j)+1}=O(e^{2N(k)+2}\|\mu\|_{C^{3k_0}}).
\end{eqnarray*}
Using Lemma \ref{rot-no1} and the inductive hypothesis, we see that for $j<N(k)$ (which implies $  m\leq k-j < 3k_0-m$), we have:
\begin{eqnarray*}
\frac{a_{2(k-j)+1}\, e^{2j}}{(1-\lambda_{\frac{1}{2k+1}}^2)^j}   &=& \frac{a_{2(k-j)+1}\, e^{2j}}{\cos^{2j}(\frac{\pi}{2k+1})} (1+O(e^2))\\
&=& \frac{a_{2(k-j)+1}\, e^{2j}}{\cos^{2j}(\frac{\pi}{2k+1})} + e^{2j+2}\, O(e^{2(k-j-m)+2})  \\
&=& \frac{a_{2(k-j)+1}\, e^{2j}}{\cos^{2j}(\frac{\pi}{2k+1})} + O(e^{2N(k)+2}).
\end{eqnarray*}
Clearly, for $j=N(k)$
\begin{eqnarray*}
\frac{a_{2(k-N(k))+1}\, e^{2N(k)}}{(1-\lambda_{\frac{1}{2k+1}}^2)^{N(k)}}   &=& 
 \frac{a_{2(k-N(k))+1}\, e^{2N(k)}}{\cos^{2N(k)}(\frac{\pi}{2k+1})} + O(e^{2N(k)+2}).
\end{eqnarray*}
Hence,  \eqref{eq1sec7.1} reduces to
\begin{equation}\label{eqodd1}
a_{2k+1}+\sum_{j=1}^{N(k)}\frac{\xi_{j,j}(2k+1-2j)e^{2j}}{\cos^{2j}(\frac{\pi}{2k+1})}a_{2k+1-2j}=O(e^{2N(k)+2}\|\mu\|_{C^{3k_0}})\\
\end{equation}
for $k=k_0,\ldots, 3k_0-m$.
}

\vspace{10 pt}
 
{If $k\geq 2k_0$,  then $\frac{2}{2k+1}<\frac{1}{2k_0}=\frac{1}{q_0}$.  
Since $\Omega$ is $q_0$-integrable, then we have also the existence of  integrable rational caustics with  rotation numbers 
$\frac{2}{2k+1}$. Hence, proceeding as above, for $k=2k_0,\dots, 3k_0-m$, we get 
\begin{equation}\label{eqodd3}
a_{2k+1}+\sum_{j=1}^{N(k)}
\frac{\xi_{j,j}(2k+1-2j)e^{2j}}{\cos^{2j}(\frac{2\pi}{2k+1})}a_{2(k-j)+1}=O(e^{2N(k)+2}\|\mu\|_{C^{3k_0}})
\end{equation}
for $k=2k_0,\ldots, 3k_0-m$.\\}

\vspace{10 pt}

\begin{remark}\label{remarkmatrixAodd}
{We obtain $3k_0-2m+2$ linear equations with $3k_0-2m+2$ unknown variables: $a_{2{m}-1},\; a_{2m+1},\dots,a_{2(3k_0-m)+1}.$
Let us consider the  system of linear equations consisting of: 
\begin{itemize}
\item the $k_0$ equations corresponding to \eqref{eqodd1} for $k=k_0,\ldots, 2k_0-1$;
\item the $(k_0-m+1)$ couples of equations corresponding to \eqref{eqodd1} and \eqref{eqodd3} for $k=2k_0,\ldots, 3k_0-m.$
\end{itemize}
Recall that $q_0=2k_0$. Denote by $\cA_{q_0,m}^{\tiny(odd)}\in {\cM}_{3k_0-2m+2}(\R)$ 
the square matrix of the coefficients associated to this  system.
In particular,
the matrix $\cA_{q_0,m}^{\tiny(odd)}$ has the following structure 
\begin{equation}\label{matrixAodd}
\cA_{q_0,m}^{\tiny (odd)} = 
\left(
\begin{array}{c|c|c}
* & {\mathcal L} & {\mathcal O}  \\
\hline
* & * & {\mathcal K} 
\end{array}
\right)
\end{equation}
where
\begin{itemize}
\item ${\mathcal L}$ is a lower triangular $k_0\times k_0$ matrix with $1$'s on the diagonal;
\item ${\mathcal K}$ is a $(k_0-m+1)\times 2(k_0-m+1)$ matrix of the form
{\small
$$
{\mathcal K}= 
\left(
\begin{array}{ccccc}
1 & 0 & 0 & \ldots & 0 \\
1 & 0 & 0 & \ldots & 0 \\
* &1 & 0 & \ldots & 0 \\
* &1 & 0 & \ldots & 0 \\
\vdots & \vdots &  \vdots & \vdots &\vdots  \\
%
* &* & * & \ldots &  1\\
* &* & * & \ldots &  1
\end{array}
\right);
$$}
\item ${\mathcal O}$ is a block of zeros of size $k_0 \times (k_0-m+1)$;
\item observe that each row has a unit on it and the $*$ entries are a multiple of $e^{2j}$, where $j \in \N$ represents the  ``distance'' from the unit within the row;
in particular all $*$ entries are of the form ${\xi}\,{ \cos^{-2j} ( w \pi )  e^{2j}}$, where $\xi \in \Q$,  $w\in{ \{\frac{1}{2k+1}, \frac{2}{2k+1}: \; k>j \}}$.
\item 
{Notice that this hierarchical structure of the powers of $e$ in a given row/column,  implies a similar hierarchical 
structure for the rows of its inverse, as we have already pointed out in  
\eqref{hierstruc1},  \eqref{hierstruc2}, \eqref{hierstruc3}, \eqref{hierstruc4}, \eqref{hierstruc5} and \eqref{hierstruc6}.}
\end{itemize}
}
\end{remark}

If $\cA_{q_0,m}^{\tiny (odd)}$ is {\it non-degenerate},  then solving this system of linear equations, 
we  obtain 
$$a_{2k+1}=O(e^{2(k-(m-1))+2}\|\mu\|_{C^{{3k_0}}}),\quad k=m-1,m,\dots, k_0.$$  
With this new relation, using the arguments in Lemma \ref{Fourier5}, we show that 
 assumption \eqref{assumption-i} still holds by replacing $m$ with  $m-1$. 
{This completes the proof of the inductive step}. \\

Iterating  the procedure until {$m=1$},  we conclude that
$$a_{2k+1}=O(e^2\|\mu\|_{C^{{3k_0}}}),\quad k=1,\dots, k_0-1.$$
In the same way, we may show that $$b_{2k+1}=O(e^2\|\mu\|_{C^{{3k_0}}}),\quad k=1,\dots, k_0-1.$$\\

\medskip


\subsection{Fourier coefficients of even order Fourier modes} \label{subseceven}
Let $1\leq m\leq k_0$. Denote $N_m= 3k_0+ 3\left\lfloor \frac{k_0-m}{2} \right\rfloor + \nu_{k_0,m}$, where
\[
\nu_{k_0,m}:=\begin{cases}1&\mbox{if $k_0-m$ is even}\\
2 &\mbox{if $k_0-m$ is odd}\end{cases}
\]
{and $\lfloor\cdot \rfloor$ denotes the floor function. This choice of $N_m$ will be clarified in Remark \ref{whyNm}.}\\

Assume  that for  some ${1}\leq m\leq k_0$
we have 
\begin{equation}\label{assumption-ii}
a_{2k}=\left\{\begin{array}{ll}
O(e^{2(k-m)}\|\mu\|_{C^{5k_0}}),& k=m+1,\dots, {N_{m}},\\
O(e^{2(N_m-m+1)}\|\mu\|_{C^{5k_0}}),& k= {N_m+1},\dots, {2{N_m}-m}.
\end{array}
\right.
\end{equation}

\medskip

{\begin{remark}
The above estimates hold with sharper choices of the norms $\|\cdot\|_{C^k}$ (see \eqref{fourier1} and  Lemma \ref{Fourier5}). However, for the sake of simplicity we have opted for a common choice that is suitable for all steps involved in the  algorithm (see Remark \ref{choicenorm2}).\\
\end{remark}}

Observe that Lemma \ref{Fourier5} implies  the assumption above  for $m=k_0$.\\

{We denote $N'(k):=k-m$.}\\

{ Let us fix $k\in \{k_0+1,\dots, N_m\}$; observe that for such a choice of $k$, there exists 
 an integrable rational caustic with rotation number  $\frac{1}{2k}$. Let us now apply  \eqref{fourier1} with $N=N'(k)$:
\begin{equation}\label{eq1sec7.2}
a_{2k}+ \sum_{n=1}^{N'(k)} \sum_{|l|\leq n} \frac{\xi_{n,l}(2(k-l))\, e^{2n}}{(1-\lambda_{\frac{1}{2k}}^2)^n} a_{2(k-l)}=O(e^{2N'(k)+2}\|\mu\|_{C^{N'(k)+1}}).
\end{equation}}

{\begin{remark}\label{choicenorm2}
Notice that all estimates involve $\|\mu\|_{C^{k+1}}$ and $\|\mu\|_{C^{N'(k)+1}}$, for some $k\leq N_m$ and $m\geq 1$; in particular,  
$N'(k)+1\leq N_m+1 \leq 5k_0$ (as one can easily verify, by choosing $m=1$ and estimating the corresponding expression  both for $k_0 \geq 2$ even or odd).
Hence, we can choose to bound all terms with respect to $\|\mu\|_{C^{5k_0}}$. Hereafter, in order to simplify the notation, we will neglect this term and concentrate on
the part involving powers of the eccentricity $e$.\\
\end{remark}}


{Similarly to what we have done in the odd-order case, we want to show that in \eqref{eq1sec7.2} the only terms in the sum that are not of the same order as the remainder are the ones corresponding to $l=n$.}\\

{
\begin{itemize}
\item Observe that if $0\leq l \leq n-1$, then 
$$
k-l \geq k-(N'(k)-1) = m+1
$$
and clearly $k-l \leq N_m$.
Using the inductive hypothesis, the fact that $0\leq l\leq n-1$ and \eqref{expansionlambda}, we 
get:
\begin{eqnarray*}
\frac{ a_{2(k-l)}}{(1-\lambda_{\frac{1}{2k}}^2)^n}  e^{2n} &=& O(e^{2(k-l-m)}) \cdot O(e^{2l+2}) \cdot (1+ O(e^2))\\
&=& O(e^{2(k-m)+2}) \cdot (1+ O(e^2)) \\
&=& O(e^{2N'(k)+2}).
\end{eqnarray*}
\item Let us now consider negative $l$. \\
\medskip
First observe that if $l=-N'(k)$, then clearly  
\begin{eqnarray*}
\frac{ a_{2(k+N(k))}}{(1-\lambda_{\frac{1}{2k}}^2)^{N'(k)}}  e^{2N'(k)} = O(e^{2N'(k)+2}),
\end{eqnarray*}
where we have used that 
$a_{2(k+N'(k))}=O(e^2)$, as it follows applying \eqref{fourier1} with $N=0$ (in fact, since $k+N'(k)\geq k_0$, there exists by assumption an integrable rational caustic of rotation number $\frac{1}{2(k+N'(k))}$).\\ 
\medskip
Let us now assume that $k-N_m \leq l <0$, hence $ m+1\leq k-l \leq N_m$. Using the inductive hypothesis we get:
\begin{eqnarray*}
\frac{ a_{2(k-l)}}{(1-\lambda_{\frac{1}{2k}}^2)^n}  e^{2n} &=& O(e^{2(k-l-m)}) \cdot O(e^{2}) \cdot (1+ O(e^2))\\
&=& O(e^{2(k-m)+2}) \cdot (1+ O(e^2)) \\
&=& O(e^{2N'(k)+2}).
\end{eqnarray*}
On the other hand, if $ -n\leq l < k-N_m$, then 
$$  k-l\geq N_m+1 \quad {\rm and} \quad k-l \leq k+N'(k) \leq 2N_m-m.$$  
Therefore, using the inductive hypothesis we get (we use that $n\geq -l \geq N_m-k \geq 0$):
\begin{eqnarray*}
\frac{ a_{2(k-l)}}{(1-\lambda_{\frac{1}{2k}}^2)^n}  e^{2n} &=& O(e^{2(N_m-m+1)}) \cdot O(e^{2n}) \cdot (1+ O(e^2))\\
&=&  O(e^{2(N_m-m+1)}) \cdot O(e^{2(N_m-k)}) \cdot (1+ O(e^2))\\  
&=& O(e^{2(2N_m-m-k)+2}) \cdot (1+ O(e^2))\\
&=& O(e^{2N'(k)+2}). 
\end{eqnarray*}
\end{itemize}
}

{Using these estimates, we see that \eqref{eq1sec7.2} becomes: 
\begin{eqnarray*}
a_{2k}+\sum_{j=1}^{N'(k)}\frac{\xi_{j,j}(2k-2j)e^{2j}}{(1-\lambda_{\frac{1}{2k}}^2)^j}a_{2(k-j)}=O(e^{2N'(k)+2}\|\mu\|_{C^{5k_0}}).
\end{eqnarray*}
Using Lemma \ref{rot-no1} and the inductive hypothesis, we see that for $j<N'(k)$ (which implies $  m+1\leq k-j < N_m$), we have:
\begin{eqnarray*}
\frac{a_{2(k-j)}\, e^{2j}}{(1-\lambda_{\frac{1}{2k}}^2)^j}   &=& \frac{a_{2(k-j)}\, e^{2j}}{\cos^{2j}(\frac{\pi}{2k})} (1+O(e^2))\\
&=& \frac{a_{2(k-j)}\, e^{2j}}{\cos^{2j}(\frac{\pi}{2k})} + e^{2j+2}\, O(e^{2(k-j-m)})  \\
&=& \frac{a_{2(k-j)+1}\, e^{2j}}{\cos^{2j}(\frac{\pi}{2k})} + O(e^{2N'(k)+2}).
\end{eqnarray*}
Clearly, for $j=N'(k)$
\begin{eqnarray*}
\frac{a_{2(k-N'(k))}\, e^{2N'(k)}}{(1-\lambda_{\frac{1}{2k}}^2)^{N'(k)}}   &=& 
 \frac{a_{2(k-N'(k))}\, e^{2N'(k)}}{\cos^{2N'(k)}(\frac{\pi}{2k})} + O(e^{2N'(k)+2}).
\end{eqnarray*}
Hence,  \eqref{eq1sec7.2} reduces to
\begin{equation}\label{eqeven1}
a_{2k}+\sum_{j=1}^{N'(k)}\frac{\xi_{j,j}(2k-2j)e^{2j}}{\cos^{2j}(\frac{\pi}{2k})}a_{2(k-j)}=O(e^{2N'(k)+2}\|\mu\|_{C^{5k_0}}),
\end{equation}
for $k=k_0+1,\ldots, N_m$.\\
}

{If $k\geq 3k_0+1$, since $\Omega$ is $q_0$-integrable (recall that $q_0=2k_0)$, then we also have the existence of  integrable rational caustics with rotation numbers 
$\frac{3}{2k}$. \\
In particular, let $2k\not\equiv 0\;(\text{mod.}\; 3)$, with { $k=3k_0+1,\dots, N_{m}$}. Proceeding as above, using the existence of a caustic with rotation number 
$\frac{3}{2k}$, we can conclude that:
\begin{equation}\label{eqeven2}
a_{2k}+\sum_{j=1}^{N'(k)}\frac{\xi_{j,j}(2k-2j)e^{2j}}{\cos^{2j}(\frac{3\pi}{2k})}a_{2k-2j}=O(e^{2N'(k)+2}\|\mu\|_{C^{5k_0}}).\\
\end{equation}}

{\begin{remark}\label{whyNm}
We obtain {$N_m-m+1$} linear equations in  {$N_m-m+1$} variables: $a_{2k}$ with
$k=m,\dots, N_m$.\\  
Observe, in fact, that 
the number $N_m$ was chosen in such a way  that  the number of equations is the same as the number of unknowns. Indeed:
\begin{itemize}
\item For $k= k_0+1, \ldots, 3k_0$, we obtain $2k_0$ equations.
\item For $k= 3k_0+1, \ldots, N_m$, we have {$N_m-3k_0$} values of $k$ which contribute with $2$ equations when $k$ is not a multiple of $3$, and with  only one equation otherwise.
Hence, each group $\{3j+1, 3j+2, 3(j+1)\}$ produces $5$ equations and in our case $j=k_0, \ldots, \lfloor{N_m/3}\rfloor$. 
Let us define $\alpha_m \in\{0,1, 2\}$ such that $N_m = 3\lfloor{N_m/3}\rfloor + \alpha_m$, namely $\alpha_m \equiv N_m$ (mod. $3$).
\item Hence, the number of total equations is
$$
\underbrace{2k_0}_{\tiny \mbox{$1^{\rm st}$ block}} + \underbrace{5\left(\lfloor{N_m/3}\rfloor - k_0 \right) + 2\alpha_m }_{\tiny \mbox{$2^{\rm nd}$ block}}.
$$
\item The number of unknowns that we get is $N_m-m+1$.
\end{itemize}
In conclusion, we want to choose $N_m$ such that
\begin{eqnarray}
&& \quad  2k_0  + 5\left(\lfloor{N_m/3}\rfloor - k_0 \right) + 2\alpha_m  = N_m-m+1 \nonumber\\
&\Longleftrightarrow& \quad
5 \lfloor{N_m/3}\rfloor - 3k_0 + 2\alpha_m = 3 \lfloor{N_m/3}\rfloor + \alpha_m - m +1\nonumber\\
&\Longleftrightarrow& \quad
2 \lfloor{N_m/3}\rfloor   =  3k_0 - m +1-  \alpha_m. \label{eqforNm}
\end{eqnarray}
We want to solve this equation. We  distinguish two cases according to the parity of $3k_0 - m$ (or, equivalently, of $k_0-m$):
\begin{itemize}
\item {If $k_0 - m$ is even,  \eqref{eqforNm} can have an integral solution only if $\alpha_m=1$;} in this case:
$$
\left\lfloor{\frac{N_m}{3}} \right\rfloor = \frac{3k_0-m}{2}
$$
and
$$
N_m = 3 \left\lfloor \frac{3k_0-m}{2} \right\rfloor + 1 = 3k_0 + \left\lfloor \frac{k_0-m}{2} \right\rfloor + 1.
$$
\item Similarly, if $k_0 - m$ is odd, {\eqref{eqforNm} can have an integral solution only if $\alpha_m=0$ or $2$;} 
in case $\alpha_m=2$:
{\[\left\lfloor{\frac{N_m}{3}} \right\rfloor = \frac{3k_0-m-1}{2} = \left\lfloor \frac{3k_0-m}{2} \right\rfloor \]}
and
{\[N_m = 3 \left(\left\lfloor \frac{3k_0-m}{2} \right\rfloor \right) + 2 = 3k_0+3\left\lfloor \frac{k_0-m}{2} \right\rfloor +2.\]}
{Observe that if we choose  $\alpha_m=0$, then we could get  a larger $N_m$, namely
$
3k_0+3 \left\lfloor \frac{k_0-m}{2} \right\rfloor  + 3
$.}
\medskip
\end{itemize}

Summarizing, we choose
{\[N_m:=3k_0+3 \left\lfloor \frac{k_0-m}{2} \right\rfloor +\nu_{k_0,m},\]
where 
\[\nu_{k_0,m}=\begin{cases}1 &\mbox{if $k_0-m$ is even}\\
2 &\mbox{if $k_0-m$ is odd.}\end{cases}\]}
{Observe that for $m=k_0$, we have exactly $N_{k_0} = 3k_0+1$, as needed to recover \eqref{assumption-ii} from Lemma 7.1.}
\end{remark}}

\medskip

{\begin{remark}\label{remarkmatrixAeven}
{Recall that $q_0=2k_0$. 
We denote by $\cA_{q_0,m}^{\tiny(even)}\in {\cM}_{N_m-m+1}(\R)$ the square matrix of the coefficients associated to the linear system of equations, consisting of
\begin{itemize}
\item the first $2k_0$ equations correspond to \eqref{eqeven1} for $k=k_0+1,\ldots, 3k_0$,
\item the other $N_m-2k_0$ rows correspond to the  equations \eqref{eqeven1}--\eqref{eqeven2} for $k=3k_0+1, \ldots, N_m$.
\end{itemize}
In particular,
the matrix $\cA_{q_0,m}^{\tiny(even)}$ has the following structure 
\begin{equation}\label{matrixEven}
\cA_{q_0,m}^{\tiny (even)} = 
\left(
\begin{array}{c|c|c}
* & {\mathcal L'} & {\mathcal O'}  \\
\hline
* & * & {\mathcal K'} 
\end{array}
\right),
\end{equation}
where
\begin{itemize}
\item ${\mathcal L'}$ a lower triangular $2k_0\times2k_0$ matrix with $1$'s on the diagonal;
\item ${\mathcal K'}$ is a $(N_m-2k_0)\times (N_m-3k_0+m-1)$ matrix of the form
$$
{\mathcal K'}= 
\left(
\begin{array}{ccccc}
1 & 0 & 0 & \ldots & 0 \\
1 & 0 & 0 & \ldots & 0 \\
* &1 & 0 & \ldots & 0 \\
* &1 & 0 & \ldots & 0 \\
* &* & 1 & \ldots & 0\\
\vdots & \vdots &  \vdots & \vdots &\vdots  \\
%
* &* & * & \ldots &  1\\
* &* & * & \ldots &  1
\end{array}
\right);$$
\noindent actually, the above matrix is just an example of ${\mathcal K}'$, for some choice of $k_0$ and $m$: 
 the actual form and size of this block, in fact, may vary according to the arithmetic properties  of $k_0$ and $m$;
\item ${\mathcal O'}$ is a block of zeros of size $2k_0 \times (N_m-3k_0+m-1)$;
\item observe that each row has a unit on it and the $*$ entries are a multiple of $e^{2j}$, where $j \in \N$ represents the  ``distance'' from the unit within the row;
in particular all $*$ entries are of the form ${\xi}\,{ \cos^{-2j} ( w \pi )  e^{2j}}$, where $\xi \in \Q$,  $w\in \{\frac{1}{2k}, \frac{3}{2k}: \; k>j \}$.
\item 
{Notice that this hierarchical structure of the powers of $e$ within a given row/column,  implies a similar hierarchical structure for the rows of its inverse, as we have already pointed out
before, for example in  \eqref{hierstruc1},  \eqref{hierstruc2}, \eqref{hierstruc3}, \eqref{hierstruc4}, \eqref{hierstruc5} and \eqref{hierstruc6}.}
\end{itemize}
}
\end{remark}}

If $\cA_{q_0,m}^{\tiny (even)}$ is  {\it non-degenerate}, then  solving the linear system, we get that
{$$a_{2k}=O(e^{2k-2m+2}\|\mu\|_{C^{5k_0}}),\quad k=m,\dots, k_0.$$}
Then one can show that replacing $m$ by $m-1$,  assumption \eqref{assumption-ii} continues to hold.
Therefore, iterating the procedure until {$m=1$,} we conclude that
$$a_{2k}=O(e^2\|\mu\|_{C^{5k_0}}),\quad k=2,\dots, k_0.$$

\smallskip

Similarly, one can show that 
$$
b_{2k}=O(e^2\|\mu\|_{C^{5k_0}}),\quad k=2, \dots, k_0.
$$

\bigskip
{To summarize,  the discussion in Subsections \ref{subsecodd} and \ref{subseceven} leads to the following statement.\medskip}

{
\begin{proposition}\label{basic-estimate} If all the $q_0-2$ matrices in \eqref{matrixAodd} and \eqref{matrixEven} are 
non-degenerate, then there exists $C_{q_0}>0$ depending only on $q_0$  such that 
$$|a_k|, |b_k|\leq C_{q_0}e^2\|\mu\|_{C^{{3q_0}}}, \quad k=3,\dots,q_0.$$
\end{proposition}
}

\bigskip

\begin{remark}
{Notice that the  algorithm that we have described,  can be easily implemented on a computer, 
hence  all of the above non-degeneracy conditions can be explicitely verified, via symbolic computations, 
for arbitrary $q_0$; see Sections \ref{case3}--\ref{case5} for the  cases corresponding to 
$q_0=3,4,5$.\medskip}
\end{remark}

\section{Deformed Fourier modes}\label{deformed-fourier}
Let $\Omega$ be a strictly convex domain and let $s$ denote the arc-length 
parametrization of $\partial\Omega$ {and denote by $|\partial \Omega|$ its length}. 
Let $\rho(s)$ be its radius of curvature at $s$. Observe that if $\Omega$ is $C^r$, 
then $\rho$ is $C^{r-2}$. The Lazutkin parametrization of the boundary, first introduced 
in \cite{Lazutkin}, is defined as 
$$
x(s)=C_{\Omega}\int_0^s\rho(\sigma)^{-2/3}d\sigma,\quad\text{where}
\quad C_{\Omega}:=2\pi\Big[\int_{0}^{|{\partial\Omega}|}\rho(\sigma)^{-2/3}d\sigma\Big]^{-1}.
$$
Observe that if $\partial\Omega=\cE_e$ is an ellipse, $\rho$ is analytic, thus, 
the Lazutkin parametrization is itself an analytic parametrization of $\cE_e$. 
Let $(\mu,\varphi)$ be the elliptic coordinates associated to the ellipse $\cE_e$, 
$$\cE_e=\{(\mu_0,\varphi): \; \varphi\in[0,2\pi)\}.$$
Let $\varphi_L(x)$ denote the change of parametrization from $x$ to $\varphi$. 
Then  we have the following lemma. 
\begin{lemma}\label{lazutkin-change} For each $r\in\mathbb{N}$, 
there exists $C_r$ such that 
$$\|\varphi_L(x)-x\|_{C^r}\leq C_re^2.$$
\end{lemma}
{The proof of this lemma is straightforward. The reader is kindly referred to  \cite[Appendix A]{KS} for some details.}

Now let us introduce the change of variables from the action-angle parametrization $\theta$ of 
{$ \cE_e$}, derived  from the smooth convex caustic with rotation number $1/q$,  
to the Lazutkin parametrization $x$, {\it i.e}, 
$$x=X_q(\theta):=\varphi_L^{-1}\Big(\varphi_{\lambda_{1/q}}(\theta)\Big).$$
The following lemma is proven in \cite[Lemma 11]{ADK}.
\begin{lemma}\label{lemma-q-lazutkin}There exists $C(e)$, with $C(e)\to0$ as $e\to 0^+$, such that
$$\|X_q(\cdot)-\mathbb{Id}(\cdot)\|_{C^1}\leq \frac{C(e)}{q^2}.$$
where $\mathbb{Id}$ stands for the identity map.
\end{lemma}

Let us denote $L^2(\mathbb{T})$ the $L^2$-space of $2\pi$-periodic functions, with
 trigonometric basis $\{v_k\}_{k\in\mathbb{Z}}$, where
$$v_0=1, \; v_k(x)=\frac{1}{\sqrt{\pi}}\cos k x,\quad v_{-k}=\frac{1}{\sqrt{\pi}}\sin kx ,\quad k=1,2,\dots.$$
Consider another set of functions $\{c_k\}_{k\in\mathbb{Z}}$, where
$$c_0(x)=v_0,\quad c_{k}(x)=v_{k}(x),\quad c_{-k}(x)=v_{-k}(x)\quad k=1,\dots, q_0,$$
and for $k> q_0$,
$$
c_k(x)=\frac{\cos \left(k X^{-1}_k(x)\right)}{\sqrt{\pi}\,X'_k\big(X_k^{-1}(x)\big)}\quad ,\quad 
c_{-k}(x)=\frac{\sin \left(k X^{-1}_k(x)\right)}{\sqrt{\pi}\,X'_k\big(X_k^{-1}(x)\big)}.
$$
Note here that the functions $c_{\pm k}$ have zero average.
From  Lemma \ref{lemma-q-lazutkin}, for each $k\ge 1$ we have 
\begin{equation}\label{basis-difference}\|c_k-v_k\|_{C^0}\leq \frac{C(e)}{k}
\quad , \quad 
\|c_{-k}-v_{-k}\|_{C^0}\leq \frac{C(e)}{k},
\end{equation}
and 
$$
C(e)\longrightarrow 0\quad \text{as}\quad e\to0^+;
$$
for $e$ small enough,
$\{c_k\}_{k\in\mathbb{Z}}$ form a basis of $L^2$ (see \cite[Proposition 22]{ADK} and
 Lemma \ref{basis-proof} hereafter).
 
For any integer $r\geq 1$,  we consider the Sobolev space $H^r(\mathbb{T})$, 
which is defined as
$$
H^r(\mathbb{T}):=\{ u\in L^2(\mathbb{T}): \; u^{(r)}\in L^2(\mathbb{T})\},
$$
{where $u^{(r)}$ denotes the $r$-th (weak) derivative of $u$}.
Recall that $H^r(\mathbb{T})$ is a Hilbert space with  inner product
 $$
 \langle u, v\rangle_{r}=
 \Big(\int_{\mathbb{T}}u dx\Big)\Big(\int_{\mathbb{T}}vdx\Big)+
 \int_{\mathbb{T}} u^{(r)}v^{(r)}dx,
 $$
 and we have  
 $$
 \|u\|_{r}^2=\sum_{k\in\mathbb{Z}}(|k|^{2r}\wedge1)\hat{u}_k^2=\langle u,u\rangle_r,
 $$
 where $a\wedge b=\max\{a,b\}$ and  $\hat{u}_k$ are the Fourier coefficient of $u$, {\it i.e.},
 $$\hat{u}_k=\int_{\mathbb{T}}u(x)v_k(x)dx,\quad k\in\mathbb{Z}.$$
{Notice that the choice of norms is somewhat 
non-standard and for each $r\ge 1$ we have 
$$
\|u\|_{r} \leq \|u\|_{r+1}.
$$
}
 Denote $\mathcal{V}_k(x)$ be the functions  that have zero average
 and
 $$
 \mathcal{V}^{(r)}_k(x)=v_k(x), \quad k\in\mathbb{Z}\setminus\{0\}.
 $$
 Then,  we have that the set of functions $\{\mathcal{V}_0=1, \mathcal{V}_k,\; 
 k\in\mathbb{Z}\setminus\{0\}\}$ form an orthonormal basis of $H^r(\mathbb{T})$, {\it i.e.},
 $$\langle \mathcal{V}_k,\mathcal{V}_j\rangle_r=\delta_{k,j}, \quad \forall\; k,\;j\in\mathbb{Z},$$
 and for every $u\in H^{r}(\mathbb{T})$, we have
 $$u(x)=\sum_{k\in\mathbb{Z}} u_k \mathcal{V}_k(x),$$
 and 
 $$\|u\|_r^2=\sum_{k\in\mathbb{Z}}u_k^2,$$
 where $u_k=\langle u,\mathcal{V}_k\rangle_r.$ 
 Observe that $u_k^2=(k^{2r}\wedge1)\hat{u}_k^2$, for $k\in\mathbb{Z}$.
 
 Now we introduce a set of functions
 $$
 \{\cC_0=1, \cC_k,\; \cC_{-k}, k\in\mathbb{Z}_+\},
 $$ 
 where $\cC_{\pm k}$, $k\in\mathbb{Z}_+$ are the zero average functions  on $\mathbb{T}$ such that
 $$
 \cC^{(r)}_k(x)=c_k(x),\quad \cC^{(r)}_{-k}(x)=c_{-k}(x),\quad k\in\mathbb{Z}\setminus\{0\}.
 $$
 Therefore, we have
 $$
 \cC_k=\mathcal{V}_k\quad ,\quad \cC_{-k}=\mathcal{V}_{-k},\quad k\in\mathbb{Z}_+, \ k\leq q_0,
 $$
 and 
 using \eqref{basis-difference}{
\begin{eqnarray}\label{basis1}
 \|\mathcal{V}_k-\mathcal{C}_k\|_r^2 &=&
 \langle \mathcal{V}_k-\cC_k, \mathcal{V}_k-\mathcal{C}_k\rangle_r
 \leq \frac{[C(e)]^2}{k^2}, \nonumber
\\
 \|\mathcal{V}_{-k}-\mathcal{C}_{-k}\|_r^2 &=&
 \langle \mathcal{V}_{-k}-\cC_{-k}, \mathcal{V}_{-k}-\mathcal{C}_{-k}\rangle_r
 \leq \frac{[C(e)]^2}{k^2} 
 \end{eqnarray}
 for $k\in\mathbb{Z}$, $k>q_0$}.
 Consider the linear operator 
 $$
 \cL: H^r(\mathbb{T})\to H^r(\mathbb{T}),\; \ \ 
 u\mapsto \cL u=u_0+\sum_{k\in\mathbb{Z}_+}u_k\cC_k(x)+u_{-k}\cC_{-k}(x),
 $$
 where $u_k=\langle u,\mathcal{V}_k\rangle_r$.
 
Define $D(q_0):=
\Big[\sum_{|k|>q_0,k\in\mathbb{Z}}
\frac{1}{k^2}\Big]^{\frac{1}{2}}<\sqrt{\frac{\pi^2}{3}}.$
 \begin{lemma}\label{basis-proof} 
 {Let $C(e)$ be from Lemma \ref{lemma-q-lazutkin}.} Assume   $e_0$  satisfies
 $$
 C(e_0)D(q_0)<1.
 $$ 
 Then, for each  $e\in[0,e_0]$ 
 the operator $\cL$ is bounded and invertible in the Hilbert space 
 $H^r(\mathbb{T})$. In particular, $\{\cC_0,\;\cC_k,\;\cC_{-k}, k\in\mathbb{Z}_+\}$ 
 form a basis of $H^r(\mathbb{T})$.
 \end{lemma}
 
 \begin{proof} 
 Observe that if $\|\cL-\mathbb{Id}\|_{H^r\to H^r}<1$, then $\cL$ is a bounded 
 invertible operator with a bounded  inverse; {recall that
\begin{equation}\label{defnormHrtoHr}
\|\cL-\mathbb{Id}\|_{H^r\to H^r}:=\sup_{\|u\|_r\leq 1}\|[\cL-\mathbb{Id}](u)\|_r.
\end{equation}
 }
 For each $v\in H^r$, we have
 $$u=\sum_{k\in\mathbb{Z}}u_k \mathcal{V}_k,\quad u_k=
 \langle u_k, \mathcal{V}_k\rangle_r,\quad k\in\mathbb{Z}.$$
 By the definition of the operator $\cL$, we have 
 $$[\cL-\mathbb{Id}](u)=\sum_{k\in\mathbb{Z}, |k|>q_0}u_k(\cC_k-\mathcal{V}_k).$$
 By the Cauchy inequality, we have
 \[\|[\cL-\mathbb{Id}](u)\|_r\leq 
 \sum_{k\in\mathbb{Z}, |k|>q_0}|u_k|\,\cdot\,\|\cC_k-\mathcal{V}_k\|_r\leq \left[\sum_{k\in\mathbb{Z}}u_k^2\right]^{\frac{1}{2}}\left[\sum_{k\in\mathbb{Z}}\|\cC_k-\mathcal{V}_k\|_r^2\right]^{\frac{1}{2}}.\]
 By \eqref{basis1}, we have
 \[\left[\sum_{k\in\mathbb{Z}}\|\cC_k-\mathcal{V}_k\|_r^2\right]^{\frac{1}{2}}\leq C(e)\left[\sum_{|k|>q_0, k\in\mathbb{Z}}\frac{1}{k^2}\right]^{\frac{1}{2}}<C(e)\sqrt{\frac{\pi^2}{3}}.\]
 Therefore, the assertion of the lemma follows from {\eqref{defnormHrtoHr}}
 and the fact that $\|u\|^2_r=\sum_{k\in\mathbb{Z}}u_k^2.$
 \end{proof}
 
 \medskip
 
 \begin{remark} {Observe that the  basis $\{\cC_0,\;\cC_k,\;\cC_{-k}, k\in\mathbb{Z}_+\}$ 
  of $H^r(\mathbb{T})$ is not necessarily an orthogonal basis. }
 \end{remark}

 \medskip
 
 \begin{corollary}\label{key-corollary1}
 There exists $C'(e)>0$, with $C'(e)\to1$ as $e\to0^+$, such that for each $u\in H^r(\mathbb{T})$,
$$\|u\|_r^2\leq C'(e)\sum_{k\in\mathbb{Z}}\tilde{u}_k^2,$$
where $\tilde{u}_k=\langle u,\cC_k\rangle_r$.
 \end{corollary}
 \begin{proof}The operator $\cL$ is bounded and invertible with a bounded inverse, 
 so it is its adjoint operator $\cL^*$. Let us denote 
 $$
 C'(e)=\|(\cL^*)^{-1}\|_{H^r\to H^r}.
 $$  
 Hence we have that for each $u\in H^r(\mathbb{T})$,
 \[\begin{split}\|u\|_{r}^2&=\|(\cL^*)^{-1}\cL^*u\|_r\leq C'(e)\|\cL^*u\|_r^2\\
 &\leq C'(e)\sum_{k\in\mathbb{Z}}\langle \cL^*u,\mathcal{V}_k\rangle_r^2=
 C'(e)\sum_{k\in\mathbb{Z}}\langle u,\cL\mathcal{ V}_k\rangle_r^2.
 \end{split}\]
 Since $\cL \mathcal{V}_k=\cC_k$, we have 
 $$\|u\|_{C^r}\leq C'(e)\sum_{k\in\mathbb{Z}}\tilde{u}_k^2.$$
 The assertion that $C'(e)\to1$ as $e\to0^+$ follows from the fact that 
 $\|\cL-\mathbb{Id}\|_{H^r\to H^r}\to0$ as $e\to0^+$.
 \end{proof}
 
 \medskip
 
 \begin{corollary}\label{high-basis}
 Let $u(x)\in H^{r+1}(\mathbb{T})$. Then, there exists $C''(e)>0$ such that
 $$\big|\langle u, \cC_k\rangle_r\big|\leq \frac{C''(e)\|u\|_{r+1}}{|k|} \quad \forall\; k\in\mathbb{Z}\setminus\{0\}.$$
 \end{corollary}
 \begin{proof}
 Using \eqref{basis1}, we have
 $$\big|\langle u,\mathcal{V}_k-\cC_k\rangle_r\big|\leq \|u\|_{r}\|\mathcal{V}_k-\cC_k\|_r\leq \frac{C(e)\|u\|_{r}}{|k|}.$$
 Since $u\in H^{r+1}$, we have
 $$\big|\langle u, \mathcal{V}_k\rangle_r\big|\leq \frac{\|u\|_{r+1}}{|k|}.$$
 Therefore we have
 $$\big|\langle u, \cC_k\rangle_r\big|\leq \frac{C''(e)\|u\|_{r+1}}{|k|},$$
 where $C''(e)=1+C(e)$.
 \end{proof}

 \bigskip
 
 Consider a domain $\Omega$, whose boundary $\partial\Omega$ is close to the ellipse $\cE_e$, written in  elliptic coordinates associated to $\cE_e$ as
 $$\partial\Omega=\cE_e+\mu(\varphi),$$ 
 where $\|\mu\|_{C^m}\leq M$ with $m>r+2$ and $\|\mu\|_{C^1}$ is small enough. 
 Let $x$ denote the Lazutkin parametrization of  $\cE_e$. Define
 $$f_{\mu}(x)=\mu(\varphi_L(x)).$$
 
Then  we have:

 \begin{lemma}\label{change-lazutkin-1} For any integer $r>0$, there exists $C_r>0$ independent of  {$\f$} and $\mu$, such that
$$ (1-C_re^2)\|\mu\|_{C^r}\leq \|f_{\mu}\|_{C^r}\leq (1+C_re^2)\|\mu\|_{C^r}.$$
 \end{lemma}
 
{Moreover, the following holds.}
 
 \begin{lemma} \label{change-lazutkin-2}There exists $C>0$ such that
$$ |\hat{f}_k-\hat{\mu}_k|\leq Ce^2\|\mu\|_{C^1},$$
 where $\hat{f}_k$ and $\hat{\mu}_k$ are the Fourier coefficients of the functions $f_{\mu}$ and $\mu$.
 \end{lemma}

 The two lemmata above directly follow from Lemma \ref{lazutkin-change}.\medskip
 
{Let us now show the following result, {where we assume $0< e\leq e_0$ and $C_re^2\leq \frac{1}{2}$.}}
 
 \begin{lemma}\label{key-lemma1}
For any integer $q>q_0$, if the billiard dynamics inside the domain $\Omega$ 
admits {an integrable rational caustic} with rotation number $1/q$, then
 $$\big|\langle f_{\mu},\cC_{\pm q}\rangle_r\big|\leq C(M)q^7\|f_{\mu}\|_{C^1}^{\frac{2(m-r-2)}{m-1}}.$$
 \end{lemma}
 
 \medskip
 
 \begin{proof}By Lemmata  \ref{rot-no1} and \ref{rational-condition}, 
 from  the existence of a smooth convex caustic with rotation number $1/q$,  we have that
$$ \sum_{k=1}^qf_{\mu}(X_q(\theta+\frac{k}{q}2\pi))=c_{1/q}+\Upsilon(X_q(\theta)),$$
and denoting 
{
$\widetilde \Upsilon=\Upsilon(X_q(\theta))$, 
$$
\|\widetilde  \Upsilon\|_{C^0}\leq q^8C\|f_{\mu}\|_{C^1}^2,\quad \text{and}\quad 
\|\widetilde  \Upsilon\|_{C^{m-1}}\leq q^2C'(M).
$$}
By the Sobolev interpolating inequality
$$\|u\|_{C^r}\leq C\|u\|_{H^{r+1}}\leq C\|u\|_{C^{m-1}}^{\frac{r+1}{m-1}}\|u\|_{C^0}^{\frac{m-r-2}{m-1}},$$
we have
$$\|\widetilde \Upsilon\|_{C^r}\leq q^8C'(M)\|f_{\mu}\|_{C^1}^{\frac{2(m-r-2)}{m-1}}.$$
Notice that
\[\begin{split}&\int_0^{2\pi}D^r\sum_{k=1}^qf_{\mu}(X_q(\theta+\frac{k}{q}2\pi))\sin q\theta\, d\theta\\
&=\sum_{k=1}^q\int_0^{2\pi}D^rf_{\mu}(X_q(\theta+\frac{k}{q}2\pi))\sin q\theta\, d\theta\\
&=q\int_{0}^{2\pi}D^rf_{\mu}(X_q(\theta))\sin q\theta \,d\theta,\end{split}\]
{here we denote $D^r$  for the $r$-th  derivative.}
Then
$$\Big|\int_{0}^{2\pi}D^rf_{\mu}(X_q(\theta))\sin q\theta \,d\theta\Big|\leq \frac{\|{\widetilde \Upsilon}\|_{C^r}}{q}\leq q^{7}C'(M)\|f_{\mu}\|_{C^1}^{\frac{2(m-r-2)}{m-1}}.$$
Let $x=X_q(\theta)$ and $\theta=X_q^{-1}(x)$, we have 
$$d\theta=\frac{1}{X_{q}'(X_{q}^{-1}(x))}dx.$$
Therefore,
\[\begin{split}&\int_0^{2\pi}D^rf_{\mu}(X_q(\theta))\sin q\theta \,d\theta=\int_{0}^{2\pi}D^rf_{\mu}(x)\sin qX_{q}^{-1}(x)\frac{1}{X'(X_q^{-1}(x))}dx\\
&=\sqrt{\pi}\int_{0}^{2\pi}D^rf_{\mu}(x)D^r\cC_{-q}(x)dx=\sqrt{\pi}\langle f_{\mu},\cC_{-q}\rangle_r.\end{split}\]
Hence 
$$
\big|\langle f_{\mu},\cC_{-q}\rangle_r\big|\leq q^{7}C'(M)\|f_{\mu}\|_{C^1}^{\frac{2(m-r-2)}{m-1}}.
$$
{Repeating a similar argument, we obtain the corresponding inequality for $\langle f_{\mu},\cC_{q}\rangle_r$.} 
 \end{proof}

\section{Proof of The Main result}\label{main-proof}
In this section, we prove  Theorems \ref{main-thm} and \ref{main-thm2}.

Denote $$n=3q_0,\quad \text{and}\quad  m=40q_0.$$  Let $\cE_e$ be an ellipse with eccentricity $e\in(0, 4e_0/5]$ and the semi-major axis $1$, 
where $e_0$ is from Lemma \ref{basis-proof}.  
Consider a $C^m$-smooth domain $\Omega$, which is a $C^n$-perturbation of the ellipse $\cE_e$, {\it i.e.},
in the elliptic coordinates associated to $\cE_e$,
$$\partial\Omega=\cE_e+\mu(\varphi),$$
where $$\|\mu\|_{C^n}\leq \varepsilon,\quad \text{and}\quad \|\mu\|_{C^m}\leq M.$$
Here $\varepsilon\leq e^{6q_0}$ is a small parameter to be determined below and $M>0$ 
is a fixed constant.  We make the following assumption:

\smallskip

{\it Assumption A: The domain $\Omega$ is $q_0$-rationally  integrable and the non-degeneracy conditions 
in Proposition~\ref{basic-estimate} \ hold true if $q_0\geq 6$. More exactly, matrices \eqref{matrixAodd} and 
\eqref{matrixEven} are non-degenerate.}

\smallskip

 {The proof consists of two main steps: 
\begin{itemize}
\item Find an ellipse $\mathcal E''$, close to $\mathcal E_e$, which  best approximates $\Omega$.
\item Show that $\Omega=\mathcal E''$. 
\end{itemize}}
 
\smallskip

{\bf Step 1} Denote $\mathbb E_\e=\mathbb E_\e(\cE_e)$ the set of ellipses whose $C^0$-Hausdorff 
distance to $\cE_e$ is not greater than $2\varepsilon$, {\it i.e.},
$$\mathbb E_\e:=\{\cE'\subset \mathbb{R}^2:\; \text{dist}_{H} (\cE',\;\cE_e)\leq 2\varepsilon\}.$$
Clearly, $\mathbb E_\e$ is a compact set in any $C^r$-topology 
(it  is completely determined by 5 parameters). 
We  choose $\varepsilon$ small enough so that the eccentricities of all  ellipses in $\mathbb E_\e$  
are between $4e/5$ and $5e/4$. For each $\cE'\in \mathbb E_\e$, we can write 
the domain $\Omega$ in the elliptic-coordinate frame associated to $\cE'$, as
$$
\partial\Omega=\cE'+\mu_{\cE'}(\varphi).
$$
Choosing a smaller $\varepsilon$ if necessary,  assuming 
$\|\mu_{\cE'}\|_{C^m}\leq 2M$, 
\mbox{$\forall \cE'\in \mathbb E_\e$},
from Lemma \ref{change-coordinate}, we know that 
$\|\mu_{\cE'}\|_{C^n}$ changes continuously 
with respect to $\cE'$. \medskip

{The proof is by contradiction. Assume that the statement 
of the theorem is not true -- {namely}, $\partial\Omega$ is 
not an ellipse -- since $\mathbb E_\e$ is compact, then we choose $\cE ''\in \mathbb E$ such that 
\[
\|\mu_{\cE''}\|_{C^n}=
\min \,\{\cE '\in \mathbb E:\ \|\mu_{\cE'}\|_{C^n}\}>0.
\] }
We also have that
$$
\|\mu_{\cE''}\|_{C^m}\leq 2M\quad \text{ and }\quad 
\|\mu_{\cE''}\|_{C^n}\leq \|\mu_{\cE'}\|_{C^n}.
$$
\

{{\bf Step 2.} {We prove the following: 
\blm 
There exists an ellipse 
$\bar{\cE}\in\mathbb E_\e$ such that in the elliptic-coordinate 
frame associated to $\bar{\cE}$
$$
\|\mu_{\bar{\cE}}\|_{C^n}<\frac{1}{2}\|\mu_{\cE''}\|_{C^n}.
$$
\elm

Notice that this contradicts minimality of 
$\|\mu_{\cE''}\|_{C^n}> 0$ among all $\cE' \in\mathbb E_\e.$ }\\

\begin{proof}
By Lemma \ref{ellipse}, there exists an ellipse 
$\bar{\cE}\in\mathbb E_\e$ such that in the elliptic-coordinate frame associated to $\bar{\cE}$, the domain 
$\Omega$ reads as
$$
\partial\Omega=\bar{\cE}+\mu_{\bar{\cE}}(\varphi),
$$
with
$$
\|\mu_{\bar{\cE}}\|_{C^m}\leq 2M,\quad 
\|\mu_{\bar{\cE}}\|_{C^n}\leq 2\|\mu_{\cE''}\|_{C^n}.$$
and the first five Fourier coefficients of $\mu_{\bar{\cE}}$ 
satisfy \eqref{first-5}. Write $\mu_{\bar{\cE}}$ as Fourier series, {\it i.e.},
$$
\mu_{\bar{\cE}}(\varphi):=\sum_{k=0}^{+\infty}a_{k}\cos (k\varphi) +b_k\sin (k\varphi).
$$
We split the perturbation into four parts: 
\begin{enumerate}
\item (Elliptic motions) $|k|{\leq} 2$;
\item (Low-order modes) $2<|k|\leq q_0$;
\item (Intermediate-order modes) $q_0<|k|<N:=\|\mu_{\cE''}\|_{C^n}^{-1/15}$;
\item (High-order modes) $|k|\geq N$.\medskip
\end{enumerate}
Each of these regimes requires different type of estimates. 
\medskip 

{\it $\bullet$ Elliptic motions: $|k|{\leq} 2$.}

\medskip 

By Lemma \ref{ellipse}, there exists $C>0$ such that 
\begin{equation}\label{first-5}
|a_k|\leq Ce^2\|\mu_{\cE''}\|_{C^1}, \quad |b_k|\leq Ce^2\|\mu_{\cE''}\|_{C^1}, \quad |k|\leq 2.
\end{equation}

\medskip 

{\it $\bullet$ Low-order modes: $2<|k|\leq q_0$.}

\medskip 

{With  Assumption A, from Proposition \ref{basic-estimate}   we have that there exists 
$C_{q_0}>0$ depending only on $q_0$ such  that 
\begin{equation}\label{lower-q0-mode}|a_k|,\;|b_k|\leq 
C_{q_0}e^2\|\mu_{\bar{\cE}}\|_{C^n}\leq 2C_{q_0}e^2\|\mu_{\cE''}\|_{C^n}, \quad 
3\leq k\leq q_0.\end{equation}}
Denote by $x$ the Lazutkin parametrization of the ellipse $\bar{\cE}$. Define 
$$
F(x):=\mu_{\bar{\cE}}(\varphi(x)).
$$ 
By Lemma \ref{change-lazutkin-1}, we have
{$$\|\mu_{\bar{\cE}}\|_{C^n}\leq (1-C_ne^2)^{-1}\|F\|_{C^n}$$
and 
$$
\|F\|_{C^n}\leq (1+C_ne^2)\|\mu_{\bar\cE}\|_{C^n}\leq 2(1+C_ne^2)\|\mu_{\cE''}\|_{C^n}.
$$ }

We consider the Hilbert space $H^{n+1}(\mathbb{T})$ and define the basis 
\mbox{$\{\cC_k,k\in\mathbb{Z}\}$} for  $H^{n+1}(\mathbb{T})$ like the one defined  in 
Section \ref{deformed-fourier}.  
Denote 
$$\alpha_k^F=\langle F, \cC_k\rangle_{n+1}\quad k\in\mathbb{Z}.
$$ 
Then, due to Lemma \ref{change-lazutkin-2}, \eqref{first-5} and \eqref{lower-q0-mode}, we have that there exists $\bar{C}_{q_0}>0$ such that
$$|\alpha^F_k|\leq(1\wedge|k|^{n+1})\bar C_{q_0}e^2\|\mu_{\cE''}\|_{C^n} \quad |k|\leq q_0.
$$
Therefore, we have
$$\sum_{k=-q_0}^{q_0}(\alpha_k^F)^2\leq 3\bar C_{q_0}q_0^{2n+3}e^4\|\mu_{\cE''}\|^2_{C^n}.
$$

\medskip 

{\it $\bullet$ Intermediate-order modes: $q_0<|k|<N:=\|\mu_{\cE''}\|_{C^n}^{-1/15}$.}

\medskip

By Lemma \ref{key-lemma1}, for $q_0<|k|<N$ we have that
$$|\alpha_k^F|\leq C(M)|k|^7\|F\|_{C^1}^{\frac{2(m-n-3)}{m-1}}\leq 
4C(M)|k|^7\|\mu_{\cE''}\|_{C^{1}}^{\frac{2(m-n-3)}{m-1}}.$$
So we have,
$$\sum_{q_0<|k|<N}(\alpha_k^F)^2\leq C(M)N^{15}\|\mu_{\cE''}\|_{C^{1}}^{\frac{4(m-n-3)}{m-1}}.$$

\medskip 

{\it $\bullet$  High-order modes: $|k|\geq N$.}

\medskip

For $|k|\geq N$, due to Lemma \ref{high-basis}, we have
$$|\alpha_k^F|\leq \frac{\|F\|_{C^{n+2}}}{|k|}\leq \frac{C\|\mu_{\cE''}\|_{C^{n+2}}}{|k|}.$$
So we have 
$$\sum_{|k|\geq N}(\alpha_k^F)^2\leq 
\sum_{|k|\geq N}C\frac{\|\mu_{\cE''}\|^{2}_{C^{n+2}}}{k^2}\leq \frac{C}{N}\|\mu_{\cE''}\|_{C^{n+2}}^2.$$
Using  Sobolev interpolation inequality, we have
$$
\|\mu_{\cE''}\|_{C^{n+2}}\leq C\|\mu_{\cE''}\|_{H^{n+3}}\leq 
C\|\mu_{\cE''}\|_{C^m}^{\frac{3}{m-n}}\|\mu_{\cE''}\|_{C^{n}}^{\frac{m-n-3}{m-n}}\leq 
C(M)\|\mu_{\cE''}\|_{C^{n}}^{\frac{m-n-3}{m-n}}.$$
 Since $m=40q_0$ and $n=3q_0$, we have 
 $$
 \frac{2(m-n-3)}{m-n}\geq\frac{72}{37}\quad 
 \text{and}\quad \frac{4(m-n-3)}{m-1}\geq\frac{18}{5}.
 $$ 
 Choose $N=\|\mu_{\cE''}\|_{C^n}^{-1/15}$. Then, we obtain
$$\frac{1}{N}C\|\mu_{\cE''}\|_{C^{n+2}}^2\leq 
CM\|\mu_{\cE''}\|_{C^n}^{72/37+1/15}=C(M)\|\mu_{\cE''}\|_{C^n}^{2+7/555}$$
and
$$C(M)N^{15}\|\mu_{\cE''}\|_{C^{n}}^{\frac{4(m-n-3)}{m-1}}\leq C(M)\|\mu_{\cE''}\|_{C^n}^{13/5}.$$
Then, due to Corollary \ref{key-corollary1}, we conclude
$$
\|F\|^2_{H^{n+1}}\leq 
C'(e)\big(3\bar C_{q_0}q_0^{2n+3}e^4\|\mu_{\cE''}\|_{C^n}^2+C(M)\|\mu_{\cE''}\|_{C^n}^{2+7/555}\big).$$
By  Sobolev embedding theorem, we have 
\[\begin{split}\|\mu_{\bar{\cE}}\|^2_{C^n}&\leq \frac{1}{(1-C_ne^2)^{-2}}\|F\|^2_{C^n} 
\leq \frac{C'_n}{(1-C_ne^2)^{-2}}\|F\|^2_{n+1}\\
&\leq \frac{3C'_nC'(e)\bar C_{q_0}}{(1-C_ne^2)^{-2}}q_0^{2n+3}e^4\|\mu_{\cE''}\|_{C^n}^2+
C(M)\|\mu_{\cE''}\|_{C^n}^{2+7/555}.\end{split}\]
Hence, if 
\begin{equation}
\frac{3C'_nC'(e)\bar C_{q_0}}{(1-C_ne^2)^{-2}}q_0^{2n+3}e^4<\frac{1}{16}
\end{equation} 
and $\varepsilon$ is small enough,  we get
$$
\|\mu_{\bar{\cE}}\|_{C^n}<\frac{1}{2}\|\mu_{\cE''}\|_{C^n},
$$
which contradicts the minimality of $\|\mu_{\cE''}\|_{C^n}.$
So $\partial\Omega$ must be an ellipse.
\end{proof}

 \appendix
 \section{Elliptic Polar Coordinates}
Consider an ellipse 
\[\
\cE=\Big\{(x,y)\in\mathbb{R}^2:\;\frac{x^2}{a^2}+\frac{y^2}{b^2}=1\Big\},\quad a>b>0.
\]
Associated to $\cE$, there exists an elliptic-coordinate frame $(\mu,\varphi)$ given by the relations
\[
\left\{\begin{array}{l}
x=c\,\cosh \mu\,\cos\varphi \\\ y=c\,\sinh \mu\,\sin\varphi,
\end{array}\right.
\]
where $c=\sqrt{a^2-b^2}$ is the semi-focal distance of $\cE$. 

Let $e$ denote the eccentricity and $\mu_0:=\cosh^{-1}(e^{-1})$; 
then,  $\cE$ in this elliptic-coordinate frame is  represented by
$$
\cE=\{(\mu_0,\varphi):\;\varphi\in[0,2\pi)\}.
$$
Hence, any (small) smooth perturbation $\partial \Omega$ of the ellipse $\cE$ can be written in elliptic-coordinate frame as
$$
\partial\Omega=\{(\mu_0+\mu(\varphi),\varphi):\; \varphi\in[0,2\pi)\},
$$
where $\mu(\varphi)$ is a  $2\pi$-periodic smooth functions;
hereafter we will use the shorthand
$$\partial\Omega=\cE+\mu(\varphi).
$$

\medskip

\begin{lemma}\label{change-coordinate}\cite[Lemma 35]{KS}
Let  $\E_{e_0,c}$ be an ellipse of eccentricity $e_0=1/\cosh \mu_0$ and semi-focal distance $c$, and
suppose that $\partial \Omega$ is a  perturbation of $\E_{e_0,c}$,  which can be written (in the elliptic-coordinate  frame $(\mu,\f)$ associated to $\E_{e_0,c}$) as $\partial\Omega = \E_{e_0,c} + \mu_\Omega(\f).
$
Consider another ellipse $\overline{\E}$ sufficiently close to $\E_{e_0,c}$, which can be written (in elliptic-coordinates frame associated to $\E_{e_0,c}$)  as
$$
\overline{\E} = \E_{e_0,c} + \mu_{\overline{\E}}.
$$
If $\overline{\E}$ is sufficiently close to $\E_{e_0,c}$, we can write (in the elliptic-coordinate frame $(\overline{\mu},\overline{\f})$ associated to $\overline{\E}$)
$\partial \Omega = \overline{\E} + \overline{\mu}_\Omega(\overline{\f})$, for some function $\overline{\mu}_\Omega$. Then, there exists $C=C(e_0,c,n)$ such that 
\begin{equation}\label{stima1}
\|\mu_\Omega (\f) - (\mu_{\overline{\E}}(\f) + \overline{\mu}_\Omega(\f) )\|_{C^n} \leq C \|\mu_{\overline{\E}}\|_{C^n} \| \| \mu_\Omega - \mu_{\overline{\E}}\|_{C^n}.
\end{equation}
In particular,  for any $C'>1$, if $\overline{\E}$ is sufficiently close to $\E_{e_0,c}$ then we have
\begin{equation}\label{stima2}
\frac{1}{C'} \|\mu_\Omega - \mu_{\overline{\E}}\|_{C^n} \leq \|\overline{\mu}_{\Omega}\|_{C^n} \leq
{C'} \|\mu_\Omega - \mu_{\overline{\E}}\|_{C^n}.
\end{equation}
\end{lemma}

{\begin{remark} Lemma 35 in \cite{KS} is stated for $C^1$-norm. The same arguments also work for $C^n$-norm, $n=1,\dots,m$.\\
\end{remark}}

 \section{Elliptic Motions in elliptic coordinates}
 In this section  we consider a special class of  perturbations of the ellipse $\cE_{e,c}$ (see also \cite[Appendix B]{KS}).
These perturbations  written in the corresponding elliptic coordinates are of the form
 $$
 \partial \Omega=\cE_{e,c}+\tilde\mu(\varphi),
 $$
 with 
 $$
 \tilde\mu(\varphi)=a_0+a_1\cos\varphi+a_{-1}\sin\varphi+a_2\cos2\varphi+a_{-2}\sin2\varphi.
 $$
 We show that for this type of perturbations, there exists an ellipse $\bar{\cE}$, represented 
 in  elliptic coordinates as $ \bar{\cE}=\cE_{e,c}+\bar{\mu}(\varphi)$ such that
 $$\tilde\mu(\varphi)-\bar{\mu}(\varphi)=O(e^2\tilde{\mu}).$$

 Let us consider a domain $\cD\subset\mathbb{R}^2$ close to $\cE_{e,c}$, 
 \[\partial\cD:\quad 
 \begin{cases}x=c\;\cosh\big(\mu_0+\mu(\varphi)\big)\cos\varphi,\\
 y=c\;\sinh \big(\mu_0+\mu(\varphi)\big)\sin\varphi,
 \end{cases}\quad \varphi\in[0,2\pi],
\]
where $\mu_0 = \cosh^{-1}(1/e)$,
$\mu(\varphi)$ is a smooth $2\pi$-periodic function and we assume that $\|\mu\|_{C^1}$ is small enough.

 Let us define 
  \[\begin{split}r_{\mu}(\varphi)&:=(c\cosh(\mu_0+\mu(\varphi))\cos\varphi)^2+(c\sinh(\mu_0+\mu(\varphi))\sin\varphi)^2\\
 &=(a\cos\varphi+ a\sqrt{1-e^2}\mu(\varphi)\cos\varphi+O(\mu^2))^2\\
 &\quad+(a\sqrt{1-e^2}\sin\varphi+ a\mu(\varphi)\sin\varphi+O(\mu^2))^2\\
 &=a^2\cos^2\varphi+2a^2\sqrt{1-e^2}\mu(\varphi)\cos^2\varphi+O(\mu^2)\\
 &\quad+a^2(1-e^2)\sin^2\varphi+2a^2\sqrt{1-e^2}\mu(\varphi)\sin^2\varphi+O(\mu^2)\\
 &=a^2(1-e^2\sin^2\varphi)+2a^2\sqrt{1-e^2}\mu(\varphi)+O(\mu^2).
 \end{split}\]
 Here we have used Taylor's expansion and the fact that 
$$c\cosh\mu_0=a,\quad c\sinh\mu_0=b=a\sqrt{1-e^2}.$$
 \subsection{Homotheties}
For any   $\lambda\in\mathbb{R}$,  let us denote the homothety of the ellipse $\cE_{e,c}$ by
$$\cE[\lambda,\cE_{e,c}]:=\exp[\lambda] \cE_{e,c}.$$
Let $\mu_{\lambda}(\varphi)$ be the function representing $\cE[\lambda,\cE_{e,c}]$  in the elliptic-coordinate frame associated to $\cE_{e,c}$. Then we have
\[\left(\begin{array}{c}c\; \cosh\big(\mu_0+\mu_{\lambda}(\varphi)\big)\cos\varphi \\
c\;\sinh\big(\mu_0+\mu_{\lambda}(\varphi)\big)\sin\varphi\end{array}\right)=\exp[\lambda]\left(\begin{array}{c}c\;\cosh(\mu_0)\cos\big(\varphi_{\lambda}(\varphi)\big) \\ c\;\sinh(\mu_0)\sin\big(\varphi_{\lambda}(\varphi))\end{array}\right),
 \]
 where $\|\varphi_{\lambda}(\varphi)-\varphi\|_{C^n}\leq C_n\lambda.$
 For $|\lambda|$ small enough, using Taylor's expansion, denoting 
 $\Delta \varphi_{\lambda}:=\varphi_{\lambda}-\varphi$, we have
 \[\begin{split}r_{\lambda}(\varphi):&=(\exp[\lambda]a\cos(\varphi_{\lambda}))^2+(\exp[\lambda]a\sqrt{1-e^2}\sin\varphi_{\lambda})^2\\
 &=a^2\big(\cos\varphi-\Delta\varphi_{\lambda}\sin\varphi+\lambda \cos\varphi+O(\lambda^2)\big)^2\\
 &\quad+a^2(1-e^2) \big(\sin\varphi+\Delta\varphi_{\lambda}\cos\varphi+\lambda\sin\varphi+O(\lambda^2)\big)^2\\
 &=a^2(\cos^2\varphi-2\Delta\varphi_{\lambda}\sin\varphi\cos\varphi+2\lambda\cos^2\varphi+O(\lambda^2))\\
 &\quad +a^2(1-e^2)(\sin^2\varphi+2\Delta\varphi_{\lambda}\sin\varphi\cos\varphi+2\lambda\sin^2\varphi+O(\lambda^2))\\
&=a^2[1-e^2\sin^2\varphi+2\lambda-2\lambda e^2\sin^2\varphi-e^2\Delta\varphi_{\lambda}\sin2\varphi+O(\lambda^2)].
 \end{split}\]
 From $r_{\mu_{\lambda}}(\varphi)=r_{\lambda}(\varphi)$, we get
 \[\begin{split}\mu_{\lambda}(\varphi)=\frac{\lambda}{\sqrt{1-e^2}}-\frac{2\lambda e^2\sin^2\varphi-e^2\Delta\varphi\sin2\varphi}{2\sqrt{1-e^2}}+O(\lambda^2)
 =\lambda+O(e^2\lambda).\end{split}\]
 \subsection{Translations} For any $\alpha=(\alpha_1,\alpha_2)\in\mathbb{R}^2$, 
let us denote the translation of $\cE_{e,c}$ in the direction of the vector $\alpha$ by 
 $$\mathcal{T}[\alpha,\cE_{e,c}]:=\cE_{e,c}+\alpha.$$
 We look for the function $\mu_{\alpha}(\varphi)$  that defines $\mathcal{T}[\alpha,\cE_{e,c}]$ 
 in the elliptic coordinates of the ellipse $\cE_{e,c}$. Then, we have
 \[\left(\begin{array}{c}c\; \cosh\big(\mu_0+\mu_{\alpha}(\varphi)\big)\cos\varphi \\
c\;\sinh\big(\mu_0+\mu_{\alpha}(\varphi)\big)\sin\varphi\end{array}\right)=\left(\begin{array}{c}c\;\cosh(\mu_0)\cos\big(\varphi_{\alpha}(\varphi)\big)+\alpha_1 \\ c\;\sinh(\mu_0)\sin\big(\varphi_{\alpha}(\varphi))+\alpha_2\end{array}\right),
 \]
 where $$\|\varphi_{\alpha}-\varphi\|_{C^n}\leq C_n|\alpha|.$$

 For $|\alpha|$ small enough, using Taylor's expansion and denoting 
 $$\Delta\varphi_{\alpha}:=\varphi_{\alpha}-\varphi,$$ we have
 \[\begin{split}r_{\alpha}(\varphi):&=(a\cos\varphi_{\alpha}+\alpha_1)^2+(a\sqrt{1-e^2}\sin\varphi_{\alpha}+\alpha_2)^2\\ 
 &=(a\cos\varphi-a\Delta\varphi_{\alpha}\sin\varphi+O(|\alpha|^2)+\alpha_1)^2\\
 &\quad+ (b\sin\varphi+b\Delta\varphi_{\alpha}\cos\varphi+O(|\alpha|^2)+\alpha_2)^2\\
 &=a^2\cos^2\varphi-2a^2\Delta\varphi_{\alpha}\cos\varphi\sin\varphi+2\alpha_1a\cos\varphi+O(|\alpha|^2)\\
 &\quad +b^2\sin^2\varphi+2b^2\Delta\varphi_{\alpha}\cos\varphi\sin\varphi+2b\alpha_2\sin\varphi+O(|\alpha|^2)\\
 &=a^2-a^2e^2\sin^2\varphi+2\alpha_1a\cos\varphi+2\alpha_2b\sin\varphi-a^2e^2\Delta\varphi_{\alpha}\sin2\varphi+O(|\alpha|^2).
 \end{split}\]
 Since $r_{\mu_{\alpha}}=r_{\alpha}$, we obtain
\[\begin{split}\mu_{\alpha}(\varphi)&=\frac{\alpha_1\cos\varphi+\alpha_2\sqrt{1-e^2}\sin\varphi-\frac{1}{2}ae^2\sin2\varphi+O(|\alpha|^2)}{a\sqrt{1-e^2}}\\
&=\frac{\alpha_1}{a}\cos\varphi+\frac{\alpha_2}{a}\sin\varphi+O(e^2|\alpha|).
\end{split}\]
 
 \subsection{Hyperbolic rotations} For any $\beta=(\beta_1,\beta_2)$, let us denote by $\cH[\beta,\cE_{e,c}]$  the ellipse obtained by applying to $\cE_{e,c}$ the  hyperbolic rotation generated by the linear map
  \[\cH[\beta]=\exp\left(\begin{array}{cc}\beta_1 & \beta_2 \\
  \beta_2 & -\beta_1\end{array}\right)=\left(\begin{array}{cc}1+\beta_1 & \beta_2 \\
  \beta_2 & 1-\beta_1\end{array}\right)+O(|\beta|^2).\]
  Let $\mu_{\beta}(\varphi)$ be the function that defines $\cH[\beta,\cE_{e,c}]$ 
  in the elliptic-coordinate frame associated to $\cE_{e,c}$. Then we have
\[\begin{split}&\quad\left(\begin{array}{c}c\; \cosh\big(\mu_0+\mu_{\beta}(\varphi)\big)\cos\varphi \\
c\;\sinh\big(\mu_0+\mu_{\beta}(\varphi)\big)\sin\varphi\end{array}\right)\\
&=\cH[\beta]\left(\begin{array}{c}c\;\cosh(\mu_0)\cos\big(\varphi_{\beta}(\varphi)\big) \\ c\;\sinh(\mu_0)\sin\big(\varphi_{\beta}(\varphi))\end{array}\right)\\
&=\left(\begin{array}{c}(1+\beta_1)a\cos\varphi_{\beta}+\beta_2b\sin\varphi_\beta \\ \beta_2a\cos\varphi_\beta+(1-\beta_1)b\sin\varphi_\beta\end{array}\right)+O(|\beta|^2),
\end{split} \]
where
$\|\varphi_{\beta}-\varphi\|_{C^n}\leq C_n|\beta|.$
  For $|\beta|$ small enough, using Taylor's expansion and denoting
  $\Delta\varphi_{\beta}:=\varphi_{\beta}-\varphi,$
  we have
  \[\begin{split}r_{\beta}(\varphi):&=[(1+\beta_1)a\cos\varphi_\beta+\beta_2b\sin\varphi_\beta]^2\\\ 
  &\quad +[\beta_2a\cos\varphi_\beta+(1-\beta_1)b\sin\varphi_\beta]^2+O(|\beta|^2)\\ 
  &=[a\cos\varphi-a\Delta\varphi_\beta\sin\varphi+\beta_1a\cos\varphi+\beta_2b\sin\varphi+O(|\beta|^2)]^2\\
  &\quad +[\beta_2a\cos\varphi+b\sin\varphi+b\Delta\varphi_\beta\cos\varphi-b\beta_1\sin\varphi+O(|\beta|^2)]^2\\
  &=a^2\cos^2\varphi-a^2\Delta\varphi_{\beta}\sin2\varphi+2\beta_1a^2\cos^2\varphi+\beta_2ab\sin2\varphi\\
  &\quad+b^2\sin^2\varphi+ab\beta_2\sin2\varphi+b^2\Delta\varphi_\beta\sin2\varphi-2b^2\beta_1\sin^2\varphi+O(|\beta|^2)\\
  &=a^2-a^2e^2\sin^2\varphi+2ab\beta_2\sin2\varphi+2a^2\beta_1\cos2\varphi\\
  &\quad +2a^2e^2\sin^2\varphi+a^2e^2\Delta_{\beta}\sin2\varphi+O(|\beta|^2).
  \end{split}\]
  From $r_{\mu_{\beta}}(\varphi)=r_\beta(\varphi)$, we get
  \[\begin{split}\mu_{\beta}&=\frac{2ab\beta_2\sin2\varphi+
  2a^2\beta_1\cos2\varphi+2a^2e^2\sin^2\varphi+
  a^2e^2\Delta_\beta\sin2\varphi+O(|\beta|^2)}{2a^2\sqrt{1-e^2}}\\
  &=\beta_1\cos2\varphi+\beta_2\sin2\varphi+O(e^2|\beta|).
  \end{split}\]
 
 \medskip
  
 To sum up, combining with Lemma \ref{change-coordinate}, we obtain the following result.  \\
 
 \begin{lemma}\label{ellipse}Let $\cE_{e,c}$ be an ellipse of eccentricity $e\in(0,\frac{1}{2})$, and $\Omega$ be a small perturbation of $\cE_{e,c}$, which written  in the elliptic-coordinate frame associated to $\cE_{e,c}$ as
 $$\mu(\varphi)=a_0+a_1\cos\varphi+a_{-1}\sin\varphi+a_2\cos2\varphi+a_{-2}\sin2\varphi.$$
 Assume $\|\mu\|_{C^n}$ small enough {for some $n\geq 2$}, then there exists $C_n$, independent of the eccentricity $e$ and $\mu$, and an ellipse $\bar{\cE}$,
 $$\bar{\cE}= \cE_{e,c}+\bar{\mu}(\varphi),$$
 such that
 $$\|\mu-\bar{\mu}\|_{C^n}\leq C_ne^2\|\mu\|_{C^n}.$$
 \end{lemma}

\medskip 

 \section{Expansion with respect to $e$} \label{e-expansion-sec}
The action-angle parametrization $\theta$ of the elliptic coordinate $\varphi$ corresponding to the caustic $C_\lambda$, expanded up to order $O(e^{2N+2})$, $J\in\mathbb{N}$ is as follows:
\begin{equation}\label{expansion1}\varphi(\theta,\lambda,e)=\theta+\sum_{j=1}^N\varphi_j(\theta)\frac{a^je^{2j}}{(a^2-\lambda^2)^j}+O(e^{2N+2}),\end{equation}
where the functions $\varphi_N(\theta)$ are of the form $$\varphi_j(\theta)=\sum_{l=1}^j\beta_{j,l}\sin (2l\theta).$$
We give below the explicit formulae for $\varphi_N(\theta)$ for $j=1,\dots,6$.  $$\varphi_1(\theta)=\frac{1}{8}\sin 2\theta,$$
$$\varphi_2(\theta)=\frac{1}{256}(16\sin 2\theta+\sin4\theta),$$
$$\varphi_3(\theta)=\frac{83\sin2\theta}{2048}+\frac{\sin4\theta}{256}+\frac{\sin6\theta}{6144},$$
$$\varphi_4(\theta)=\frac{121\sin2\theta}{4096}+\frac{29\sin4\theta}{8192}+\frac{\sin6\theta}{4096}+\frac{\sin8\theta}{131072},$$
$$\varphi_5(\theta)=\frac{12071\sin2\theta}{524288}+\frac{13\sin4\theta}{4096}+\frac{37\sin6\theta}{131072}+\frac{\sin8\theta}{65536}+\frac{\sin10\theta}{2621440},$$
$$\varphi_6(\theta)=\frac{19651\sin2\theta}{1048576}+\frac{47955\sin4\theta}{16777216}+\frac{235\sin6\theta}{786432}+\frac{45\sin8\theta}{2097152}+\frac{\sin10\theta}{1048576}+\frac{\sin12\theta}{50331648}.$$

\

{
\blm \label{lm:mu-expansion}
Let $$\mu(\varphi)=a_0+\sum_{k=1}^{+\infty}a_k\cos (k\varphi) +b_k\sin (k\varphi),$$
and $\mu(\varphi)\in C^m(\mathbb{T}).$ Then for $N\leq m-1$,  the expansion of the function $\mu(\varphi(\theta,\lambda,e))$ with respect to $e$ up to order $O(e^{2N+2})$ is
\begin{equation}\label{expansion2}\mu(\varphi(\theta,\lambda,e))=\mu(\theta)+
\sum_{j=1}^N P_j(\theta)\frac{a^je^{2j}}{(a^2-\lambda^2)^j}+O(e^{2N+2}\|\mu\|_{C^{N+1}}),\end{equation}
where the functions $P_j(\theta)$ are of the form
$$P_j(\theta)=\sum_{k=1}^{+\infty}\sum_{l=-j}^{j}\xi_{j,l}(k)\Big(a_k\cos ((k+2l)\theta) +b_k\sin ((k+2l)\theta) \Big).$$
The coefficients $\xi_{j,l}(k)$ can be explicitely computed and, for small $j$ and $l$,  they are presented below. 
\elm }

\

The functions $P_j(\theta)$, $j=1,\dots,6$ are explicitly given by
  \[\begin{split}P_1(\theta)&=\mu'(\theta)\varphi_1(\theta)
=\sum_{k=1}^{+\infty}\sum_{l=-1}^{1}\xi_{1,l}(k)\Big(a_k\cos ((k+2l)\theta) +b_k\sin ((k+2l)\theta) \Big),
\end{split}
\]
where
$$\xi_{1,-1}(k)=-\frac{k}{16},\quad\xi_{1,0}(k)=0,\quad\text{and}\quad \xi_{1,1}(k)=\frac{k}{16}.$$
\[\begin{split}P_2(\theta)&=\mu'(\theta)\varphi_2(\theta)+\frac{1}{2}\mu''(\theta)(\varphi_1(\theta))^2\\
&=\sum_{k=1}^{+\infty}\sum_{l=-2}^{2}\xi_{2,l}(k)\Big(a_k\cos ((k+2l)\theta) +b_k\sin ((k+2l)\theta) \Big),
\end{split}\]
where 
$$\xi_{2,-2}(k)=\frac{k^2-k}{512},\;\;\xi_{2,-1}(k)=-\frac{16k}{512},\;\;\xi_{2,0}=-\frac{2k^2}{512},$$
$$\xi_{2,1}(k)=\frac{16k}{512},\quad\xi_{2,2}(k)=\frac{k^2+k}{512}.$$
\[\begin{split}P_3(\theta)=&\mu'(\theta)\varphi_3(\theta)+\frac 22\mu''(\theta)\varphi_1(\theta)\varphi_2(\theta)+\frac{1}{6}\mu'''(\theta)(\varphi_1)^3\\
&=\sum_{k=1}^{+\infty}\sum_{l=-3}^3\xi_{3,l}(k)\Big(a_k\cos ((k+2l)\theta) +b_k\sin ((k+2l)\theta) \Big),\end{split}\]
where 
$$\xi_{3,-3}(k)=-\frac{k}{12288}+\frac{k^2}{8192}-\frac{k^3}{24576},$$
$$\xi_{3,-2}(k)=-\frac{k}{512}+\frac{k^2}{512},\quad\xi_{3,-1}(k)=-\frac{83k}{4096}-\frac{k^2}{8192}+\frac{k^3}{8192},$$
$$\xi_{3,0}(k)=-\frac{k^2}{256},\quad\xi_{3,1}(k)=\frac{83k}{4096}-\frac{k^2}{8192}-\frac{k^3}{8192},\quad\xi_{3,2}(k)=\frac{k+k^2}{512},$$
$$\xi_{3,3}(k)=\frac{k}{12288}+\frac{k^2}{8192}+\frac{k^3}{24576}.$$
 \[\begin{split}P_4(\theta)&=\mu'(\theta)\varphi_4(\theta)+\frac{1}{2}\mu''(\theta)[(\varphi_2(\theta))^2+2\varphi_1(\theta)\varphi_3(\theta)]\\
&\quad+\frac{1}{6}\mu'''(\theta)3(\varphi_1(\theta))^2\varphi_2(\theta)+\frac{1}{24}\mu^{(4)}(\theta)(\varphi_1(\theta))^4\\
&=\sum_{k=1}^{+\infty}\sum_{l=-4}^4\xi_{4,l}(k)\Big(a_k\cos ((k+2l)\theta) +b_k\sin ((k+2l)\theta) \Big),\end{split}\]
where $\xi_{4,j}(k)$, $j=-4,\dots, 4$ are polynomials in  $k$ of at most degree 4, and 
 $$\xi_{4,4}(k)=\frac{k}{262144}+\frac{11k^2}{1572864}+\frac{k^3}{262144}+\frac{k^4}{1572864}.$$
 \[\begin{split}&P_5(\theta)=\mu'(\theta)\varphi_5(\theta)+\frac{2}{2}\mu''(\theta)[\varphi_2(\theta)\varphi_2(\theta)+\varphi_1(\theta)(\varphi_2(\theta))^2]\\
&\quad+\frac{3}{6}\mu'''(\theta)[\varphi_1(\theta)(\varphi_2(\theta))^2+(\varphi_1(\theta))^2\varphi_3(\theta)]\\
&\quad+\frac{4}{24}\mu^{(4)}(\varphi_1(\theta))^3\varphi_2(\theta)+
\frac{1}{120}\mu^{(5)}(\theta)(\varphi_1(\theta))^5\\
&=\sum_{k=1}^{+\infty}\sum_{l=-5}^{5}\xi_{5,l}(k)\Big(a_k\cos ((k+2l)\theta) +b_k\sin ((k+2l)\theta) \Big),\end{split}\]
where $\xi_{5,j}$, $j=-5,\dots, 5$ are polynomials in $k$ of at most degree 5, and  
$$\xi_{5,5}(k)=\frac{k}{5242880}+\frac{5k^2+k^4}{12582912}+\frac{7k^3}{25165824}+\frac{k^5}{125829120}.$$
 \[\begin{split}P_6(\theta)&=\mu'(\theta)\varphi_6(\theta)+\frac{1}{2}\mu''(\theta)[2\varphi_1(\theta)\varphi_5(\theta)+2\varphi_2(\theta)\varphi_4(\theta)+(\varphi_3(\theta))^2]\\
 &\quad+\frac{1}{6}\mu'''(\theta)[3(\varphi_1(\theta))^2\varphi_4(\theta)+
 3\varphi_1(\theta)\varphi_2(\theta)\varphi_3(\theta)+(\varphi_2(\theta))^3]\\
 &\quad +\frac{1}{24}\mu^{(4)}(\theta)[4(\varphi_1(\theta))^3\varphi_3(\theta)+
 6(\varphi_1(\theta))^2(\varphi_2(\theta))^2]\\
 &\quad+\frac{1}{120}\mu^{(5)}(\theta)[5(\varphi_1(\theta))^4\varphi_2(\theta)]+\frac{1}{720}\mu^{(6)}(\theta)[(\varphi_1(\theta))^6]\\
 &=\sum_{k=1}^{+\infty}\sum_{l=-6}^{6}\xi_{6,l}(k)\Big(a_k\cos ((k+2l)\theta) +b_k\sin ((k+2l)\theta) \Big),\end{split}\]
 where $\xi_{6,j}(k)$ are polynomials in $k$ of at most order 6, and 
 $$\xi_{6,6}(k)=\frac{k}{100663296}+\frac{137k^2}{6039797760}+\frac{11k^3+k^5}{805306368}+\frac{17k^4}{2415919104}+\frac{k^6}{12079595520}.$$

\medskip 

  \section{The inverse and adjugate of a matrix}\label{appendix-matrix}
We  recall the definition of the adjugate of a matrix and its relation to the inverse of a square matrix in this section.
 
Let $A$ be a $n\times n$ matrix with real entries. The adjugate $\text{adj}(A)$  of~$A$ is the transpose of 
the cofactor matrix $C$ of $A$, 
$$
\text{adj} (A)=C^T.
$$
The cofactor matrix  of $A$ is the $n\times n$ matrix  $C$ whose $(i,j)$-entry is 
the $(i,j)$-cofactor of $A$,
$$
C_{ij}=(-1)^{i+j}M_{ij},
$$
where $M_{i,j}$ is the determinant of the $(n-1)\times(n-1)$ matrix that results from deleting 
the $i$-th row and the $j$-th column  of $A$. Therefore, the adjugate of matrix $A$ is the $n\times n$ 
matrix $\text{adj(A)}$ whose $(i,j)$-entry is the $(j,i)$-cofactor of $A$,
$$
\text{adj}(A)_{ij}=C_{ji}=(-1)^{j+i}M_{ji}.
$$ 
\begin{theorem} \label{thm:inv-matrix} For a square matrix $A=(a_{ij})$,
\begin{enumerate}
\item  $\text{det}(A)=\sum_{i=1}^na_{ij}C_{ij}$, for $j=1,\dots,n$.

\item $A$ is invertible if and only if $\text{det}(A)\neq0$. Moreover, 
the inverse has the form 
$$
A^{-1}=\frac{1}{{\rm{det}}(A)}\ {\rm{adj}}(A).
$$
\end{enumerate}
 \end{theorem}
Now consider the coefficient matrix in \eqref{big1-section6}, which has the form
$$A=(a_{ij})= 
\left(\begin{array}{cccccc}A_{11}e^4&A_{12}e^2 & 1 & 0 & 0&0 
\\ A_{21}e^6&A_{22}e^4 & A_{23}e^2 & 1 & 0&0 
\\ A_{31}e^8&A_{32}e^6 &A_{33}e^4& A_{34}e^2  & 1 &0
\\ A_{41}e^8&A_{42}e^6 & A_{43}e^4 & A_{44}e^2 & 1&0
\\ A_{51}e^{10}&A_{52}e^{8}&A_{53}e^6&A_{54}e^4&A_{55}e^2&1
\\ A_{61}e^{10}&A_{62}e^{8}&A_{63}e^6&A_{64}e^4&A_{65}e^2&1
\end{array}\right)$$

Direct calculation shows that 
$$\text{det}(A)=\sum_{\sigma\in S_6}\text{sgn}(\sigma)\prod_{i=1}^6a_{i\sigma_i}=\sum (\cdots)e^{16}=\mathcal{A} e^{16},$$
where $S_6$ is the group of  all   permutations of $\{1,\dots,6\}$. 
One key feature here is that the nonzero quantities in the summation are all exactly of order $e^{16}$. 
Using part (1) of Theorem \ref{thm:inv-matrix}, we obtain that 
$$C_{11}=c_1e^{12},\; C_{12}=c_2e^{10},\; C_{13}=c_3e^{8},\; C_{14}=c_4e^8,\; C_{15}=c_5e^6,\; C_{16}=c_6e^{6}.$$
Then, using part (2) of Theorem \ref{thm:inv-matrix}, if $\det {A}\neq0$, then the first row of the inverse $A^{-1}$ has the form
$$(O(e^{-4}), O(e^{-6}), O(e^{-8}), O(e^{-8}), O(e^{-10}), O(e^{-10})).$$
In the same way, we obtain that the second row of $A^{-1}$ is of the form
$$(O(e^{-2}), O(e^{-4}), O(e^{-6}), O(e^{-6}), O(e^{-8}), O(e^{-8})).$$
All of the  matrices appearing in Sections \ref{case3}--\ref{case-general} could be treated  similarly.

\end{document}